\documentclass[final]{siamart220329}

\usepackage{amsmath,amssymb,mathtools,amsfonts}
\usepackage{bm}
\usepackage{booktabs}
\usepackage{array}
\usepackage{enumitem}
\usepackage{microtype}
\usepackage{url}
\usepackage{tikz}
\usetikzlibrary{arrows.meta,positioning,calc}

\newcommand{\EPO}{EPO}
\newcommand{\COSDG}{\textsf{COS(DG)}}

\newsiamthm{assumption}{Assumption}
\newsiamremark{remark}{Remark}
\newsiamremark{example}{Example}
\newsiamremark{notation}{Notation}
\newsiamremark{procedure}{Procedure}

\numberwithin{equation}{section}

\headers{\EPO{}: Entropy Stability, Positivity, and Oscillation Suppression}{K. Wu}

\title{\EPO{}: A Unified Framework for Entropy Stability, Positivity, and Oscillation Suppression}

\author{Kailiang Wu\thanks{Department of Mathematics, Southern University of Science and Technology, Shenzhen, Guangdong 518055, P.R.~China (\email{wukl@sustech.edu.cn}). }}

\begin{document}

\maketitle

\begin{abstract}
High-order finite volume and discontinuous Galerkin methods are often stabilized by separate nonlinear devices for admissibility, entropy control, and oscillation suppression. This separation hides a simple geometric fact: all three act on the same cellwise candidate state. We propose a general framework (termed EPO) unifying fully discrete entropy stability, positivity/bound preservation, and spurious oscillation elimination. Starting from a candidate update, we scale along the ray anchored at its updated cell average. The admissible-state constraint, the entropy constraint, and the oscillation-suppressing constraint each define an admissibility radius on that ray, and the applied limiter is their minimum.

The decisive analytical ingredient is a {\em weak entropy stability} at the level of the updated cell average. A two-point Lax--Friedrichs/Riemann-average entropy inequality yields local cell-average entropy budgets, and the same radial scaling mechanism behind Zhang--Shu positivity preservation lifts these weak budgets to strong quadrature-based entropy inequalities. The framework is therefore not a summation-by-parts, split-form, or flux-differencing construction: EPO acts on a candidate finite volume or discontinuous Galerkin update and converts weak average information into fully discrete nodal entropy stability. {\em The construction also works for any prescribed finite family of convex entropy pairs. Each pair yields its own entropy radius, and taking the minimum enforces fully discrete entropy stability for all of them simultaneously.} We prove the preservation of cell averages, invariant-set preservation, local and global strong entropy inequalities, stagewise budgets for strong-stability-preserving (SSP) Runge--Kutta methods, an SSP multistep variant that retains the designed high-order temporal accuracy, and extensions on rectangular and unstructured triangular meshes.
\end{abstract}



\tableofcontents

\section{Introduction}

High-order finite volume (FV) and discontinuous Galerkin (DG) methods for hyperbolic conservation laws usually need three kinds of nonlinear stability properties: invariant-domain or positivity preservation, entropy stability, and suppression of spurious oscillations near shocks. In many implementations these are introduced by separate devices, even though all three act on the same cellwise candidate state. The purpose of this paper is to formulate a single local framework for them.

The starting point is the weak-to-strong mechanism, first proposed by Zhang and Shu for positivity preservation  \cite{ZhangShuMPP2010,ZhangShuPP2010,ZhangShuSurvey2011}. One first proves a weak statement for the updated cell average and then restores a nodal or pointwise statement by scaling the candidate state toward that average. This idea is now standard in positivity and invariant-domain preservation, and recent work has clarified the geometric content of nonlinear constraints, optimal cell average decomposition, and related frameworks \cite{WuShuGQL2023,CuiDingWuOCAD2024,CuiKurganovWuBPFramework2024,WuZhangShuSurvey2025}. The present paper uses the same cell-average-anchored scaling to treat entropy stability.

For entropy stability, however,  our approach proposed  here is different from the classical one. Tadmor's flux-differencing theory and its high-order descendants, including summation-by-parts, split-form, and suitable-quadrature constructions, build entropy conservation or entropy dissipation directly into the spatial discretization \cite{Tadmor1987,Tadmor2003,FisherCarpenter2013,Gassner2013,CarpenterFisherNielsenFrankel2014,SvardNordstrom2014,DelReyFernandezHickenZingg2014,ChenShu2017,ChanFernandezCarpenter2019}. The framework developed below is not another summation-by-parts (SBP) construction. Instead, it starts from a candidate FV/DG update and asks what weak cell-average information is sufficient to recover a strong quadrature-based entropy inequality after a local correction. The main answer is a \emph{weak entropy stability} for the updated cell average, which can then be lifted by the same radial scaling used in positivity preservation.

This distinction also matters at the level of entropy pairs. Most classical entropy-stable constructions are designed for one prescribed convex entropy pair. In contrast, the EPO geometry is local and constraint based. Any prescribed finite family of convex entropy pairs may be enforced simultaneously: each pair yields an entropy budget and an entropy radius, and the final limiter uses the minimum of those radii.

The local construction is simple. Let $\bar{\mathbf U}_j^\star$ be the candidate updated cell average in cell $I_j$ and let $\mathbf U_j^\star$ denote the candidate DG evolved polynomial, FV reconstruction polynomial, or equivalent nodal array. The admissible modifications considered in this paper lie on the cell-average-anchored scaling ray
\[
\mathcal S_{\bar{\mathbf U}_j^\star}(\theta;\mathbf U_j^\star)
=
\bar{\mathbf U}_j^\star
+
\theta\bigl(\mathbf U_j^\star-\bar{\mathbf U}_j^\star\bigr),
\qquad
0\le \theta\le 1.
\]
Along this ray, the geometric module contributes a radius $\theta_j^{\mathrm P}$. For one entropy pair, the entropy module contributes the positivity-first radius $\theta_j^{\mathrm{PE}}$; when the full candidate ray is already contained in the entropy domain, this reduces to $\theta_j^{\mathrm{PE}}=\min\{\theta_j^{\mathrm P},\theta_j^{\mathrm E}\}$. The oscillation module contributes a radius $\theta_j^{\mathrm O}$. The final limiter uses
\[
\theta_j^{\mathrm{EPO}}
=
\min\bigl\{\theta_j^{\mathrm{PE}},\theta_j^{\mathrm O}\bigr\}.
\]
For several prescribed entropy pairs $\{(\eta^{(r)},\mathcal Q^{(r)})\}_{r=1}^M$, one computes one entropy radius for each pair and then replaces $\theta_j^{\mathrm{PE}}$ by
\[
\theta_j^{\mathrm{PE,all}}:=\min_{1\le r\le M}\theta_j^{\mathrm{PE},(r)}.
\]
The geometry is therefore unchanged: there is still one scaling ray and one final minimum.

In this paper, we propose a novel framework, termed EPO, which unifies fully discrete entropy stability, positivity/bound preservation, and spurious oscillation suppression. 
Our main results can be summarized as follows.
\begin{enumerate}[label=(\roman*),leftmargin=2em]
\item A two-point Lax--Friedrichs/Riemann-average entropy inequality yields a local weak entropy stability for candidate cell averages.
\item The convex scaling limiter similar to that used in Zhang--Shu positivity preservation lifts this weak statement to a strong quadrature-based entropy inequality.
\item The local EPO admissible set is closed and convex, so the final limiter is the maximal radial projection into that set.
\item For explicit strong-stability-preserving (SSP)  Runge--Kutta methods the correct entropy input can be stagewise. For SSP multistep methods we propose a single end-of-step entropy budget and therefore a one-shot end-of-step entropy correction while retaining the designed temporal order.
\item The same local geometry extends to rectangular and unstructured triangular meshes.
\end{enumerate}

\begin{table}[t]
\caption{The three local constraints combined by EPO.}
\label{tab:epo-modules}
\centering
\footnotesize
\renewcommand{\arraystretch}{1.15}
\begin{tabular}{>{\raggedright\arraybackslash}p{0.15\textwidth} >{\raggedright\arraybackslash}p{0.28\textwidth} >{\raggedright\arraybackslash}p{0.14\textwidth} >{\raggedright\arraybackslash}p{0.21\textwidth}}
\toprule
Module & Weak input / candidate information & Radius on the ray & Strong output after limiting \\
\midrule
Geometry &
Updated cell average lies in the admissible convex set $G$ &
$\theta_j^{\mathrm P}$ &
All limited nodal states lie in $G$ \\

Entropy &
For each prescribed entropy pair, the updated cell average satisfies a local budget $\eta^{(r)}(\bar{\mathbf U}_j^\star)\le B_j^{n,(r)}$ &
$\theta_j^{\mathrm{PE},(r)}$ &
For each prescribed pair, $\mathcal E^{(r)}(\mathbf U_j^{n+1})\le B_j^{n,(r)}$ \\

Oscillation &
A COS-compatible closed convex oscillation set $\mathcal O_j^\star$ is available &
$\theta_j^{\mathrm O}$ &
The limited state belongs to $\mathcal O_j^\star$ \\
\bottomrule
\end{tabular}
\end{table}

The framework is local. A concrete method must still provide candidate nodal values, candidate cell averages, and weak budgets appropriate to its quadrature, numerical flux, and mesh. Once those inputs are available, the conclusions of the EPO theorem follow from convexity and scaling.

The remainder of the paper is organized as follows. Section~\ref{sec:preliminaries} introduces the notation, quadrature representation, admissible set, entropy pair, and cell-average-anchored scaling rays. Section~\ref{sec:weak} develops the weak cell-average budgets, including the two-point entropy inequality and an abstract weak entropy theorem. Section~\ref{sec:oscillation} presents the oscillation module and general convex-oscillation-suppressing (COS)-compatible convex sets. Section~\ref{sec:entropy_limiter} formulates the entropy limiter and the local EPO admissible set. Section~\ref{sec:main} contains the one-dimensional EPO theorem, and Section~\ref{sec:time} discusses the stagewise SSP Runge--Kutta and SSP multistep realizations. Section~\ref{sec:2d} extends the framework to two-dimensional rectangular and triangular meshes. Section~\ref{sec:realizations} records representative realizations and implementation formulas, Section~\ref{sec:num} presents several numerical examples, and Appendix~\ref{app:explicit} collects auxiliary explicit expressions.

\section{Preliminaries}\label{sec:preliminaries}

We consider the one-dimensional hyperbolic conservation law system
\begin{equation}\label{eq:pde}
\partial_t \mathbf U + \partial_x \mathbf F(\mathbf U)=0,
\qquad
x\in \mathbb R,\quad t>0,
\end{equation}
where $\mathbf U(x,t)\in \mathbb R^m$ is the conservative variable and $\mathbf F:\mathbb R^m\to\mathbb R^m$ is the flux. We work at the level of a generic high-order spatial discretization on a uniform mesh
\[
I_j = [x_{j-\frac12},x_{j+\frac12}],
\qquad
\Delta x = x_{j+\frac12}-x_{j-\frac12}.
\]
The numerical solution in cell $I_j$ is represented either by a polynomial, a reconstruction, or any equivalent finite-dimensional object determined by a collection of nodal values. Since the EPO framework is local and algebraic, the precise representation space is less important than the following positive quadrature structure.

\subsection{Cell averages, quadrature nodes, and nodal arrays}

Fix a positive quadrature rule on the reference cell, represented in each physical cell by nodes
\[
x_j^{(1)},x_j^{(2)},\dots,x_j^{(L)}\in I_j
\]
and weights
\[
\omega_1,\omega_2,\dots,\omega_L >0,
\qquad
\sum_{\nu=1}^L \omega_\nu = 1.
\]
We assume that the discrete representation under consideration satisfies the exact cell-average identity
\begin{equation}\label{eq:cellavg-quad}
\bar{\mathbf U}_j
=
\sum_{\nu=1}^L \omega_\nu \mathbf U_{j,\nu},
\end{equation}
where $\mathbf U_{j,\nu}$ denotes the nodal value at $x_j^{(\nu)}$ and
\[
\bar{\mathbf U}_j
=
\frac{1}{\Delta x}\int_{I_j}\mathbf U_h(x)\,dx
\]
is the cell average. In a nodal DG method on Gauss--Lobatto points, \eqref{eq:cellavg-quad} is immediate; in FV or modal DG, we also use the Gauss--Lobatto points and take $L = \lceil \frac{k+3}2 \rceil$ for solution polynomials of degree  $k$. In FV/DG positivity analyses, the same structure appears through classical or optimal cell average decomposition (CAD/OCAD) \cite{ZhangShuMPP2010,CuiDingWuOCAD2024}.


\subsection{Admissible state set}

The EPO framework separates the \emph{candidate high-order state} from the \emph{admissibility mechanism}. The admissible set is assumed to be a convex subset of the state space.

\begin{assumption}[Closed convex admissible set]\label{ass:G}
There is a nonempty closed convex set $G\subset \mathbb R^m$ such that all admissible numerical states must lie in $G$.
\end{assumption}

\begin{remark}
For many physical systems the genuine admissible set is open rather than closed. For instance, the Euler equations require $\rho>0$ and $p>0$. In practical positivity-preserving analysis one often works instead with a closed numerical admissible set
$
G_\varepsilon \subset G
$ 
for some small $\varepsilon>0$, for example
$ 
G_\varepsilon = \{\rho\ge \varepsilon,\ p\ge \varepsilon\}.
$ 
All the arguments in this paper apply directly to such closed sets. Since the local scaling is cell-average preserving and continuous, one can let $\varepsilon$ depend on roundoff or on the resolution in applications. We keep the notation $G$ throughout the paper for simplicity.
\end{remark}

\subsection{Entropy pair and local discrete entropy}

Let $(\eta,\mathcal Q)$ be a convex entropy pair for \eqref{eq:pde}. Convexity of $\eta$ is used throughout the analysis.

\begin{assumption}[Entropy domain and entropy pair]\label{ass:entropy}
There exists an open convex set $D_\eta\subset \mathbb R^m$ with $G\subset D_\eta$ such that the entropy function $\eta:D_\eta\to \mathbb R$ is twice continuously differentiable and convex on $D_\eta$. The associated entropy flux is denoted by $\mathcal Q$: in one space dimension $\mathcal Q:D_\eta\to \mathbb R$, while in several space dimensions we write $\mathcal Q=(\mathcal Q_1,\dots,\mathcal Q_d)$ with $\mathcal Q:D_\eta\to \mathbb R^d$.  
\end{assumption}

For ${\bf U}_h$, we define the local quadrature-based entropy
\begin{equation}\label{eq:local-entropy}
\mathcal E(\mathbf U_j)
:=
\sum_{\nu=1}^L \omega_\nu \eta(\mathbf U_{j,\nu}).
\end{equation}
When cell indices are needed we also write $\mathcal E_j(\mathbf U_j)$, but the definition is independent of $j$ when the quadrature is fixed.

\begin{remark}
The quantity $\mathcal E(\mathbf V)$ is the natural entropy functional for a nodal or CAD-based limiter, because the limiter acts directly on nodal values. In general,
\[
\mathcal E(\mathbf U_j)
\neq
\eta\!\left(\sum_{\nu=1}^L \omega_\nu \mathbf U_{j,\nu} \right)
\]
unless $\eta$ is affine. By Jensen's inequality,
\begin{equation}\label{eq:jensen-local}
\eta\!\left(\sum_{\nu=1}^L \omega_\nu \mathbf U_{j,\nu}  \right)
\le
\mathcal E(  \mathbf U_{j}   ).
\end{equation}
Accordingly, the term \emph{strong entropy stability} in this paper always means \emph{quadrature-based} or \emph{nodal} strong entropy stability. The positivity-first entropy pathway introduced below is designed precisely so that every pointwise entropy evaluation takes place on $G\subset D_\eta$, where the entropy is guaranteed to be well defined.
\end{remark}

\begin{remark}[Several entropy pairs]\label{rem:multi-entropy-prelim}
Let $\{(\eta^{(r)},\mathcal Q^{(r)})\}_{r=1}^M$ be a finite family of convex entropy pairs with $G\subset \bigcap_{r=1}^M D_{\eta^{(r)}}$. For each pair one defines
\[
\mathcal E^{(r)}(\mathbf U_j):=\sum_{\nu=1}^L \omega_\nu \eta^{(r)}( \mathbf U_{j,\nu} ).
\]
All convexity arguments below apply to each pair separately. Simultaneous enforcement of finitely many entropy pairs amounts to intersecting finitely many entropy sublevel sets, which therefore remains a closed convex constraint on the scaling ray.
\end{remark}

\subsection{Scaling rays}

The cell average will act as the anchor point for all local limiting. Given a cell average $\bar{\mathbf U}\in \mathbb R^m$ and a polynomial solution $\mathbf U_j$ (or its nodal values in $I_j$) with
\[
\sum_{\nu=1}^L \omega_\nu {\bf U}_{j,\nu} = \bar{\mathbf U},
\]
we define the scaling ray
\begin{equation}\label{eq:scaling-ray}
\mathcal S_{\bar{\mathbf U}}(\theta;\mathbf U)
:=
\bar{\mathbf U} + \theta\bigl(\mathbf U-\bar{\mathbf U}\bigr),
\qquad
0\le \theta\le 1.
\end{equation}
The scaling ray preserves the cell average:
\begin{equation}\label{eq:mean-preserving-ray}
\sum_{\nu=1}^L \omega_\nu \mathcal S_{\bar{\mathbf U}}(\theta;\mathbf U)_\nu
=
\bar{\mathbf U}
\quad
\text{for all }\theta\in[0,1].
\end{equation}
This identity is the basic reason that convex scaling is compatible with conservative discretizations.

\subsection{Candidate stage values}

Throughout the paper, a superscript $\star$ denotes the \emph{candidate} state produced by the underlying high-order method before EPO limiting. Thus, in cell $I_j$ we have candidate nodal values
\[
\mathbf U_{j,\nu}^\star,\qquad \nu=1,\dots,L,
\]
and candidate cell average
\[
\bar{\mathbf U}_j^\star
=
\sum_{\nu=1}^L \omega_\nu \mathbf U_{j,\nu}^\star.
\]
The EPO limiter will only modify the nodal deviations from $\bar{\mathbf U}_j^\star$; therefore, it will preserve the conservative cell-average update produced by the underlying method.

\begin{remark}[DG stage polynomials and FV reconstructions]
In a DG method the candidate object is the cell polynomial produced by the stage update. In a FV method the same notation denotes any conservative reconstruction, subcell profile, or reconstructed polynomial whose weighted mean equals the updated cell average. Every EPO module in this paper depends only on this cellwise candidate function or reconstruction, its weighted nodal evaluations, and compact neighboring-cell data. Consequently the framework is representation-agnostic and applies to both FV and DG.
\end{remark}

\section{Weak cell-average budgets and weak entropy stability}\label{sec:weak}

The Zhang--Shu philosophy separates the analysis into two stages. First one proves a \emph{weak} property at the level of the cell average. Then one lifts this weak property to a \emph{strong} nodal or pointwise property by local scaling. In EPO, the same logic is applied simultaneously to invariant-set preservation and entropy stability.

\subsection{Weak positivity / weak invariant-set preservation}

We begin with the standard weak geometric property.

\begin{assumption}[Weak positivity / weak invariant-set preservation]\label{ass:weakP}
For every cell $I_j$ and each candidate time step or stagewise update under consideration, the updated cell average satisfies
\begin{equation}\label{eq:weakP}
\bar{\mathbf U}_j^\star \in G.
\end{equation}
\end{assumption}

\begin{remark}
Assumption~\ref{ass:weakP} is precisely the \emph{weak positivity} or \emph{weak bound-preserving} property in the Zhang--Shu framework. In concrete schemes this property may be established by expressing $\bar{\mathbf U}_j^\star$ as a convex combination of admissible first-order building blocks and internal nodal values, typically using CAD or OCAD \cite{ZhangShuMPP2010,ZhangShuPP2010,CuiDingWuOCAD2024}. The EPO limiter is agnostic about how \eqref{eq:weakP} is obtained.
\end{remark}

\subsection{A two-point Lax--Friedrichs/Riemann-average entropy inequality}

We now prove the fundamental entropy counterpart of the first-order weak geometric property. Let $\mathbf U_L,\mathbf U_R\in G$ and let $\alpha>0$. Define
\begin{equation}\label{eq:Halpha}
H_\alpha(\mathbf U_L,\mathbf U_R)
:=
\frac{\mathbf U_L+\mathbf U_R}{2}
-
\frac{\mathbf F(\mathbf U_R)-\mathbf F(\mathbf U_L)}{2\alpha}.
\end{equation}
This is the Lax--Friedrichs (or local Lax--Friedrichs / Rusanov-type) two-point average with parameter $\alpha$.

\begin{theorem}[Two-point entropy inequality]\label{thm:2point-entropy}
Let $\mathbf U_L,\mathbf U_R\in G$. Assume that the Riemann problem for \eqref{eq:pde} with left state $\mathbf U_L$ and right state $\mathbf U_R$ admits an entropy solution of the self-similar form
\[
\mathbf U(x,t)=\mathcal R(\xi),\qquad \xi=\frac{x}{t},
\]
whose values lie in $G$, and that all waves are contained in the cone $|\xi|\le \alpha$. Then
\begin{equation}\label{eq:2point-entropy}
\eta\!\left(H_\alpha(\mathbf U_L,\mathbf U_R)\right)
\le
\frac{\eta(\mathbf U_L)+\eta(\mathbf U_R)}{2}
-
\frac{\mathcal Q(\mathbf U_R)-\mathcal Q(\mathbf U_L)}{2\alpha}.
\end{equation}
Moreover,
\begin{equation}\label{eq:Halpha-average}
H_\alpha(\mathbf U_L,\mathbf U_R)
=
\frac{1}{2\alpha}\int_{-\alpha}^{\alpha}\mathcal R(\xi)\,d\xi,
\end{equation}
and therefore $H_\alpha(\mathbf U_L,\mathbf U_R)\in G$ by convexity of $G$.
\end{theorem}

\begin{proof}
Fix $t>0$ and consider the Riemann entropy solution
\[
\mathbf U(x,t)=\mathcal R\!\left(\frac{x}{t}\right).
\]
Since all waves are contained in $|\xi|\le \alpha$, we have
\[
\mathbf U(-\alpha t+,t)=\mathbf U_L,
\qquad
\mathbf U(\alpha t-,t)=\mathbf U_R.
\]
Integrating \eqref{eq:pde} over the moving interval $[-\alpha t,\alpha t]$ and using the Reynolds transport formula gives
\begin{align*}
\frac{d}{dt}\int_{-\alpha t}^{\alpha t}\mathbf U(x,t)\,dx
&=
\int_{-\alpha t}^{\alpha t}\partial_t\mathbf U(x,t)\,dx
+\alpha \mathbf U(\alpha t-,t)
+\alpha \mathbf U(-\alpha t+,t) \\
&=
-\mathbf F(\mathbf U(\alpha t-,t))
+\mathbf F(\mathbf U(-\alpha t+,t))
+\alpha \mathbf U_R + \alpha \mathbf U_L \\
&=
\alpha(\mathbf U_L+\mathbf U_R)
+\mathbf F(\mathbf U_L)-\mathbf F(\mathbf U_R).
\end{align*}
The right-hand side is constant in time, and the interval collapses to a point at $t=0$, so after integration in time we obtain
\[
\int_{-\alpha t}^{\alpha t}\mathbf U(x,t)\,dx
=
t\Bigl[\alpha(\mathbf U_L+\mathbf U_R)+\mathbf F(\mathbf U_L)-\mathbf F(\mathbf U_R)\Bigr].
\]
By self-similarity,
\[
\int_{-\alpha t}^{\alpha t}\mathbf U(x,t)\,dx
=
t\int_{-\alpha}^{\alpha}\mathcal R(\xi)\,d\xi.
\]
Dividing by $2\alpha t$ yields \eqref{eq:Halpha-average}.

Next, apply the same argument to the entropy inequality
\[
\partial_t \eta(\mathbf U)+\partial_x \mathcal Q(\mathbf U)\le 0
\]
in the sense of distributions. Integrating over $[-\alpha t,\alpha t]$ gives
\begin{align*}
\frac{d}{dt}\int_{-\alpha t}^{\alpha t}\eta(\mathbf U(x,t))\,dx
&\le
-\mathcal Q(\mathbf U(\alpha t-,t))
+\mathcal Q(\mathbf U(-\alpha t+,t)) \\
&\quad
+\alpha\eta(\mathbf U(\alpha t-,t))
+\alpha\eta(\mathbf U(-\alpha t+,t)) \\
&=
\alpha\bigl(\eta(\mathbf U_L)+\eta(\mathbf U_R)\bigr) \\
&\quad
+\mathcal Q(\mathbf U_L)-\mathcal Q(\mathbf U_R).
\end{align*}
Integrating in time and using self-similarity once more,
\[
t\int_{-\alpha}^{\alpha}\eta(\mathcal R(\xi))\,d\xi
\le
t\Bigl[
\alpha\bigl(\eta(\mathbf U_L)+\eta(\mathbf U_R)\bigr)
+\mathcal Q(\mathbf U_L)-\mathcal Q(\mathbf U_R)
\Bigr].
\]
After division by $2\alpha t$ we find
\begin{equation}\label{eq:avg-entropy-bound}
\frac{1}{2\alpha}\int_{-\alpha}^{\alpha}\eta(\mathcal R(\xi))\,d\xi
\le
\frac{\eta(\mathbf U_L)+\eta(\mathbf U_R)}{2}
-
\frac{\mathcal Q(\mathbf U_R)-\mathcal Q(\mathbf U_L)}{2\alpha}.
\end{equation}
Because $\eta$ is convex and $H_\alpha(\mathbf U_L,\mathbf U_R)$ is the average of $\mathcal R$ over $[-\alpha,\alpha]$, Jensen's inequality gives
\[
\eta\!\left(H_\alpha(\mathbf U_L,\mathbf U_R)\right)
\le
\frac{1}{2\alpha}\int_{-\alpha}^{\alpha}\eta(\mathcal R(\xi))\,d\xi.
\]
Combining this with \eqref{eq:avg-entropy-bound} yields \eqref{eq:2point-entropy}. Finally, \eqref{eq:Halpha-average} shows that $H_\alpha(\mathbf U_L,\mathbf U_R)$ is an average of states in the convex set $G$, hence $H_\alpha(\mathbf U_L,\mathbf U_R)\in G$.
\end{proof}

\begin{remark}
Theorem~\ref{thm:2point-entropy} is the entropy analogue of the weak positivity of the Lax--Friedrichs first-order update. Its proof is elementary but conceptually important: it uses only three ingredients---a control-volume identity for the Riemann solution, the entropy inequality, and convexity of $\eta$. The resulting bound is valid for \emph{any} convex entropy pair and therefore serves as a flexible first-order entropy building block.
\end{remark}

\subsection{Weak entropy stability}
We now establish the weak entropy stability for the forward Euler time discretization. 
Note that the Gauss--Lobatto quadrature nodes satisfy
\[
x_j^{(1)} = x_{j-\frac12}^+,\qquad x_j^{(L)} = x_{j+\frac12}^-.
\]
Denote the endpoint nodal values at time level $n$ by
\[
\mathbf U_{j,1}^n = \mathbf U_{j-\frac12}^{+,n},
\qquad
\mathbf U_{j,L}^n = \mathbf U_{j+\frac12}^{-,n}.
\]
Let the numerical flux at $x_{j+\frac12}$ be
\begin{equation}\label{eq:LFflux}
\widehat{\mathbf F}_{j+\frac12}^n
=
\frac{\mathbf F(\mathbf U_{j,L}^n)+\mathbf F(\mathbf U_{j+1,1}^n)}{2}
-
\frac{\alpha}{2}\bigl(\mathbf U_{j+1,1}^n-\mathbf U_{j,L}^n\bigr),
\end{equation}
where $\alpha>0$ bounds the local Riemann wave speeds. (The following idea can be directly extended to local Lax--Friedrichs flux and HLL-type flux, which are detailed in the submitted version of this manuscript.) 
The candidate cell-average update is
\begin{equation}\label{eq:avg-update-LF}
\bar{\mathbf U}_j^\star
=
\bar{\mathbf U}_j^n
-
\lambda\bigl(\widehat{\mathbf F}_{j+\frac12}^n-\widehat{\mathbf F}_{j-\frac12}^n\bigr),
\qquad
\lambda := \frac{\Delta t}{\Delta x}.
\end{equation}

\begin{proposition}[Weak entropy stability]\label{prop:canonical-budget}
Under the CFL condition
\begin{equation}\label{eq:CFL-canonical}
\lambda \alpha \le \omega_1,
\end{equation}
\eqref{eq:avg-update-LF} can be rewritten as a convex combination form as 
\begin{align}
\bar{\mathbf U}_j^\star
&=
\sum_{\nu=2}^{L-1}\omega_\nu \mathbf U_{j,\nu}^n
+
\bigl(\omega_1-\lambda \alpha\bigr)\mathbf U_{j,1}^n
+
\bigl(\omega_L-\lambda \alpha\bigr)\mathbf U_{j,L}^n \notag\\
&\qquad
+\lambda\alpha
H_{\alpha}\!\left(\mathbf U_{j-1,L}^n,\mathbf U_{j+1,1}^n\right)
+\lambda\alpha
H_{\alpha}\!\left(\mathbf U_{j,1}^n,\mathbf U_{j,L}^n\right).
\label{eq:canonical-decomp}
\end{align}
Consequently, we obtain the weak entropy stability: 
\begin{equation}\label{eq:canonical-weak-local}
\eta(\bar{\mathbf U}_j^\star)\le 
B_j^n
:=
\sum_{\nu=1}^{L}\omega_\nu \eta(\mathbf U_{j,\nu}^n)
-
\lambda\bigl(\widehat{\mathcal Q}_{j+\frac12}^n-\widehat{\mathcal Q}_{j-\frac12}^n\bigr),
\end{equation}
with numerical entropy flux
\begin{equation}\label{eq:numerical-entropy-flux}
\widehat{\mathcal Q}_{j+\frac12}^n
:=
\frac{\mathcal Q(\mathbf U_{j,L}^n)+\mathcal Q(\mathbf U_{j+1,1}^n)}{2} - \frac{\alpha}{2}\bigl(\eta(\mathbf U_{j+1,1}^n)-\eta(\mathbf U_{j,L}^n)\bigr).
\end{equation}
If the domain is periodic, then
\begin{equation}\label{eq:canonical-budget-global}
\sum_j B_j^n
=
\sum_j \sum_{\nu=1}^L \omega_\nu \eta(\mathbf U_{j,\nu}^n)
-
\lambda \sum_j \bigl(\widehat{\mathcal Q}_{j+\frac12}^n-\widehat{\mathcal Q}_{j-\frac12}^n\bigr).
\end{equation}
\end{proposition}

\begin{proof}
Using \eqref{eq:cellavg-quad}, \eqref{eq:avg-update-LF}, and \eqref{eq:LFflux}, we obtain
\begin{align*}
\bar{\mathbf U}_j^\star
&=
\sum_{\nu=1}^L \omega_\nu \mathbf U_{j,\nu}^n
-\frac{\lambda}{2}\Bigl[
\mathbf F(\mathbf U_{j,L}^n)+\mathbf F(\mathbf U_{j+1,1}^n)-\alpha(\mathbf U_{j+1,1}^n-\mathbf U_{j,L}^n) \\
&\qquad\qquad\qquad\qquad
-\mathbf F(\mathbf U_{j-1,L}^n)-\mathbf F(\mathbf U_{j,1}^n)+\alpha(\mathbf U_{j,1}^n-\mathbf U_{j-1,L}^n)
\Bigr].
\end{align*}
Regrouping terms yields exactly \eqref{eq:canonical-decomp}. Under \eqref{eq:CFL-canonical}, all coefficients are nonnegative and sum to one.  Using the two-point entropy inequality in Theorem~\ref{thm:2point-entropy} at the two interfaces gives
\begin{align*}
\eta(\bar{\mathbf U}_j^\star)
&\le
\sum_{\nu=2}^{L-1}\omega_\nu \eta(\mathbf U_{j,\nu}^n)
+
\bigl(\omega_1-\lambda \alpha\bigr)\eta(\mathbf U_{j,1}^n)
+
\bigl(\omega_L-\lambda \alpha\bigr)\eta(\mathbf U_{j,L}^n) \\
&\quad
+\lambda \alpha
\left[
\frac{\eta(\mathbf U_{j-1,L}^n)+\eta(\mathbf U_{j+1,1}^n)}{2}
-
\frac{\mathcal Q(\mathbf U_{j+1,1}^n)-\mathcal Q(\mathbf U_{j-1,L}^n)}{2\alpha}
\right] \\
&\quad
+\lambda \alpha
\left[
\frac{\eta(\mathbf U_{j,1}^n)+\eta(\mathbf U_{j,L}^n)}{2}
-
\frac{\mathcal Q(\mathbf U_{j,L}^n)-\mathcal Q(\mathbf U_{j,1}^n)}{2\alpha}
\right].
\end{align*}
After collecting terms we obtain \eqref{eq:canonical-weak-local}. For the global identity \eqref{eq:canonical-budget-global}, note that in the periodic case each nodal entropy at a left endpoint receives the complementary coefficient coming from the right-neighbor term of the cell to its left, and similarly each right endpoint receives the complementary coefficient from the left-neighbor term of the cell to its right. Hence the endpoint weights recombine exactly to $\omega_1$ and $\omega_L$ when summed over $j$. The entropy flux contribution telescopes in the usual way.
\end{proof}

\begin{remark}
The weak entropy stability in Proposition~\ref{prop:canonical-budget} exhibits the same structure as weak positivity: the high-order candidate average is controlled by a convex combination of internal nodal states and two-point first-order entropy-admissible interface states. The local budget involves a one-cell stencil, while its global sum reduces to the expected nodal entropy sum at time level $n$ plus a telescoping entropy flux contribution.
\end{remark}

\subsection{Compatible local entropy budgets}

For the rest of the paper we do not require a specific form of the budget; we only need a local budget and, for global results, a compatibility condition on the sum of these budgets.

\begin{assumption}[Weak entropy budget]\label{ass:weakE}
For each cell and each candidate time-step or stagewise update under consideration, there exists a real number $B_j^n$ such that
\begin{equation}\label{eq:weakE}
\eta(\bar{\mathbf U}_j^\star)\le B_j^n.
\end{equation}
\end{assumption}

\begin{assumption}[Global budget compatibility]\label{ass:budget-compat}
The local budgets satisfy
\begin{equation}\label{eq:budget-compat}
\sum_j B_j^n
\le
\sum_j \mathcal E(\mathbf U_j^n)
-
\lambda \sum_j \bigl(\widehat{\mathcal Q}_{j+\frac12}^n-\widehat{\mathcal Q}_{j-\frac12}^n\bigr)
\end{equation}
for some numerical entropy fluxes $\widehat{\mathcal Q}_{j+\frac12}^n$.
\end{assumption}

\begin{remark}
The budgets in Proposition~\ref{prop:canonical-budget} satisfy Assumptions~\ref{ass:weakE} and \ref{ass:budget-compat}. {\bf In other schemes the construction may differ.  The EPO limiter itself only needs the scalar quantities $B_j^n$, not the details of how they were obtained.}
\end{remark}

\section{The oscillation module: COS-type admissibility and oscillation radii}\label{sec:oscillation}

The third component of EPO is oscillation suppression. In this paper we model it as closely as possible on the \textsf{COS(DG)} philosophy of Cao, Huang, Li, and Wu \cite{CaoHuangLiWu2026}. The essential point is that the COS mechanism is not a flux modification, not an artificial-viscosity closure, and not a troubled-cell projection defined in a different state space. Instead, it is a \emph{cellwise convex scaling} applied after the candidate high-order update. In a DG method the input is the candidate stage polynomial on each cell. In a FV method the input is any conservative reconstruction polynomial or reconstructed subcell profile whose mean equals the candidate cell average. In both cases the oscillation module acts only through the pair $(\bar{\mathbf U}_j^\star,\mathbf U_j^\star)$, where $\bar{\mathbf U}_j^\star$ is the candidate updated cell average and $\mathbf U_j^\star$ denotes the candidate cellwise polynomial or reconstruction.

\subsection{Canonical COS operator for DG and FV candidate states}

Let $I_j$ be a one-dimensional cell of length $h_j$. Denote by $\mathbf U_j^\star(x)$ the candidate cellwise function on $I_j$ and by $\bar{\mathbf U}_j^\star$ its cell average. The canonical COS operator is defined cellwise by
\begin{equation}\label{eq:cos-op-1d}
\mathcal F^{\mathrm{COS}}(\mathbf U_h^\star)\big|_{I_j}:=(1-\lambda_j^{\mathrm{COS}})\,\bar{\mathbf U}_j^\star+\lambda_j^{\mathrm{COS}}\,\mathbf U_j^\star(x),\qquad 0\le \lambda_j^{\mathrm{COS}}\le 1.
\end{equation}
At the nodal level this is simply
\begin{equation}\label{eq:cos-op-1d-nodes}
\bigl(\mathcal F^{\mathrm{COS}}(\mathbf U_h^\star)\bigr)_{j,\nu}=\bar{\mathbf U}_j^\star+\lambda_j^{\mathrm{COS}}\bigl(\mathbf U_{j,\nu}^\star-\bar{\mathbf U}_j^\star\bigr),\qquad \nu=1,\dots,L.
\end{equation}
Thus the canonical COS process is literally a scaling along the same cell-average-anchored ray used throughout the present paper.

\subsection{Entropy-induced discrepancy measures and the canonical COS coefficient}

To define the coefficient $\lambda_j^{\mathrm{COS}}$ we follow \textsf{COS(DG)} and use a convex reference function $g$, typically chosen as a convex entropy of the underlying system. For two vector-valued functions $\mathbf W_1$ and $\mathbf W_2$ on a cell $K$, we consider the entropy-induced Jensen distance
\begin{equation}\label{eq:cos-jensen-distance}
\|\mathbf W_1-\mathbf W_2\|_{L_g^2(K)}^2:=\int_K 4\Bigl(g(\mathbf W_1)+g(\mathbf W_2)-2g\!\left(\frac{\mathbf W_1+\mathbf W_2}{2}\right)\Bigr)\,dx,
\end{equation}
and, when convenient, the frozen-Hessian surrogate
\begin{equation}\label{eq:cos-frozen-distance}
\|\mathbf W_1-\mathbf W_2\|_{L_{g,\mathrm{fr}}^2(K)}^2:=\int_K (\mathbf W_1-\mathbf W_2)^\top \nabla^2 g(\bar{\mathbf U}_j^\star)(\mathbf W_1-\mathbf W_2)\,dx.
\end{equation}
The first definition is the exact entropy-induced distance used in \textsf{COS(DG)}; the second, denoted by $L_{g,\mathrm{fr}}^2$, is the frozen-Hessian variant introduced there to avoid evaluating $g$ outside the admissible state set when extrapolated polynomials are used. In either case, the distance is nonnegative and vanishes only when the two arguments coincide almost everywhere on $K$. Whenever the frozen-Hessian realization is used, one simply replaces $L_g^2$ below by $L_{g,\mathrm{fr}}^2$ in the corresponding formulas.

In one space dimension, let $I_{j-1}$ and $I_{j+1}$ be the immediate neighbors of $I_j$. Define
\begin{equation}\label{eq:cos-sigma-1d}
\Theta_j^{\pm}:=\|\mathbf U_{j\pm1}^\star-\bar{\mathbf U}_j^\star\|_{L_g^2(I_{j\pm1})}^2,\qquad
\sigma_j^{\pm}:=\begin{cases}
C_k\dfrac{\|\mathbf U_{j\pm1}^\star-\mathbf U_j^\star\|_{L_g^2(I_{j\pm1})}^2}{\Theta_j^{\pm}},& \Theta_j^{\pm}\ge \varepsilon_j, \\[1.0em]
0,&\text{otherwise},
\end{cases}
\end{equation}
where $C_k>0$ is the polynomial-degree-dependent constant used in \textsf{COS(DG)} and $\varepsilon_j>0$ is a roundoff floor. Set
\begin{align}
\sigma_j^{\mathrm{COS}}
&:=\sigma_j^-+\sigma_j^+,
\notag\\
\lambda_j^{\mathrm{COS}}
&:=\exp\!\left(
-\frac{\alpha_j^{\mathrm{COS}}\,\Delta t}{h_j}\,\sigma_j^{\mathrm{COS}}
\right), \quad \alpha_j^{\mathrm{COS}}
:=\max_{j-1\le \ell\le j+1}
\operatorname{spr}\!\left(
\frac{\partial \mathbf F}{\partial \mathbf U}(\bar{\mathbf U}_\ell^\star)
\right). 
\label{eq:cos-lambda-1d}
\end{align}
Here $\operatorname{spr}(A)$ denotes the spectral radius. In the scalar case $\alpha_j^{\mathrm{COS}}$ reduces to a local bound on $|f'(u)|$. Formula \eqref{eq:cos-lambda-1d} is the canonical one-dimensional COS coefficient.

\begin{remark}[Local scale and evolution invariance]
The coefficient $\lambda_j^{\mathrm{COS}}$ is dimensionless. In the scalar case, when $g(u)=u^2/2$, the ratio in \eqref{eq:cos-sigma-1d} is invariant under affine rescalings $u\mapsto au+b$ with $a\neq0$, hence $\sigma_j^{\mathrm{COS}}$ and $\lambda_j^{\mathrm{COS}}$ are unchanged. Under a uniform rescaling of the wave speeds, for example replacing $\partial_t \mathbf U+\partial_x\mathbf F(\mathbf U)=0$ by $\partial_t \mathbf U+\kappa\partial_x\mathbf F(\mathbf U)=0$ at fixed CFL, the factor $\alpha_j^{\mathrm{COS}}\Delta t/h_j$ remains invariant. These are exactly the local scale- and evolution-invariant structures emphasized in \textsf{COS(DG)}.
\end{remark}

\subsection{A local COS variant in one space dimension}

The canonical coefficient \eqref{eq:cos-lambda-1d} damps every cell by using both immediate neighbors. Following \COSDG{}, one may activate only those interfaces that appear shock-like under a symmetry-preserving eigenvalue test. Define
\[
\Lambda^+(\mathbf U):=\lambda_{\max}\!\left(\frac{\partial \mathbf F}{\partial \mathbf U}(\mathbf U)\right),
\qquad
\Lambda^-(\mathbf U):=\lambda_{\min}\!\left(\frac{\partial \mathbf F}{\partial \mathbf U}(\mathbf U)\right),
\]
and fix a tolerance $0<\delta<1$. The local interface marker is
\begin{equation}\label{eq:cos-local-criterion-1d}
\chi_{j+\frac12}
:=
\begin{cases}
1,
& \Lambda^+(\bar{\mathbf U}_j^\star)-\Lambda^+(\bar{\mathbf U}_{j+1}^\star)
>
\delta\Bigl(
|\Lambda^+(\bar{\mathbf U}_j^\star)|
+
|\Lambda^+(\bar{\mathbf U}_{j+1}^\star)|
\Bigr),\\[0.7em]
1,
& \Lambda^-(\bar{\mathbf U}_j^\star)-\Lambda^-(\bar{\mathbf U}_{j+1}^\star)
>
\delta\Bigl(
|\Lambda^-(\bar{\mathbf U}_j^\star)|
+
|\Lambda^-(\bar{\mathbf U}_{j+1}^\star)|
\Bigr),\\[0.7em]
0,
& \text{otherwise.}
\end{cases}
\end{equation}
The corresponding local COS coefficient is obtained by replacing \eqref{eq:cos-sigma-1d} with
\begin{equation}\label{eq:cos-sigma-local-1d}
\widetilde \sigma_j^\pm := \chi_{j\pm\frac12}\sigma_j^\pm,
\qquad
\widetilde \sigma_j^{\mathrm{COS}}
:=
\widetilde \sigma_j^-+\widetilde \sigma_j^+,
\qquad
\widetilde \lambda_j^{\mathrm{COS}}
:=
\exp\!\left(
-\frac{\alpha_j^{\mathrm{COS}}\Delta t}{h_j}\,
\widetilde \sigma_j^{\mathrm{COS}}
\right).
\end{equation}
Equivalently, one simply multiplies the contribution of each interface by its marker before forming the total discrepancy. When both neighboring markers vanish, the oscillation module is inactive in cell $I_j$.

\begin{remark}
	The COS operator uses only the candidate cellwise function, its cell average, and immediate neighboring-cell data entering $\lambda_j^{\mathrm{COS}}$; it does not modify the conservative update of the cell average. In implementations it is often applied only at the final RK stage, but for the fully stagewise SSP-EPO theorem it may equally well be applied stage by stage.
\end{remark}

\subsection{Elementary properties of COS operator}

The next proposition records the properties of the COS step that are used later in EPO.

\begin{proposition}[Elementary properties of the canonical COS operator]\label{prop:cos-basic}
Let $\mathbf U_h^\star$ be any candidate DG or FV state and let $\mathcal F^{\mathrm{COS}}$ be defined by \eqref{eq:cos-op-1d}. Then the following hold.
\begin{enumerate}[label=(\alph*),leftmargin=2em]
\item \textbf{Mean preservation.} The cell average is preserved:
$$
\frac1{h_j}\int_{I_j}\mathcal F^{\mathrm{COS}}(\mathbf U_h^\star)\,dx=\bar{\mathbf U}_j^\star.
$$
\item \textbf{$L^2$ nonexpansiveness.} The operator is nonexpansive in $L^2$:
$$
\|\mathcal F^{\mathrm{COS}}(\mathbf U_h^\star)\|_{L^2(\Omega)}\le \|\mathbf U_h^\star\|_{L^2(\Omega)}.
$$
\item \textbf{Continuous entropy inheritance.} For every convex entropy $\eta$,
$$
\int_{I_j}\eta\!\left(\mathcal F^{\mathrm{COS}}(\mathbf U_h^\star)\right)dx\le \int_{I_j}\eta(\mathbf U_j^\star(x))\,dx.
$$
\item \textbf{Quadrature entropy inheritance.} Whenever the nodal states and the cell average lie in $D_\eta$, the local quadrature entropy also decreases:
\begin{equation}\label{eq:cos-quad-entropy-inheritance}
\mathcal E_j\!\left(\mathcal F^{\mathrm{COS}}(\mathbf U_h^\star)\right)
\le
\mathcal E_j(\mathbf U_j^\star).
\end{equation}
\end{enumerate}
\end{proposition}

\begin{proof}
For the proofs of parts (a)--(c), see \cite{CaoHuangLiWu2026}. 

For part (d), evaluate the same convexity inequality at each nodal state and sum with the positive quadrature weights $\omega_\nu$:
\[
\mathcal E_j\!\left(\mathcal F^{\mathrm{COS}}(\mathbf U_h^\star)\right)
\le
(1-\lambda)\eta(\bar{\mathbf U})+\lambda\mathcal E_j(\mathbf U_j^\star).
\]
Because $\bar{\mathbf U}=\sum_\nu \omega_\nu \mathbf U_{j,\nu}^\star$ and $D_\eta$ is convex, Jensen's inequality gives
\[
\eta(\bar{\mathbf U})\le \mathcal E_j(\mathbf U_j^\star).
\]
Substituting this bound yields \eqref{eq:cos-quad-entropy-inheritance}.  
\end{proof}

\subsection{Canonical COS admissible segments and general COS-compatible sets}

The canonical COS coefficient naturally determines a canonical oscillation-admissible segment along the scaling ray:
\begin{equation}\label{eq:canonical-cos-segment}
	\mathcal O_j^{\mathrm{COS},\star}
	:=
	\Bigl\{
	\mathcal S_{\bar{\mathbf U}_j^\star}(\theta;\mathbf U_j^\star):0\le \theta\le \lambda_j^{\mathrm{COS}}
	\Bigr\}.
\end{equation}
This is the most direct translation of \textsf{COS(DG)} into the EPO framework.

\begin{proposition}[Canonical COS admissible segment]\label{prop:canonical-cos-segment}
	The set $\mathcal O_j^{\mathrm{COS},\star}$ defined by \eqref{eq:canonical-cos-segment} is nonempty, closed, and convex. It contains the constant state $\bar{\mathbf U}_j^\star$, and its oscillation radius along the scaling ray is exactly
	\[
	\theta_j^{\mathrm O}=\lambda_j^{\mathrm{COS}}.
	\]
\end{proposition}

\begin{proof}
	The interval $[0,\lambda_j^{\mathrm{COS}}]$ is nonempty and closed. Since
	$\theta\mapsto \mathcal S_{\bar{\mathbf U}_j^\star}(\theta;\mathbf U_j^\star)$
	is continuous and affine, its image is a closed line segment, hence closed and convex. At $\theta=0$ one obtains the constant state $\bar{\mathbf U}_j^\star$. By construction, the largest admissible scaling parameter on this segment is $\lambda_j^{\mathrm{COS}}$.
\end{proof}

The canonical segment is the basic oscillation set used in this paper. For greater flexibility, we also allow more general oscillation-admissible sets that retain the same ray-based geometry.

\begin{assumption}[General COS-compatible oscillation-suppressing sets]\label{ass:osc}
	For each cell $I_j$ and candidate average $\bar{\mathbf U}_j^\star$, there is a nonempty closed convex set $\mathcal O_j^\star$ such that
	\begin{equation}\label{eq:O-constant}
		\bar{\mathbf U}_j^\star\in \mathcal O_j^\star.
	\end{equation}
	The canonical segment \eqref{eq:canonical-cos-segment} is one admissible choice.
\end{assumption}

\subsection{The oscillation radius}

For a candidate nodal array $\mathbf U_j^\star$, define the oscillation radius by
\begin{equation}\label{eq:thetaO}
	\theta_j^{\mathrm O}
	:=
	\sup\Bigl\{
	\theta\in[0,1]:
	\mathcal S_{\bar{\mathbf U}_j^\star}(\theta;\mathbf U_j^\star)\in \mathcal O_j^\star
	\Bigr\}.
\end{equation}
Because $\bar{\mathbf U}_j^\star=\mathcal S_{\bar{\mathbf U}_j^\star}(0;\mathbf U_j^\star)$ belongs to $\mathcal O_j^\star$, the admissible set in \eqref{eq:thetaO} is nonempty. Since $\mathcal O_j^\star$ is closed, the supremum is attained.

\begin{proposition}[Basic properties of the oscillation radius]\label{prop:thetaO-basic}
	Under Assumption~\ref{ass:osc}, the number $\theta_j^{\mathrm O}$ is well defined and belongs to $[0,1]$. Moreover,
	\begin{equation}\label{eq:thetaO-property}
		\mathcal S_{\bar{\mathbf U}_j^\star}(\theta;\mathbf U_j^\star)\in \mathcal O_j^\star
		\qquad\text{for all }0\le \theta\le \theta_j^{\mathrm O}.
	\end{equation}
	In the canonical COS case $\mathcal O_j^\star=\mathcal O_j^{\mathrm{COS},\star}$, one has $\theta_j^{\mathrm O}=\lambda_j^{\mathrm{COS}}$.
\end{proposition}

\begin{proof}
	The set in \eqref{eq:thetaO} is nonempty because $\theta=0$ is admissible, and it is contained in $[0,1]$, so the supremum exists. Since $\mathcal O_j^\star$ is closed and
	$\theta\mapsto \mathcal S_{\bar{\mathbf U}_j^\star}(\theta;\mathbf U_j^\star)$
	is continuous, the admissible set of $\theta$ values is closed.
	
	Now let $\theta_1$ be admissible and let $0\le \theta\le \theta_1$. Then
	\[
	\mathcal S_{\bar{\mathbf U}_j^\star}(\theta;\mathbf U_j^\star)
	=
	\Bigl(1-\frac{\theta}{\theta_1}\Bigr)\bar{\mathbf U}_j^\star
	+
	\frac{\theta}{\theta_1}
	\mathcal S_{\bar{\mathbf U}_j^\star}(\theta_1;\mathbf U_j^\star).
	\]
	Both states on the right-hand side belong to the convex set $\mathcal O_j^\star$, so the whole convex combination also belongs to $\mathcal O_j^\star$. This proves \eqref{eq:thetaO-property}. The canonical COS statement follows from Proposition~\ref{prop:canonical-cos-segment}.
\end{proof}

\subsection{A COS-compatible convex-gauge realization}

The next proposition records a useful general class of COS-compatible oscillation sets.

\begin{proposition}[Gauge-based COS module]\label{prop:seminorm-O}
	Let $\Omega_j^\star:(\mathbb R^m)^L\to [0,\infty)$ be continuous, convex, and positively homogeneous:
	\[
	\Omega_j^\star(c\mathbf W)=c\,\Omega_j^\star(\mathbf W)
	\qquad\text{for all }c\ge 0.
	\]
	Fix a threshold $\Gamma_j^\star\ge 0$ and define
	\begin{equation}\label{eq:seminorm-O}
		\mathcal O_j^\star
		:=
		\Bigl\{
		\mathbf U_j:
		\Omega_j^\star\bigl(\mathbf U_j-\bar{\mathbf U}_j^\star\bigr)\le \Gamma_j^\star
		\Bigr\},
	\end{equation}
	where $\mathbf U_j-\bar{\mathbf U}_j^\star$ means the nodal array whose $\nu$th entry is
	$\mathbf U_{j,\nu}-\bar{\mathbf U}_j^\star$.
	Then $\mathcal O_j^\star$ satisfies Assumption~\ref{ass:osc}. Moreover,
	\begin{equation}\label{eq:thetaO-explicit}
		\theta_j^{\mathrm O}=
		\begin{cases}
			1,
			&
			\Omega_j^\star\bigl(\mathbf U_j^\star-\bar{\mathbf U}_j^\star\bigr)\le \Gamma_j^\star,\\[0.8em]
			\dfrac{\Gamma_j^\star}{\Omega_j^\star\bigl(\mathbf U_j^\star-\bar{\mathbf U}_j^\star\bigr)},
			&
			\Omega_j^\star\bigl(\mathbf U_j^\star-\bar{\mathbf U}_j^\star\bigr)>\Gamma_j^\star.
		\end{cases}
	\end{equation}
\end{proposition}

\begin{proof}
	Closedness follows from continuity of $\Omega_j^\star$, and convexity follows because \eqref{eq:seminorm-O} is a sublevel set of the convex function
	\[
	\mathbf U_j\longmapsto
	\Omega_j^\star\bigl(\mathbf U_j-\bar{\mathbf U}_j^\star\bigr).
	\]
	The constant state belongs to $\mathcal O_j^\star$ because the nodal array with every entry equal to $\bar{\mathbf U}_j^\star$ satisfies
	\[
	\Omega_j^\star(\mathbf 0)=0.
	\]
	Hence Assumption~\ref{ass:osc} holds.
	
	Now observe that along the scaling ray,
	\[
	\mathcal S_{\bar{\mathbf U}_j^\star}(\theta;\mathbf U_j^\star)-\bar{\mathbf U}_j^\star
	=
	\theta\bigl(\mathbf U_j^\star-\bar{\mathbf U}_j^\star\bigr),
	\]
	that is, the $\nu$th nodal entry equals
	$\theta(\mathbf U_{j,\nu}^\star-\bar{\mathbf U}_j^\star)$.
	By positive homogeneity,
	\[
	\Omega_j^\star\Bigl(
	\mathcal S_{\bar{\mathbf U}_j^\star}(\theta;\mathbf U_j^\star)-\bar{\mathbf U}_j^\star
	\Bigr)
	=
	\theta\,\Omega_j^\star\bigl(\mathbf U_j^\star-\bar{\mathbf U}_j^\star\bigr).
	\]
	Therefore the admissible values of $\theta$ are exactly those in \eqref{eq:thetaO-explicit}.
\end{proof}

\begin{example}[Entropy-deviation oscillation module]\label{ex:entropy-O}
Assume $\eta\in C^1(G)$ is convex, and fix a cell average $\bar{\mathbf U}_j^\star\in G$. Define
\begin{equation}\label{eq:entropy-O}
\Omega_{j,\eta}^\star(\mathbf V):=\sum_{\nu=1}^L \omega_\nu\Bigl[\eta({\mathbf V}_\nu)-\eta(\bar{\mathbf U}_j^\star)-\nabla\eta(\bar{\mathbf U}_j^\star)\cdot ({\mathbf V}_\nu-\bar{\mathbf U}_j^\star)\Bigr].
\end{equation}
Then $\Omega_{j,\eta}^\star(\mathbf V)\ge 0$ for all $\mathbf V$, the mapping $\mathbf V\mapsto \Omega_{j,\eta}^\star(\mathbf V)$ is convex and continuous, and $\Omega_{j,\eta}^\star(\bar{\mathbf U}_j^\star)=0$. Consequently, for any threshold $\Gamma_j^\star\ge 0$, the sublevel set
\[
\mathcal O_{j,\eta}^\star:=\bigl\{\mathbf V:\Omega_{j,\eta}^\star(\mathbf V)\le \Gamma_j^\star\bigr\}
\]
is a COS-compatible oscillation-suppressing set. Along any scaling ray
\[
\theta\longmapsto \Omega_{j,\eta}^\star\bigl(\mathcal S_{\bar{\mathbf U}_j^\star}(\theta;\mathbf U_j^\star)\bigr),
\]
the profile is convex and nondecreasing on $[0,1]$. Hence the associated oscillation radius is found by a one-dimensional root search whenever $\Gamma_j^\star$ is active.
\end{example}

\begin{proof}
The pointwise convexity inequality
\[
\eta({\mathbf V}_\nu)\ge \eta(\bar{\mathbf U}_j^\star)+\nabla\eta(\bar{\mathbf U}_j^\star)\cdot ({\mathbf V}_\nu-\bar{\mathbf U}_j^\star)
\]
implies $\Omega_{j,\eta}^\star(\mathbf V)\ge 0$. Convexity and continuity follow because $\Omega_{j,\eta}^\star$ is a weighted sum of convex continuous functions of the individual nodal states. The constant state gives equality in the first-order convexity inequality, hence $\Omega_{j,\eta}^\star(\bar{\mathbf U}_j^\star)=0$, and therefore $\mathcal O_{j,\eta}^\star$ is closed and convex. Now define
\[
\Phi_j(\theta):=\Omega_{j,\eta}^\star\bigl(\mathcal S_{\bar{\mathbf U}_j^\star}(\theta;\mathbf U_j^\star)\bigr).
\]
Convexity of $\Phi_j$ follows from convexity of $\Omega_{j,\eta}^\star$ and affinity of the scaling ray. Moreover,
\[
\Phi_j'(0)=\sum_{\nu=1}^L \omega_\nu\nabla\eta(\bar{\mathbf U}_j^\star)\cdot (\mathbf U_{j,\nu}^\star-\bar{\mathbf U}_j^\star)=\nabla\eta(\bar{\mathbf U}_j^\star)\cdot\sum_{\nu=1}^L \omega_\nu(\mathbf U_{j,\nu}^\star-\bar{\mathbf U}_j^\star)=0,
\]
because the nodal deviations have zero weighted mean. Since $\Phi_j$ is convex and has zero derivative at $\theta=0$, it is nondecreasing on $[0,1]$. The final statement follows.
\end{proof}
\section{The entropy module and the local EPO admissible set}\label{sec:entropy_limiter}

This section formulates the entropy limiter and the associated local admissible set.

\subsection{The entropy profile along a scaling ray}

Fix a cell $I_j$ and write
\[
\mathbf U_j^\star=(\mathbf U_{j,1}^\star,\dots,\mathbf U_{j,L}^\star),\qquad
\bar{\mathbf U}_j^\star = \sum_{\nu=1}^L \omega_\nu \mathbf U_{j,\nu}^\star.
\]
Whenever the whole candidate ray stays inside the entropy domain $D_\eta$, namely whenever
\[
\bar{\mathbf U}_j^\star + \theta(\mathbf U_{j,\nu}^\star-\bar{\mathbf U}_j^\star)\in D_\eta,
\qquad 0\le \theta\le 1,\ \nu=1,\dots,L,
\]
we define the direct entropy profile by
\begin{equation}\label{eq:Psi}
\Psi_j(\theta)
:=
\mathcal E\!\left(\mathcal S_{\bar{\mathbf U}_j^\star}(\theta;\mathbf U_j^\star)\right)
=
\sum_{\nu=1}^L
\omega_\nu\,
\eta\!\left(
\bar{\mathbf U}_j^\star + \theta(\mathbf U_{j,\nu}^\star-\bar{\mathbf U}_j^\star)
\right),
\qquad
0\le \theta\le 1.
\end{equation}

\begin{proposition}[Convexity and monotonicity of the direct entropy profile]\label{prop:Psi}
Assume \eqref{eq:weakP}, i.e., $\bar{\mathbf U}_j^\star\in G$, and assume that the full candidate ray stays inside $D_\eta$ as above. Then $\Psi_j$ is continuous, convex, and nondecreasing on $[0,1]$. Moreover,
\begin{equation}\label{eq:Psi0}
\Psi_j(0)=\eta(\bar{\mathbf U}_j^\star).
\end{equation}
\end{proposition}

\begin{proof}
Continuity follows from continuity of $\eta$ on $D_\eta$. Each map
\[
\theta \mapsto \eta\!\left(\bar{\mathbf U}_j^\star + \theta(\mathbf U_{j,\nu}^\star-\bar{\mathbf U}_j^\star)\right)
\]
is convex because it is the composition of the convex function $\eta$ with an affine map. Therefore $\Psi_j$, being a positive weighted sum of convex functions, is convex on $[0,1]$.

At $\theta=0$, every argument of $\eta$ equals $\bar{\mathbf U}_j^\star$, hence
\[
\Psi_j(0)
=
\sum_{\nu=1}^L \omega_\nu \eta(\bar{\mathbf U}_j^\star)
=
\eta(\bar{\mathbf U}_j^\star).
\]
For every $\theta\in[0,1]$, Jensen's inequality gives
\begin{align*}
\Psi_j(\theta)
&=
\sum_{\nu=1}^L
\omega_\nu\,
\eta\!\left(
\bar{\mathbf U}_j^\star + \theta(\mathbf U_{j,\nu}^\star-\bar{\mathbf U}_j^\star)
\right)
\\
&\ge
\eta\!\left(
\sum_{\nu=1}^L \omega_\nu
\bigl[\bar{\mathbf U}_j^\star + \theta(\mathbf U_{j,\nu}^\star-\bar{\mathbf U}_j^\star)\bigr]
\right)
=
\eta(\bar{\mathbf U}_j^\star)
=
\Psi_j(0),
\end{align*}
because the weighted deviations sum to zero. Thus $\theta=0$ is a global minimizer of the convex function $\Psi_j$ on $[0,1]$. A convex function on an interval that attains its minimum at the left endpoint is nondecreasing, and the claim follows.
\end{proof}

\subsection{An auxiliary full-ray entropy radius}

When the full candidate ray lies in $D_\eta$, one may define the auxiliary direct entropy radius by
\begin{equation}\label{eq:thetaE}
\theta_j^{\mathrm E}
:=
\sup\Bigl\{
\theta\in[0,1]:
\Psi_j(\theta)\le B_j^n
\Bigr\}.
\end{equation}

\begin{proposition}[Existence of the auxiliary full-ray entropy radius]\label{prop:thetaE}
Under Assumptions~\ref{ass:weakP} and \ref{ass:weakE}, and under the full-ray hypothesis of Proposition~\ref{prop:Psi}, the set in \eqref{eq:thetaE} is nonempty, and $\theta_j^{\mathrm E}\in [0,1]$ is well defined. Moreover,
\begin{equation}\label{eq:thetaE-property}
\mathcal E\!\left(\mathcal S_{\bar{\mathbf U}_j^\star}(\theta;\mathbf U_j^\star)\right)\le B_j^n
\qquad
\text{for all }0\le \theta\le \theta_j^{\mathrm E}.
\end{equation}
\end{proposition}

\begin{proof}
By Assumption~\ref{ass:weakE} and Proposition~\ref{prop:Psi},
\[
\Psi_j(0)=\eta(\bar{\mathbf U}_j^\star)\le B_j^n.
\]
Therefore $\theta=0$ is admissible, so the set in \eqref{eq:thetaE} is nonempty. Since the set is contained in $[0,1]$, the supremum exists. Because $\Psi_j$ is continuous, the admissible set is closed. Because $\Psi_j$ is nondecreasing on $[0,1]$, the admissible set is an interval of the form $[0,\theta_j^{\mathrm E}]$. This yields \eqref{eq:thetaE-property}.
\end{proof}

\begin{remark}
The auxiliary radius $\theta_j^{\mathrm E}$ is useful when the full candidate ray is known \emph{a priori} to stay inside $D_\eta$. In the general framework, however, the primary entropy construction is the positivity-first one below, because it guarantees that pointwise entropy is evaluated only on states already certified to lie in $G\subset D_\eta$.
\end{remark}

\subsection{The positivity / invariant-set admissibility radius}

The geometric admissibility radius is defined by
\begin{equation}\label{eq:thetaP}
\theta_j^{\mathrm P}
:=
\sup\Bigl\{
\theta\in[0,1]:
\bar{\mathbf U}_j^\star+\theta(\mathbf U_{j,\nu}^\star-\bar{\mathbf U}_j^\star)\in G
\text{ for all }\nu=1,\dots,L
\Bigr\}.
\end{equation}

\begin{proposition}[Existence of the geometric radius]\label{prop:thetaP}
Under Assumption~\ref{ass:weakP}, the set in \eqref{eq:thetaP} is nonempty, $\theta_j^{\mathrm P}\in[0,1]$ is well defined, and
\begin{equation}\label{eq:thetaP-property}
\mathcal S_{\bar{\mathbf U}_j^\star}(\theta;\mathbf U_j^\star)_\nu \in G
\qquad
\text{for all }\nu,\ 0\le \theta\le \theta_j^{\mathrm P}.
\end{equation}
\end{proposition}

\begin{proof}
Because $\bar{\mathbf U}_j^\star\in G$, the value $\theta=0$ is admissible. Hence the admissible set is nonempty. Convexity of $G$ implies that if some $\theta_1\in[0,1]$ is admissible, then every $0\le \theta\le \theta_1$ is also admissible. Closedness of $G$ and continuity of the scaling ray imply that the admissible set is closed. Therefore it is a compact interval $[0,\theta_j^{\mathrm P}]$, and \eqref{eq:thetaP-property} follows.
\end{proof}

\begin{proposition}[Positivity-first entropy radius]\label{prop:thetaE-after-P}
Assume Assumptions~\ref{ass:weakP} and \ref{ass:weakE}. Define the positivity-limited auxiliary state
\[
\mathbf U_j^{\mathrm P}
:=
\mathcal S_{\bar{\mathbf U}_j^\star}\!\bigl(\theta_j^{\mathrm P};\mathbf U_j^\star\bigr).
\]
Then $(\mathbf U_j^{\mathrm P})_\nu\in G$ for every $\nu$, and therefore the pointwise entropy is well defined at every node of $\mathbf U_j^{\mathrm P}$. Define the post-positivity entropy profile on the shortened admissible ray by
\[
\Psi_j^{\mathrm P}(\vartheta)
:=
\mathcal E\!\left(\mathcal S_{\bar{\mathbf U}_j^\star}(\vartheta;\mathbf U_j^{\mathrm P})\right),
\qquad 0\le \vartheta\le 1,
\]
and let
\[
\vartheta_j^{\mathrm E|\mathrm P}
:=
\sup\Bigl\{
\vartheta\in[0,1]:
\Psi_j^{\mathrm P}(\vartheta)\le B_j^n
\Bigr\}.
\]
Finally define the positivity-first entropy radius
\begin{equation}\label{eq:thetaPE}
\theta_j^{\mathrm{PE}}
:=
\theta_j^{\mathrm P}\,\vartheta_j^{\mathrm E|\mathrm P}.
\end{equation}
Then $\theta_j^{\mathrm{PE}}\in[0,1]$ is well defined and
\begin{equation}\label{eq:thetaPE-property}
\mathcal S_{\bar{\mathbf U}_j^\star}(\theta;\mathbf U_j^\star)_\nu\in G,
\qquad
\mathcal E\!\left(\mathcal S_{\bar{\mathbf U}_j^\star}(\theta;\mathbf U_j^\star)\right)\le B_j^n,
\qquad 0\le \theta\le \theta_j^{\mathrm{PE}},\ \nu=1,\dots,L.
\end{equation}
Moreover, if the original entropy profile $\Psi_j(\theta)$ is already well defined on the whole candidate ray $0\le \theta\le 1$, then
\begin{equation}\label{eq:PE-equals-min}
\theta_j^{\mathrm{PE}}=\min\{\theta_j^{\mathrm P},\theta_j^{\mathrm E}\},
\end{equation}
and the sequentially limited state may be written directly as
\[
\mathbf U_j^{\mathrm{PE}}
=
\mathcal S_{\bar{\mathbf U}_j^\star}\!\bigl(\theta_j^{\mathrm{PE}};\mathbf U_j^\star\bigr).
\]
\end{proposition}

\begin{proof}
Proposition~\ref{prop:thetaP} gives $(\mathbf U_j^{\mathrm P})_\nu\in G$ for every $\nu$. Since $\bar{\mathbf U}_j^\star\in G$ as well and $G\subset D_\eta$, convexity implies
\[
\mathcal S_{\bar{\mathbf U}_j^\star}(\vartheta;\mathbf U_j^{\mathrm P})_\nu
=
(1-\vartheta)\bar{\mathbf U}_j^\star+\vartheta(\mathbf U_j^{\mathrm P})_\nu
\in G\subset D_\eta
\qquad\text{for all }0\le \vartheta\le 1.
\]
Thus the profile $\Psi_j^{\mathrm P}$ is well defined on $[0,1]$. By Proposition~\ref{prop:Psi} applied to the candidate array $\mathbf U_j^{\mathrm P}$, the function $\Psi_j^{\mathrm P}$ is continuous, convex, and nondecreasing on $[0,1]$. At $\vartheta=0$ one has
\[
\Psi_j^{\mathrm P}(0)=\eta(\bar{\mathbf U}_j^\star)\le B_j^n
\]
by Assumption~\ref{ass:weakE}, so the admissible set defining $\vartheta_j^{\mathrm E|\mathrm P}$ is nonempty and forms an interval $[0,\vartheta_j^{\mathrm E|\mathrm P}]$.

If $\theta_j^{\mathrm P}=0$, then $\mathbf U_j^{\mathrm P}=\mathbf C(\bar{\mathbf U}_j^\star)$ and therefore $\Psi_j^{\mathrm P}\equiv \eta(\bar{\mathbf U}_j^\star)\le B_j^n$. In that case $\vartheta_j^{\mathrm E|\mathrm P}=1$ and $\theta_j^{\mathrm{PE}}=0$, so \eqref{eq:thetaPE-property} is immediate. Assume henceforth that $\theta_j^{\mathrm P}>0$, and fix $0\le \theta\le \theta_j^{\mathrm{PE}}$. Set $\vartheta:=\theta/\theta_j^{\mathrm P}$. Then $0\le \vartheta\le \vartheta_j^{\mathrm E|\mathrm P}$ and, by the ray-composition identity
\begin{equation}\label{eq:ray-composition}
\mathcal S_{\bar{\mathbf U}}(\vartheta;\mathcal S_{\bar{\mathbf U}}(\theta;\mathbf V))
=
\mathcal S_{\bar{\mathbf U}}(\vartheta\theta;\mathbf V),
\qquad 0\le \vartheta,\theta\le 1,
\end{equation}
we have
\[
\mathcal S_{\bar{\mathbf U}_j^\star}(\theta;\mathbf U_j^\star)
=
\mathcal S_{\bar{\mathbf U}_j^\star}(\vartheta;\mathbf U_j^{\mathrm P}).
\]
The nodal inclusion in $G$ follows because the right-hand side lies on the segment joining $\mathbf C(\bar{\mathbf U}_j^\star)$ to $\mathbf U_j^{\mathrm P}\in G^L$, while the entropy bound follows from the definition of $\vartheta_j^{\mathrm E|\mathrm P}$.

For the final statement, assume that the direct profile $\Psi_j$ is well defined on the whole candidate ray. Another application of \eqref{eq:ray-composition} gives
\[
\Psi_j^{\mathrm P}(\vartheta)
=
\mathcal E\!\left(\mathcal S_{\bar{\mathbf U}_j^\star}(\vartheta\theta_j^{\mathrm P};\mathbf U_j^\star)\right)
=
\Psi_j(\vartheta\theta_j^{\mathrm P}).
\]
Because $\Psi_j$ is nondecreasing on $[0,1]$, the admissible $\vartheta$ are exactly those satisfying $\vartheta\theta_j^{\mathrm P}\le \theta_j^{\mathrm E}$. Hence
\[
\vartheta_j^{\mathrm E|\mathrm P}
=
\min\left\{1,\frac{\theta_j^{\mathrm E}}{\theta_j^{\mathrm P}}\right\}
\qquad (\theta_j^{\mathrm P}>0),
\]
which yields \eqref{eq:PE-equals-min}. The representation of $\mathbf U_j^{\mathrm{PE}}$ then follows from \eqref{eq:ray-composition} again.
\end{proof}

\begin{remark}
Proposition~\ref{prop:thetaE-after-P} is the mathematically primary entropy construction in this paper. It makes the entropy module rigorous even when the entropy is only guaranteed to be well defined on $G$. The auxiliary full-ray radius $\theta_j^{\mathrm E}$ is retained because it is convenient in cases where the full candidate ray already lies in $D_\eta$, and then \eqref{eq:PE-equals-min} recovers the familiar formula $\min\{\theta_j^{\mathrm P},\theta_j^{\mathrm E}\}$.
\end{remark}

\subsection{The local EPO admissible set}

We now combine the three admissibility mechanisms in one local set. For a prescribed cell average $\bar{\mathbf U}\in G$ and budget $B\in\mathbb R$, define the mean-preserving affine space
\begin{equation}\label{eq:mean-space}
\mathfrak M(\bar{\mathbf U})
:=
\Bigl\{
\mathbf V\in(\mathbb R^m)^L:
\sum_{\nu=1}^L \omega_\nu {\mathbf V}_\nu = \bar{\mathbf U}
\Bigr\}.
\end{equation}
For the cell under consideration, define the local EPO admissible set
\begin{equation}\label{eq:Adef}
\mathcal A_j(B,\bar{\mathbf U})
:=
\Bigl\{
\mathbf V\in \mathfrak M(\bar{\mathbf U}) :
{\bf V}_\nu\in G\ \forall\nu,\ 
\mathcal E(\mathbf V)\le B,\ 
\mathbf V\in \mathcal O_j^\star
\Bigr\}.
\end{equation}

\begin{theorem}[Geometry of the local EPO admissible set]\label{thm:A-star}
Let $\bar{\mathbf U}\in G$ and suppose $B\ge \eta(\bar{\mathbf U})$. Under Assumption~\ref{ass:osc}, the set $\mathcal A_j(B,\bar{\mathbf U})$ is nonempty, closed, and convex. Moreover, 
\begin{equation}\label{eq:A-constant}
\bar{\mathbf U}\in \mathcal A_j(B,\bar{\mathbf U}),
\end{equation}
and for every $\mathbf V\in \mathcal A_j(B,\bar{\mathbf U})$ and every $\theta\in[0,1]$,
\begin{equation}\label{eq:A-star-property}
\mathcal S_{\bar{\mathbf U}}(\theta;\mathbf V)\in \mathcal A_j(B,\bar{\mathbf U}).
\end{equation}
\end{theorem}

\begin{proof}
The mean-preserving space $\mathfrak M(\bar{\mathbf U})$ is an affine closed convex subset of $(\mathbb R^m)^L$. The geometric constraint set
\[
G^L=\{\mathbf V: {\bf V}_\nu\in G\ \forall \nu\}
\]
is closed and convex because $G$ is closed and convex. The entropy sublevel set
\[
\{\mathbf V\in (\mathbb R^m)^L:\mathcal E(\mathbf V)\le B\}
\]
is closed and convex because $\mathcal E$ is continuous and convex. Finally, $\mathcal O_j^\star$ is closed and convex by Assumption~\ref{ass:osc}. Therefore $\mathcal A_j(B,\bar{\mathbf U})$, being the intersection of closed convex sets, is itself closed and convex.

To prove nonemptiness, note that the constant state $\bar{\mathbf U}$ belongs to $\mathfrak M(\bar{\mathbf U})$ by construction, lies in $G^L$ because $\bar{\mathbf U}\in G$, satisfies
\[
\mathcal E(\bar{\mathbf U})=\eta(\bar{\mathbf U})\le B,
\]
and belongs to $\mathcal O_j^\star$ by Assumption~\ref{ass:osc}. This proves \eqref{eq:A-constant}.

Finally, let $\mathbf V\in \mathcal A_j(B,\bar{\mathbf U})$ and $0\le \theta\le 1$. Since
\[
\mathcal S_{\bar{\mathbf U}}(\theta;\mathbf V)
=
(1-\theta)\bar{\mathbf U}+\theta \mathbf V,
\]
and both endpoints belong to the convex set $\mathcal A_j(B,\bar{\mathbf U})$, the scaled state also belongs to $\mathcal A_j(B,\bar{\mathbf U})$. This proves \eqref{eq:A-star-property}.
\end{proof}

\begin{remark}
Theorem~\ref{thm:A-star} is the geometric heart of EPO. In the COS-compatible setting used here, the local admissible set is not merely star-shaped; it is the intersection of closed convex sets coming from the mean constraint, the physical admissible set, the entropy budget, and the oscillation-suppressing constraint. This observation makes the final limiter both natural and mathematically transparent.
\end{remark}

\subsection{The EPO radius as a maximal radial projection}

With $\bar{\mathbf U}_j^\star$ and $B_j^n$ fixed, define the maximal radial admissibility radius by
\begin{equation}\label{eq:thetaA}
\theta_j^{\mathrm{EPO}}
:=
\sup\Bigl\{
\theta\in[0,1]:
\mathcal S_{\bar{\mathbf U}_j^\star}(\theta;\mathbf U_j^\star)\in
\mathcal A_j(B_j^n,\bar{\mathbf U}_j^\star)
\Bigr\}.
\end{equation}
In parallel, recall the positivity-first entropy radius $\theta_j^{\mathrm{PE}}$ from \eqref{eq:thetaPE} and the oscillation radius $\theta_j^{\mathrm O}$ from \eqref{eq:thetaO}. When the direct full-ray entropy profile is available, Proposition~\ref{prop:thetaE-after-P} gives $\theta_j^{\mathrm{PE}}=\min\{\theta_j^{\mathrm P},\theta_j^{\mathrm E}\}$.

\begin{theorem}[Min-formula for the EPO radius]\label{thm:min}
Under Assumptions~\ref{ass:weakP}, \ref{ass:weakE}, and \ref{ass:osc},
\begin{equation}\label{eq:minformula}
\theta_j^{\mathrm{EPO}}
=
\min\bigl\{\theta_j^{\mathrm{PE}},\theta_j^{\mathrm O}\bigr\}.
\end{equation}
If, in addition, the direct full-ray entropy profile is available on $[0,1]$, then
\begin{equation}\label{eq:minformula-direct}
\theta_j^{\mathrm{EPO}}
=
\min\bigl\{\theta_j^{\mathrm P},\theta_j^{\mathrm E},\theta_j^{\mathrm O}\bigr\}.
\end{equation}
Equivalently, the EPO limiter is the maximal radial projection of the candidate nodal state into the local EPO admissible set.
\end{theorem}

\begin{proof}
First, let
\[
\theta \le \min\{\theta_j^{\mathrm{PE}},\theta_j^{\mathrm O}\}.
\]
Then the scaled state $\mathcal S_{\bar{\mathbf U}_j^\star}(\theta;\mathbf U_j^\star)$ satisfies the geometric and entropy constraints by \eqref{eq:thetaPE-property}, and it satisfies the oscillation constraint by \eqref{eq:thetaO-property}. Mean preservation is automatic. Hence
\[
\mathcal S_{\bar{\mathbf U}_j^\star}(\theta;\mathbf U_j^\star)\in
\mathcal A_j(B_j^n,\bar{\mathbf U}_j^\star).
\]
Therefore
\[
\min\{\theta_j^{\mathrm{PE}},\theta_j^{\mathrm O}\}
\le
\theta_j^{\mathrm{EPO}}.
\]

Conversely, suppose $\theta$ is admissible for the set in \eqref{eq:thetaA}. Then
\[
\mathcal S_{\bar{\mathbf U}_j^\star}(\theta;\mathbf U_j^\star)\in \mathcal A_j(B_j^n,\bar{\mathbf U}_j^\star).
\]
By definition of $\mathcal A_j$, this implies separately that the geometric constraint, entropy constraint, and oscillation constraint are all satisfied. In particular, by the defining maximality property of $\theta_j^{\mathrm{PE}}$ in \eqref{eq:thetaPE-property} and of $\theta_j^{\mathrm O}$ in \eqref{eq:thetaO-property},
\[
\theta\le \theta_j^{\mathrm{PE}},\qquad
\theta\le \theta_j^{\mathrm O}.
\]
Therefore $\theta\le \min\{\theta_j^{\mathrm{PE}},\theta_j^{\mathrm O}\}$. Taking the supremum over all admissible $\theta$ shows
\[
\theta_j^{\mathrm{EPO}}
\le
\min\{\theta_j^{\mathrm{PE}},\theta_j^{\mathrm O}\}.
\]
Combining the two inequalities yields \eqref{eq:minformula}. The direct formula \eqref{eq:minformula-direct} is an immediate consequence of \eqref{eq:PE-equals-min}.
\end{proof}

\begin{remark}
Theorem~\ref{thm:min} gives the precise sense in which EPO is a \emph{single} local limiter rather than a sequential composition of three unrelated limiters. The rigorous primary formula is \eqref{eq:minformula}, built from the positivity-first entropy radius. When the direct entropy radius $\theta_j^{\mathrm E}$ is available on the whole ray, \eqref{eq:minformula-direct} recovers the more familiar three-way minimum. In either description, the final state still lies on the same original cell-average-anchored scaling ray.
\end{remark}

\section{The one-dimensional EPO theorem}\label{sec:main}

We are ready to define the final limited state and prove its properties.

\subsection{Definition of the EPO limiter}

Let $\mathbf U_j^\star=(\mathbf U_{j,1}^\star,\dots,\mathbf U_{j,L}^\star)$ denote the candidate nodal array in cell $I_j$, with average $\bar{\mathbf U}_j^\star$, and let $\theta_j^{\mathrm{EPO}}$ be given by \eqref{eq:minformula}. The EPO-limited nodal values are
\begin{equation}\label{eq:EPO-limited}
\mathbf U_{j,\nu}^{n+1}
=
\bar{\mathbf U}_j^\star
+
\theta_j^{\mathrm{EPO}}\bigl(\mathbf U_{j,\nu}^\star-\bar{\mathbf U}_j^\star\bigr),
\qquad \nu=1,\dots,L.
\end{equation}
Equivalently,
\[
\mathbf U_j^{n+1}
=
\mathcal S_{\bar{\mathbf U}_j^\star}\!\bigl(\theta_j^{\mathrm{EPO}};\mathbf U_j^\star\bigr).
\]

\begin{theorem}[Local EPO properties]\label{thm:localEPO}
Assume Assumptions~\ref{ass:weakP}, \ref{ass:weakE}, and \ref{ass:osc}. Then the EPO-limited nodal array \eqref{eq:EPO-limited} satisfies, in every cell $I_j$,
\begin{enumerate}[label=(\roman*),leftmargin=2em]
\item \textbf{mean preservation:}
\begin{equation}\label{eq:EPO-mean}
\sum_{\nu=1}^L \omega_\nu \mathbf U_{j,\nu}^{n+1}
=
\bar{\mathbf U}_j^\star;
\end{equation}
\item \textbf{invariant-set preservation:}
\begin{equation}\label{eq:EPO-P}
\mathbf U_{j,\nu}^{n+1}\in G,
\qquad \nu=1,\dots,L;
\end{equation}
\item \textbf{local strong entropy stability:}
\begin{equation}\label{eq:EPO-local-entropy}
\sum_{\nu=1}^L \omega_\nu \eta(\mathbf U_{j,\nu}^{n+1})
\le
B_j^n;
\end{equation}
\item \textbf{oscillation admissibility:}
\begin{equation}\label{eq:EPO-O}
\mathbf U_j^{n+1}\in \mathcal O_j^\star.
\end{equation}
\end{enumerate}
\end{theorem}

\begin{proof}
Property \eqref{eq:EPO-mean} follows immediately from \eqref{eq:mean-preserving-ray}, since the EPO limiter is a scaling about the cell average.

By Theorem~\ref{thm:min},
\[
\theta_j^{\mathrm{EPO}}\le \theta_j^{\mathrm{PE}}.
\]
Therefore \eqref{eq:thetaPE-property} immediately yields both \eqref{eq:EPO-P} and \eqref{eq:EPO-local-entropy}. Finally,
\[
\theta_j^{\mathrm{EPO}}\le \theta_j^{\mathrm O},
\]
and Proposition~\ref{prop:thetaO-basic} yields \eqref{eq:EPO-O}.
\end{proof}

\begin{remark}
Theorem~\ref{thm:localEPO} records the local consequences of the EPO construction. In particular, the limiter does not alter the candidate cell average, and hence it preserves conservation whenever the underlying candidate update is conservative. The theorem also shows that weak entropy stability of the average is sufficient to obtain a strong nodal entropy inequality by the same cell-average-anchored scaling mechanism used in positivity preservation.
\end{remark}

\subsection{Global strong entropy stability}

The local entropy budgets become global once they satisfy the compatibility condition in Assumption~\ref{ass:budget-compat}.

\begin{theorem}[Global strong entropy stability]\label{thm:global}
In addition to the hypotheses of Theorem~\ref{thm:localEPO}, assume Assumption~\ref{ass:budget-compat}. Then
\begin{equation}\label{eq:global-strong}
\sum_j \sum_{\nu=1}^L \omega_\nu \eta(\mathbf U_{j,\nu}^{n+1})
\le
\sum_j \sum_{\nu=1}^L \omega_\nu \eta(\mathbf U_{j,\nu}^{n})
-
\lambda \sum_j \bigl(\widehat{\mathcal Q}_{j+\frac12}^n-\widehat{\mathcal Q}_{j-\frac12}^n\bigr).
\end{equation}
In particular, under periodic boundary conditions the entropy flux term telescopes and
\begin{equation}\label{eq:global-strong-periodic}
\sum_j \sum_{\nu=1}^L \omega_\nu \eta(\mathbf U_{j,\nu}^{n+1})
\le
\sum_j \sum_{\nu=1}^L \omega_\nu \eta(\mathbf U_{j,\nu}^{n}).
\end{equation}
\end{theorem}

\begin{proof}
Summing the local bound \eqref{eq:EPO-local-entropy} over all cells gives
\[
\sum_j \sum_{\nu=1}^L \omega_\nu \eta(\mathbf U_{j,\nu}^{n+1})
\le
\sum_j B_j^n.
\]
Assumption~\ref{ass:budget-compat} then yields \eqref{eq:global-strong}. The periodic case follows by telescoping the interface entropy flux sum.
\end{proof}

\begin{remark}
The inequality \eqref{eq:global-strong} is quadrature based: the left-hand side is the discrete nodal entropy rather than the exact cell integral of $\eta(\mathbf U_h)$. This is the appropriate notion for a nodal limiter.
\end{remark}

\begin{corollary}[Several prescribed entropy pairs]\label{cor:multi-entropy}
Let $\{(\eta^{(r)},\mathcal Q^{(r)})\}_{r=1}^M$ be a finite family of convex entropy pairs with corresponding local entropy functionals $\mathcal E_j^{(r)}$ and local budgets $B_j^{n,(r)}$. Assume that, for every $r$, the hypotheses of Theorem~\ref{thm:localEPO} hold with $(\eta,\mathcal Q,\mathcal E_j,B_j^n)$ replaced by $(\eta^{(r)},\mathcal Q^{(r)},\mathcal E_j^{(r)},B_j^{n,(r)})$. Define the combined entropy radius by
\[
\theta_j^{\mathrm{PE,all}}:=\min_{1\le r\le M}\theta_j^{\mathrm{PE},(r)},
\qquad
\theta_j^{\mathrm{EPO,all}}:=\min\{\theta_j^{\mathrm{PE,all}},\theta_j^{\mathrm O}\}.
\]
Then the limited state
\[
\mathbf U_j^{n+1}
=
\mathcal S_{\bar{\mathbf U}_j^\star}\!\bigl(\theta_j^{\mathrm{EPO,all}};\mathbf U_j^\star\bigr)
\]
satisfies \eqref{eq:EPO-mean}, \eqref{eq:EPO-P}, and \eqref{eq:EPO-O}; moreover, for every $r=1,\dots,M$,
\[
\mathcal E_j^{(r)}(\mathbf U_j^{n+1})\le B_j^{n,(r)}.
\]
If, in addition, the budgets are globally compatible for the $r$th entropy pair, then the analogue of Theorem~\ref{thm:global} holds for that pair. In particular, \EPO{} yields fully discrete entropy stability for every entropy pair in the prescribed finite family.
\end{corollary}

\begin{proof}
For each $r$, the definition of $\theta_j^{\mathrm{PE},(r)}$ implies that every $0\le \theta\le \theta_j^{\mathrm{PE},(r)}$ satisfies the $r$th local entropy constraint together with nodal admissibility in $G$. Therefore every $0\le \theta\le \theta_j^{\mathrm{PE,all}}$ satisfies all prescribed entropy constraints simultaneously. Intersecting once more with the oscillation constraint gives the radius $\theta_j^{\mathrm{EPO,all}}$. The conclusions then follow from the same arguments as in Theorems~\ref{thm:localEPO} and \ref{thm:global}, applied separately to each entropy pair.
\end{proof}

\begin{remark}[Relation to classical entropy-stable discretizations]
The entropy part of EPO should be distinguished from the classical Tadmor/SBP line of work \cite{Tadmor1987,Tadmor2003,FisherCarpenter2013,Gassner2013,CarpenterFisherNielsenFrankel2014,SvardNordstrom2014,DelReyFernandezHickenZingg2014,ChenShu2017,ChanFernandezCarpenter2019}. In those constructions the entropy inequality is built into the discrete derivative, split form, or numerical flux. Here the starting point is different: a candidate FV/DG update is given, a weak entropy budget is verified for its updated cell average, and a local scaling then enforces a strong quadrature-based entropy inequality. The two viewpoints are complementary, and an entropy-stable base discretization may serve as the source of a weak or passive entropy input for EPO.
\end{remark}

\begin{remark}
If the auxiliary direct entropy radii are available on the whole candidate ray for all prescribed entropy pairs, then
\[
\theta_j^{\mathrm{PE,all}}
=
\min\bigl\{\theta_j^{\mathrm P},\theta_j^{\mathrm E,(1)},\dots,\theta_j^{\mathrm E,(M)}\bigr\}.
\]
Thus one obtains one entropy radius per entropy pair and then takes the minimum.
\end{remark}

\subsection{Conservation}

Although mean preservation is obvious from the scaling construction, it is worth recording its consequence for conservative schemes.

\begin{corollary}[Conservation inheritance]\label{cor:conservation}
Suppose the underlying candidate scheme updates cell averages conservatively, i.e.,
\begin{equation}\label{eq:candidate-conservative}
\bar{\mathbf U}_j^\star
=
\bar{\mathbf U}_j^n
-
\lambda\bigl(\widehat{\mathbf F}_{j+\frac12}^n-\widehat{\mathbf F}_{j-\frac12}^n\bigr)
\end{equation}
for some numerical fluxes $\widehat{\mathbf F}_{j\pm\frac12}^n$. Then the EPO-limited solution satisfies the same conservative cell-average update:
\begin{equation}\label{eq:EPO-conservative}
\bar{\mathbf U}_j^{n+1}
=
\bar{\mathbf U}_j^\star
=
\bar{\mathbf U}_j^n
-
\lambda\bigl(\widehat{\mathbf F}_{j+\frac12}^n-\widehat{\mathbf F}_{j-\frac12}^n\bigr).
\end{equation}
\end{corollary}

\begin{proof}
Equation \eqref{eq:EPO-mean} shows that the cell average of the limited nodal array equals the candidate average $\bar{\mathbf U}_j^\star$. Substituting \eqref{eq:candidate-conservative} completes the proof.
\end{proof}

\subsection{Interpretation}

EPO combines several local constraints on the same cell-average-anchored ray. The admissible-state constraint gives the radius $\theta_j^{\mathrm P}$. Each prescribed entropy pair gives one positivity-first entropy radius $\theta_j^{\mathrm{PE},(r)}$; for one entropy pair this is simply $\theta_j^{\mathrm{PE}}$. The oscillation set gives the radius $\theta_j^{\mathrm O}$. The final limiter uses the minimum of the relevant radii. In this sense, the method is a single radial projection rather than a sequence of unrelated corrections.

\section{Multistage and multistep high-order time discretizations}\label{sec:time}

This section discusses the EPO framework with high-order SSP time discretizations. 
 For classical explicit SSP Runge--Kutta methods, the following fully rigorous theory yields a \emph{stagewise} weak entropy budget and therefore a stagewise EPO realization. In practical implementations, it is often convenient to apply the positivity and entropy modules stage by stage while deferring the O-module to the end of the time step, since the O-module does not enter the stagewise entropy-budget recursion. If one wishes to retain the designed high-order temporal accuracy while applying the entropy module only once per step, a rigorous approach is to use an SSP multistep method, whose end-of-step candidate average is already a convex combination of forward-Euler building blocks.

\subsection{Stagewise formulation}

Let an explicit multistage method generate, for each stage $\ell$, candidate nodal values
\[
\mathbf U_{j,\nu}^{(\ell),\star},
\qquad
\nu=1,\dots,L,
\]
with candidate average
\[
\bar{\mathbf U}_j^{(\ell),\star}
=
\sum_{\nu=1}^L \omega_\nu \mathbf U_{j,\nu}^{(\ell),\star}.
\]
Assume that a stagewise admissible set $G$, stagewise budgets $B_j^{(\ell)}$, and stagewise convex oscillation-suppressing sets $\mathcal O_{j}^{(\ell),\star}$ are available. Then the entire EPO construction above applies at stage $\ell$ after replacing
\[
\mathbf U_j^\star,\ \bar{\mathbf U}_j^\star,\ B_j^n,\ \mathcal O_j^\star
\]
by
\[
\mathbf U_j^{(\ell),\star},\ \bar{\mathbf U}_j^{(\ell),\star},\ B_j^{(\ell)},\ \mathcal O_{j}^{(\ell),\star}.
\]

\begin{proposition}[Stagewise EPO]\label{prop:stagewise}
Assume that, at a given stage $\ell$, the candidate stage values satisfy the analogues of Assumptions~\ref{ass:weakP}, \ref{ass:weakE}, and \ref{ass:osc}. Then the EPO-limited stage values satisfy the analogues of Theorem~\ref{thm:localEPO}. If, in addition, the stage budgets satisfy the analogue of Assumption~\ref{ass:budget-compat}, then the analogue of Theorem~\ref{thm:global} holds at that stage.
\end{proposition}

\begin{proof}
The proof is identical to that of Theorems~\ref{thm:localEPO} and \ref{thm:global}; all arguments are cellwise and depend only on the current candidate stage values and budgets.
\end{proof}

\subsection{Weak entropy stability for explicit SSP Runge--Kutta methods}

Write the explicit $s$-stage SSP Runge--Kutta method in Shu--Osher form
\begin{equation}\label{eq:ssp-shu-osher}
\mathbf U^{(0)}=\mathbf U^n,
\qquad
\mathbf U^{(i),\star}
=
\sum_{k=0}^{i-1}\alpha_{ik}\Bigl(\mathbf U^{(k)}+\gamma_{ik}\Delta t\,\mathcal L(\mathbf U^{(k)})\Bigr),
\qquad i=1,\dots,s,
\end{equation}
where $\alpha_{ik}\ge0$, $\gamma_{ik}\ge0$, and $\sum_{k=0}^{i-1}\alpha_{ik}=1$ for each $i$. Here $\mathbf U^{(k)}$ denotes the post-limited stage state that serves as the input to later stages, while $\mathbf U^{(i),\star}$ is the new stage candidate before stage-$i$ EPO limiting. It is also convenient to introduce
\[
\beta_{ik}:=\alpha_{ik}\gamma_{ik},
\qquad
\mathcal C_{\mathrm{SSP}}
:=
\min_{\substack{1\le i\le s,\ 0\le k\le i-1\\ \beta_{ik}>0}}
\frac{\alpha_{ik}}{\beta_{ik}}
=
\min_{\substack{1\le i\le s,\ 0\le k\le i-1\\ \gamma_{ik}>0}}
\frac{1}{\gamma_{ik}},
\]
so that $\Delta t\le \mathcal C_{\mathrm{SSP}}\Delta t_{\mathrm{FE}}$ implies $\gamma_{ik}\Delta t\le \Delta t_{\mathrm{FE}}$ for every forward-Euler substep appearing in \eqref{eq:ssp-shu-osher}.

To state the stage budgets compactly, define for any stage state $\mathbf V$ the flux differences
\[
\Delta\widehat{\mathbf F}_j(\mathbf V)
:=
\widehat{\mathbf F}_{j+\frac12}(\mathbf V)-\widehat{\mathbf F}_{j-\frac12}(\mathbf V),
\qquad
\Delta\widehat{\mathcal Q}_j(\mathbf V)
:=
\widehat{\mathcal Q}_{j+\frac12}(\mathbf V)-\widehat{\mathcal Q}_{j-\frac12}(\mathbf V),
\]
and let
\begin{equation}\label{eq:fe-stage-average}
\bar{\mathbf U}_{j}^{\mathrm{FE}}(\mathbf V;\gamma)
:=
\bar{\mathbf V}_j-\gamma\lambda\,\Delta\widehat{\mathbf F}_j(\mathbf V),
\qquad \lambda:=\frac{\Delta t}{\Delta x},
\end{equation}
be the candidate cell average produced by a forward-Euler substep of effective size $\gamma\Delta t$ started from $\mathbf V$.

\begin{theorem}[Weak entropy property for explicit SSP Runge--Kutta methods]\label{thm:ssp-budget}
Assume that every forward-Euler substep of the form \eqref{eq:fe-stage-average} with $\gamma\Delta t\le \Delta t_{\mathrm{FE}}$ satisfies the weak entropy stability estimate
\begin{equation}\label{eq:fe-stage-entropy}
\eta\!\bigl(\bar{\mathbf U}_{j}^{\mathrm{FE}}(\mathbf V;\gamma)\bigr)
\le
\mathcal E_j(\mathbf V_j)-\gamma\lambda\,\Delta\widehat{\mathcal Q}_j(\mathbf V)
\end{equation}
for every cell $j$ and every admissible stage state $\mathbf V$. Assume also that the same forward-Euler substeps satisfy the weak geometric budget
\begin{equation}\label{eq:fe-stage-G}
\bar{\mathbf U}_{j}^{\mathrm{FE}}(\mathbf V;\gamma)\in G.
\end{equation}
If the SSP time step satisfies $\Delta t\le \mathcal C_{\mathrm{SSP}}\Delta t_{\mathrm{FE}}$, then every stage candidate average in \eqref{eq:ssp-shu-osher} satisfies the weak entropy stability
\begin{equation}\label{eq:ssp-stage-weakE}
\eta\!\bigl(\bar{\mathbf U}_j^{(i),\star}\bigr)
\le
B_j^{(i)},
\qquad i=1,\dots,s,
\end{equation}
with stagewise entropy budget
\begin{align}
B_j^{(i)}
&:=
\sum_{k=0}^{i-1}\alpha_{ik}
\Bigl(
\mathcal E_j(\mathbf U_j^{(k)})-\gamma_{ik}\lambda\,\Delta\widehat{\mathcal Q}_j(\mathbf U^{(k)})
\Bigr)
\label{eq:ssp-stage-budget-general}\\
&=
\sum_{k=0}^{i-1}\alpha_{ik}\mathcal E_j(\mathbf U_j^{(k)})
-\lambda\sum_{k=0}^{i-1}\beta_{ik}\,\Delta\widehat{\mathcal Q}_j(\mathbf U^{(k)}).
\notag
\end{align}
Moreover,
\begin{equation}\label{eq:ssp-stage-G}
\bar{\mathbf U}_j^{(i),\star}\in G
\qquad \text{for every stage }i.
\end{equation}
\end{theorem}

\begin{proof}
Fix a stage $i\in\{1,\dots,s\}$ and a cell $j$. For each $k\in\{0,\dots,i-1\}$ with $\alpha_{ik}>0$, define the forward-Euler substep average
\[
\bar{\mathbf W}_{j}^{(ik)}
:=
\bar{\mathbf U}_{j}^{\mathrm{FE}}\bigl(\mathbf U^{(k)};\gamma_{ik}\bigr)
=
\bar{\mathbf U}_{j}^{(k)}-\gamma_{ik}\lambda\,\Delta\widehat{\mathbf F}_j\bigl(\mathbf U^{(k)}\bigr).
\]
Because $\Delta t\le \mathcal C_{\mathrm{SSP}}\Delta t_{\mathrm{FE}}$, every effective substep size satisfies $\gamma_{ik}\Delta t\le \Delta t_{\mathrm{FE}}$. Therefore the weak entropy estimate \eqref{eq:fe-stage-entropy} and the weak geometric budget \eqref{eq:fe-stage-G} are valid for every $\bar{\mathbf W}_{j}^{(ik)}$.

Next, take the cell average of the stage formula \eqref{eq:ssp-shu-osher}. Since averaging is linear and the spatial operator enters conservatively at the level of cell averages, one obtains
\[
\bar{\mathbf U}_j^{(i),\star}
=
\sum_{k=0}^{i-1}\alpha_{ik}
\Bigl(
\bar{\mathbf U}_j^{(k)}-\gamma_{ik}\lambda\,\Delta\widehat{\mathbf F}_j(\mathbf U^{(k)})
\Bigr)
=
\sum_{k=0}^{i-1}\alpha_{ik}\,\bar{\mathbf W}_{j}^{(ik)}.
\]
Because the coefficients $\alpha_{ik}$ are nonnegative and sum to one, this is a convex combination. Convexity of $\eta$ therefore gives
\[
\eta\!\bigl(\bar{\mathbf U}_j^{(i),\star}\bigr)
\le
\sum_{k=0}^{i-1}\alpha_{ik}\,\eta\!\bigl(\bar{\mathbf W}_{j}^{(ik)}\bigr).
\]
Applying \eqref{eq:fe-stage-entropy} to each term on the right-hand side yields
\[
\eta\!\bigl(\bar{\mathbf U}_j^{(i),\star}\bigr)
\le
\sum_{k=0}^{i-1}\alpha_{ik}
\Bigl(
\mathcal E_j(\mathbf U_j^{(k)})-\gamma_{ik}\lambda\,\Delta\widehat{\mathcal Q}_j(\mathbf U^{(k)})
\Bigr)
=
B_j^{(i)},
\]
which is exactly \eqref{eq:ssp-stage-weakE}. Rewriting $\alpha_{ik}\gamma_{ik}$ as $\beta_{ik}$ gives the second form in \eqref{eq:ssp-stage-budget-general}.

Finally, \eqref{eq:fe-stage-G} gives $\bar{\mathbf W}_{j}^{(ik)}\in G$ for every $k$. Since $\bar{\mathbf U}_j^{(i),\star}$ is a convex combination of these states and $G$ is convex, \eqref{eq:ssp-stage-G} follows.
\end{proof}

Theorem~\ref{thm:ssp-budget} shows that each SSP stage has its own local entropy budget $B_j^{(i)}$, inherited from the forward-Euler building blocks by convexity. One then feeds $B_j^{(i)}$ into Proposition~\ref{prop:stagewise} and applies the local modules stage by stage. If several entropy pairs are prescribed, one repeats \eqref{eq:ssp-stage-budget-general} for each pair, computes the corresponding stagewise radii $\theta_j^{\mathrm{PE},(r)}$, and uses their minimum at that stage.

\begin{example}[The classical third-order SSPRK scheme]\label{ex:ssprk33-budgets}
For the standard three-stage third-order SSPRK method,
\begin{align}
\mathbf U^{(1),\star}
&=
\mathbf U^{(0)}+\Delta t\,\mathcal L(\mathbf U^{(0)}),
\label{eq:ssprk33-1}\\
\mathbf U^{(2),\star}
&=
\frac34\mathbf U^{(0)}+\frac14\Bigl(\mathbf U^{(1)}+\Delta t\,\mathcal L(\mathbf U^{(1)})\Bigr),
\label{eq:ssprk33-2}\\
\mathbf U^{(3),\star}
&=
\frac13\mathbf U^{(0)}+\frac23\Bigl(\mathbf U^{(2)}+\Delta t\,\mathcal L(\mathbf U^{(2)})\Bigr),
\label{eq:ssprk33-3}
\end{align}
where $\mathbf U^{(0)}=\mathbf U^n$ and $\mathbf U^{(1)}$, $\mathbf U^{(2)}$ are the post-limited stage states that feed the next stage. The corresponding stagewise weak entropy budgets from \eqref{eq:ssp-stage-budget-general} are
\begin{align}
B_j^{(1)}
&=
\mathcal E_j(\mathbf U_j^{(0)})-\lambda\,\Delta\widehat{\mathcal Q}_j(\mathbf U^{(0)}),
\label{eq:ssprk33-B1}\\
B_j^{(2)}
&=
\frac34\,\mathcal E_j(\mathbf U_j^{(0)})
+\frac14\Bigl(\mathcal E_j(\mathbf U_j^{(1)})-\lambda\,\Delta\widehat{\mathcal Q}_j(\mathbf U^{(1)})\Bigr),
\label{eq:ssprk33-B2}\\
B_j^{(3)}
&=
\frac13\,\mathcal E_j(\mathbf U_j^{(0)})
+\frac23\Bigl(\mathcal E_j(\mathbf U_j^{(2)})-\lambda\,\Delta\widehat{\mathcal Q}_j(\mathbf U^{(2)})\Bigr).
\label{eq:ssprk33-B3}
\end{align}
These are exactly the stagewise budgets that should be passed to the entropy module at stages $1$, $2$, and $3$. After the stage-$i$ EPO projection one obtains the strong nodal entropy statement
\[
\mathcal E_j(\mathbf U_j^{(i)})\le B_j^{(i)},
\]
which then becomes part of the input data for the next SSP stage.
\end{example}

\subsection{Limitations of SSPRK}

The fully rigorous SSPRK realization is therefore the following one: at every Runge--Kutta stage, first form the stage candidate average, then apply positivity and entropy limiting using the stage budget $B_j^{(i)}$, and finally, if desired, apply the O-module either at that same stage or after the last stage. Two points deserve to be stated explicitly.

First, \emph{one-step entropy limiter is not justified by the SSPRK proof above}. The reason is visible already in the third-order example \eqref{eq:ssprk33-B3}: the final-stage weak entropy budget depends on $\mathcal E_j(\mathbf U_j^{(2)})$. If the entropy module is not enforced at stage $2$, then this quantity is not controlled by a certified strong entropy inequality, and the recursion of stage budgets does not close. Thus the current theory proves stagewise SSPRK-EPO, but it does not prove that one may run a classical SSPRK method unchanged and apply the entropy limiter only once at the end of the time step.

Second, because the entropy module is inserted after every intermediate SSPRK stage, the resulting fully limited method should be regarded as only \emph{first-order in time once the entropy limiter is active}.   But as a theorem-level statement for the nonlinear limited scheme, the stagewise SSPRK realization is only a certified forward-Euler-type weak-to-strong closure at each stage.

\begin{proposition}[Stagewise P/E and endpoint-only O]\label{prop:stagewise-PE-endO}
Assume that, at every SSPRK stage, one applies the positivity and entropy modules using the stage budget $B_j^{(i)}$, thereby producing stage states that satisfy the analogue of Theorem~\ref{thm:localEPO} without the oscillation constraint. After the final stage $s$, let $\mathbf U_j^{(s),\mathrm{PE}}$ denote the resulting positivity--entropy limited state, and then apply one additional oscillation scaling
\[
\mathbf U_j^{n+1}
:=
\mathcal S_{\bar{\mathbf U}_j^{(s),\star}}\!\bigl(\theta_j^{\mathrm O,\mathrm{end}};\mathbf U_j^{(s),\mathrm{PE}}\bigr),
\qquad 0\le \theta_j^{\mathrm O,\mathrm{end}}\le 1,
\]
with any admissible endpoint oscillation radius $\theta_j^{\mathrm O,\mathrm{end}}$. Then the final state still preserves the stage-$s$ cell average, remains nodally in $G$, satisfies the strong entropy inequality
\[
\mathcal E_j(\mathbf U_j^{n+1})\le B_j^{(s)},
\]
and belongs to the endpoint oscillation set.
\end{proposition}

\begin{proof}
The stagewise P/E hypothesis gives $\mathbf U_j^{(s),\mathrm{PE}}\in G^L$ together with $\mathcal E_j(\mathbf U_j^{(s),\mathrm{PE}})\le B_j^{(s)}$. Since the entire segment from the constant state $\bar{\mathbf U}_j^{(s),\star}$ to $\mathbf U_j^{(s),\mathrm{PE}}$ lies in $G\subset D_\eta$, Proposition~\ref{prop:Psi} applies to this endpoint candidate. Hence any further scaling toward the constant state can only decrease the quadrature entropy. Nodal membership in $G$ follows by convexity, and oscillation admissibility is built into the choice of $\theta_j^{\mathrm O,\mathrm{end}}$.
\end{proof}

Proposition~\ref{prop:stagewise-PE-endO} resolves the apparent tension between the theorem-level stagewise SSPRK analysis and the common implementation habit of applying the O-module only at the end of the RK step. The P- and E-modules are the pieces that enter the stage-budget recursion; the O-module may be postponed because a final extra scaling cannot destroy the already-established positivity and entropy bounds.

\subsection{An end-of-step high-order realization via SSP multistep methods}

To keep the designed high-order time accuracy while using the entropy module only once per step, a cleaner rigorous realization is to use an SSP multistep method:
\begin{equation}\label{eq:ssp-multistep}
\mathbf U^{n+1,\star}
=
\sum_{r=0}^{k-1}\alpha_r\Bigl(\mathbf U^{n-r}+\gamma_r\Delta t\,\mathcal L(\mathbf U^{n-r})\Bigr),
\qquad
\alpha_r\ge 0,\ \gamma_r\ge 0,\ \sum_{r=0}^{k-1}\alpha_r=1.
\end{equation}
Set
\[
\beta_r:=\alpha_r\gamma_r,
\qquad
\mathcal C_{\mathrm{MS}}
:=
\min_{\substack{0\le r\le k-1\\ \gamma_r>0}}\frac1{\gamma_r}.
\]
If $\Delta t\le \mathcal C_{\mathrm{MS}}\Delta t_{\mathrm{FE}}$, then every building block in \eqref{eq:ssp-multistep} is an admissible forward-Euler step of effective size $\gamma_r\Delta t$.

\begin{theorem}[Weak entropy stability for SSP multistep cell averages]\label{thm:sspms-budget}
Assume that the forward-Euler substeps satisfy \eqref{eq:fe-stage-entropy} and \eqref{eq:fe-stage-G} whenever their effective time step does not exceed $\Delta t_{\mathrm{FE}}$. If the SSP multistep formula \eqref{eq:ssp-multistep} is used with $\Delta t\le \mathcal C_{\mathrm{MS}}\Delta t_{\mathrm{FE}}$, then the end-of-step candidate average satisfies
\begin{equation}\label{eq:sspms-weakE}
\eta\!\bigl(\bar{\mathbf U}_j^{n+1,\star}\bigr)
\le
B_j^{n+1,\mathrm{MS}},
\end{equation}
with the end-of-step entropy budget
\begin{equation}\label{eq:sspms-budget}
B_j^{n+1,\mathrm{MS}}
:=
\sum_{r=0}^{k-1}\alpha_r\Bigl(\mathcal E_j(\mathbf U_j^{n-r})-\gamma_r\lambda\,\Delta\widehat{\mathcal Q}_j(\mathbf U^{n-r})\Bigr),
\end{equation}
and also satisfies the weak geometric budget
\begin{equation}\label{eq:sspms-weakP}
\bar{\mathbf U}_j^{n+1,\star}\in G.
\end{equation}
\end{theorem}

\begin{proof}
For each $r$, define the forward-Euler candidate average
\[
\bar{\mathbf W}_j^{(r)}
:=
\bar{\mathbf U}_j^{n-r}-\gamma_r\lambda\,\Delta\widehat{\mathbf F}_j(\mathbf U^{n-r}).
\]
Because $\Delta t\le \mathcal C_{\mathrm{MS}}\Delta t_{\mathrm{FE}}$, every effective substep size satisfies $\gamma_r\Delta t\le \Delta t_{\mathrm{FE}}$. Hence \eqref{eq:fe-stage-entropy} and \eqref{eq:fe-stage-G} apply to each $\bar{\mathbf W}_j^{(r)}$. Taking cell averages of \eqref{eq:ssp-multistep} gives
\[
\bar{\mathbf U}_j^{n+1,\star}
=
\sum_{r=0}^{k-1}\alpha_r\bar{\mathbf W}_j^{(r)}.
\]
Convexity of $\eta$ therefore yields
\[
\eta\!\bigl(\bar{\mathbf U}_j^{n+1,\star}\bigr)
\le
\sum_{r=0}^{k-1}\alpha_r\,\eta\!\bigl(\bar{\mathbf W}_j^{(r)}\bigr)
\le
\sum_{r=0}^{k-1}\alpha_r\Bigl(\mathcal E_j(\mathbf U_j^{n-r})-\gamma_r\lambda\,\Delta\widehat{\mathcal Q}_j(\mathbf U^{n-r})\Bigr),
\]
which is \eqref{eq:sspms-weakE}. The weak geometric budget \eqref{eq:sspms-weakP} follows in the same way from convexity of $G$.
\end{proof}

\begin{corollary}[One-shot end-of-step EPO for SSP multistep methods]\label{cor:sspms-epo}
Apply one positivity-first EPO correction to the end-of-step candidate state of the SSP multistep method \eqref{eq:ssp-multistep}, using the budget $B_j^{n+1,\mathrm{MS}}$ from \eqref{eq:sspms-budget}. Then the conclusions of Theorems~\ref{thm:localEPO} and \ref{thm:global} hold at time level $n+1$. Because the entropy module is applied only once per step, this SSP-multistep realization is the rigorous route for retaining the designed high-order temporal accuracy while enforcing the EPO constraints. For several prescribed entropy pairs, one computes the end-of-step budget and the corresponding entropy radius for each pair and then takes the minimum.
\end{corollary}

\begin{proof}
Theorem~\ref{thm:sspms-budget} supplies exactly the weak geometric and weak entropy inputs needed by Theorem~\ref{thm:localEPO}. If the budgets are globally compatible, Theorem~\ref{thm:global} applies as well.
\end{proof}

\begin{example}[A third-order SSP multistep realization]\label{ex:sspms-third}
The third-order SSP multistep formula quoted in \cite[Section~2.3]{ZhangXiaShu2012} can be written as
\begin{equation}\label{eq:sspms-third}
\mathbf U^{n+1,\star}
=
\frac{16}{27}\Bigl(\mathbf U^n+3\Delta t\,\mathcal L(\mathbf U^n)\Bigr)
+
\frac{11}{27}\Bigl(\mathbf U^{n-3}+\frac{12}{11}\Delta t\,\mathcal L(\mathbf U^{n-3})\Bigr).
\end{equation}
Here $\alpha_0=16/27$, $\gamma_0=3$, $\alpha_3=11/27$, and $\gamma_3=12/11$, so $\mathcal C_{\mathrm{MS}}=1/3$. The corresponding end-of-step entropy budget is
\begin{equation}\label{eq:sspms-third-budget}
B_j^{n+1,\mathrm{MS}}
=
\frac{16}{27}\Bigl(\mathcal E_j(\mathbf U_j^{n})-3\lambda\,\Delta\widehat{\mathcal Q}_j(\mathbf U^{n})\Bigr)
+
\frac{11}{27}\Bigl(\mathcal E_j(\mathbf U_j^{n-3})-\frac{12}{11}\lambda\,\Delta\widehat{\mathcal Q}_j(\mathbf U^{n-3})\Bigr).
\end{equation}
This is the budget passed to the \emph{single} end-of-step entropy module. In contrast with stagewise SSPRK, no intermediate entropy corrections are required.
\end{example}

The two realizations serve different purposes. Stagewise SSPRK inherits forward-Euler weak budgets directly, but once the entropy module is active stage by stage the resulting limited method can only be certified at first order in time. SSP multistep methods provide an end-of-step weak entropy budget and therefore allow a single end-of-step entropy correction while retaining the designed temporal order. For several prescribed entropy pairs, one computes the corresponding budget and entropy radius for each pair and then takes the minimum.
\section{Two-dimensional EPO on rectangular and triangular meshes}\label{sec:2d}

This section extends EPO to two-dimensional rectangular and triangular meshes. The cellwise scaling geometry from Sections~\ref{sec:oscillation}--\ref{sec:main} is unchanged; what depends on the mesh is the construction of positive quadratures and of weak geometric and weak entropy budgets for the updated cell averages. For background on rectangular and triangular decompositions we use the positivity literature \cite{ZhangShuPP2010,ZhangXiaShu2012,WuZhangShuSurvey2025}, the optimal Cartesian decompositions of Cui, Ding, and Wu \cite{CuiDingWuMulti2023,CuiDingWuOCAD2024}, the suitable-quadrature entropy framework of Chen and Shu \cite{ChenShu2017}, and the recent triangular decomposition work of Ding, Cui, and Wu \cite{DingCuiWu2025}.

In two dimensions, the cellwise EPO geometry remains unchanged; only the construction of weak budgets depends on the mesh and quadrature. On rectangular meshes, both weak positivity and weak entropy can be obtained actively from line-separable special quadratures. On triangular meshes, positivity is supplied by feasible-quadrature/GQL arguments, while entropy may be obtained either actively from local CAD-based weak budgets or passively from an already entropy-stable candidate discretization. Once these budgets are available, the final EPO step is the same local scaling along the cell-average-anchored ray.

The one-dimensional proofs developed earlier truly use only four ingredients: positive quadrature weights, a cell-average-anchored scaling ray, weak budgets for the updated averages, and convex admissible or oscillation sets. The two-dimensional theory below does not introduce a new limiting geometry; it only supplies these four ingredients on rectangles and triangles by mesh-specific decomposition arguments.

\subsection{Cellwise EPO geometry on a generic two-dimensional cell}

Let $K\in\mathcal T_h$ be a cell of a two-dimensional mesh and let
\[
S_K=\{\mathbf x_{K,\nu}\}_{\nu=1}^{N_K}\subset \overline K
\]
be a finite node set equipped with positive weights
\[
\varpi_{K,\nu}>0,
\qquad
\sum_{\nu=1}^{N_K}\varpi_{K,\nu}=1,
\]
chosen so that the cell average of the local candidate representation is exactly given by
\begin{equation}\label{eq:2d-cellavg}
\bar{\mathbf U}_K
=
\sum_{\nu=1}^{N_K}\varpi_{K,\nu}\,\mathbf U_{K,\nu}.
\end{equation}
The associated local quadrature entropy is
\begin{equation}\label{eq:2d-local-entropy}
\mathcal E_K(\mathbf V)
:=
\sum_{\nu=1}^{N_K}\varpi_{K,\nu}\,\eta(V_{\nu}),
\qquad
\mathbf V=(V_1,\dots,V_{N_K})\in (\mathbb R^m)^{N_K},
\end{equation}
and the constant nodal state generated by a mean value $\bar{\mathbf U}$ is
\[
\mathbf C_K(\bar{\mathbf U})
:=
(\bar{\mathbf U},\dots,\bar{\mathbf U})\in (\mathbb R^m)^{N_K}.
\]
For a candidate array $\mathbf U_K^\star$ with average $\bar{\mathbf U}_K^\star$, define the cell-average-anchored scaling ray
\begin{equation}\label{eq:2d-scaling}
\mathcal S_{K,\bar{\mathbf U}_K^\star}(\theta;\mathbf U_K^\star)
:=
\bar{\mathbf U}_K^\star
+
\theta\bigl(\mathbf U_K^\star-\bar{\mathbf U}_K^\star\bigr),
\qquad 0\le \theta\le 1.
\end{equation}
All local geometric arguments from Sections~\ref{sec:oscillation}--\ref{sec:main} depend only on positivity of the quadrature weights, mean preservation, and convexity of the underlying constraints. Consequently they extend immediately from the one-dimensional index $j$ to an arbitrary two-dimensional cell $K$.

\begin{proposition}[Cellwise dimensional lifting]\label{prop:2d-lift}
Fix a cell $K$ together with positive weights $\{\varpi_{K,\nu}\}_{\nu=1}^{N_K}$ satisfying \eqref{eq:2d-cellavg}. Replace in Sections~\ref{sec:oscillation}--\ref{sec:main} the one-dimensional cell index $j$ by $K$, the weights $\omega_\nu$ by $\varpi_{K,\nu}$, and the nodal array $\mathbf U_j^\star$ by $\mathbf U_K^\star\in (\mathbb R^m)^{N_K}$. Then the cellwise analogues of Proposition~\ref{prop:Psi}, Proposition~\ref{prop:thetaP}, Proposition~\ref{prop:thetaE-after-P}, Proposition~\ref{prop:thetaO-basic}, Theorem~\ref{thm:A-star}, Theorem~\ref{thm:min}, and Theorem~\ref{thm:localEPO} remain valid verbatim.
\end{proposition}

\begin{proof}
Every proof in Sections~\ref{sec:oscillation}--\ref{sec:main} is cellwise and uses only three facts: the mean constraint is affine, the quadrature weights are positive and sum to one, and the admissibility sets or admissibility functionals are convex. None of those arguments depends on the spatial dimension or on the geometry of the cell. Replacing the index $j$ by $K$ and the weights $\omega_\nu$ by $\varpi_{K,\nu}$ therefore changes only the notation.
\end{proof}

An important consequence of Proposition~\ref{prop:2d-lift} is that entropy can also be inherited from an already entropy-stable candidate scheme.

\begin{proposition}[Entropy inheritance under further radial scaling]\label{prop:2d-passive-entropy}
Let $K$ be a cell, let $\mathbf U_K^\star\in (\mathbb R^m)^{N_K}$ have average $\bar{\mathbf U}_K^\star\in G$, and suppose the segment from $\bar{\mathbf U}_K^\star$ to $\mathbf U_K^\star$ stays inside $D_\eta$. If
\begin{equation}\label{eq:2d-passive-budget}
\mathcal E_K(\mathbf U_K^\star)\le B_K,
\end{equation}
then for every $\theta\in[0,1]$,
\begin{equation}\label{eq:2d-passive-budget-scaled}
\mathcal E_K\!\bigl(\mathcal S_{K,\bar{\mathbf U}_K^\star}(\theta;\mathbf U_K^\star)\bigr)
\le B_K.
\end{equation}
In particular, if a candidate two-dimensional DG/FV stage is already quadrature-entropy stable, then any additional positivity or COS scaling along the same cell-average-anchored ray preserves that entropy inequality.
\end{proposition}

\begin{proof}
Define
\[
\Psi_K(\theta)
:=
\mathcal E_K\!\bigl(\mathcal S_{K,\bar{\mathbf U}_K^\star}(\theta;\mathbf U_K^\star)\bigr).
\]
By Proposition~\ref{prop:2d-lift}, the proof of Proposition~\ref{prop:Psi} applies verbatim with the weights $\varpi_{K,\nu}$ in place of $\omega_\nu$, and therefore $\Psi_K$ is nondecreasing on $[0,1]$. Since \eqref{eq:2d-passive-budget} is exactly $\Psi_K(1)\le B_K$, it follows that $\Psi_K(\theta)\le \Psi_K(1)\le B_K$ for every $0\le \theta\le 1$.
\end{proof}

\subsection{Rectangular meshes: line-separable weak budgets}\label{sec:2d-rect}

We begin with Cartesian meshes, because there the special quadratures of Zhang--Shu type admit a line-separable interpretation that is perfectly aligned with the one-dimensional active-entropy theory developed earlier.

Consider the two-dimensional hyperbolic system
\begin{equation}\label{eq:2d-pde}
\partial_t \mathbf U + \partial_x \mathbf F_1(\mathbf U)+\partial_y \mathbf F_2(\mathbf U)=0
\end{equation}
on a rectangular mesh with cells
\[
I_{ij}=[x_{i-\frac12},x_{i+\frac12}]\times [y_{j-\frac12},y_{j+\frac12}],
\qquad
\Delta x = x_{i+\frac12}-x_{i-\frac12},
\quad
\Delta y = y_{j+\frac12}-y_{j-\frac12}.
\]
Let $p_{ij}^n(x,y)$ denote the local polynomial (or reconstructed local representation) at time level $n$. The cell average is denoted by
\[
\bar{\mathbf U}_{ij}^n = \frac{1}{\Delta x\Delta y}\int_{I_{ij}} p_{ij}^n(x,y)\,dx\,dy.
\]
Fix a $Q$-point Gauss quadrature in each coordinate direction with nodes $\{\widetilde x_i^{(q)}\}_{q=1}^Q$, $\{\widetilde y_j^{(q)}\}_{q=1}^Q$ and normalized weights $\{\widehat\omega_q\}_{q=1}^Q$. Fix also an $L$-point Gauss--Lobatto quadrature with nodes $\{x_i^{(\mu)}\}_{\mu=1}^L$, $\{y_j^{(\mu)}\}_{\mu=1}^L$ and positive weights $\{\omega_\mu\}_{\mu=1}^L$, where $\omega_1=\omega_L>0$.

We consider a candidate cell-average update of the form
\begin{equation}\label{eq:rect-update}
\bar{\mathbf U}_{ij}^\star
=
\bar{\mathbf U}_{ij}^n
-\lambda_x\sum_{q=1}^Q \widehat\omega_q
\Bigl(\widehat{\mathbf F}_{1,i+\frac12,q}^n-\widehat{\mathbf F}_{1,i-\frac12,q}^n\Bigr)
-\lambda_y\sum_{q=1}^Q \widehat\omega_q
\Bigl(\widehat{\mathbf F}_{2,q,j+\frac12}^n-\widehat{\mathbf F}_{2,q,j-\frac12}^n\Bigr),
\end{equation}
where
\[
\lambda_x:=\frac{\Delta t}{\Delta x},
\qquad
\lambda_y:=\frac{\Delta t}{\Delta y},
\]
and the directional numerical fluxes (e.g., LF flux) are evaluated at edge Gauss nodes,
\[
\widehat{\mathbf F}_{1,i+\frac12,q}^n
:=
\widehat{\mathbf F}_1\bigl( p_{ij}^n(x_{i+\frac12},\widetilde y_j^{(q)}),\ p_{i+1,j}^n(x_{i+\frac12},\widetilde y_j^{(q)})\bigr),
\]
\[
\widehat{\mathbf F}_{2,q,j+\frac12}^n
:=
\widehat{\mathbf F}_2\bigl( p_{ij}^n(\widetilde x_i^{(q)},y_{j+\frac12}),\ p_{i,j+1}^n(\widetilde x_i^{(q)},y_{j+\frac12})\bigr).
\]

The special quadrature used by Zhang and Shu \cite{ZhangShuPP2010} on rectangles can be written in the line-separable form
\begin{equation}\label{eq:rect-cad}
\bar{\mathbf U}_{ij}^n
=
\kappa_1\sum_{q=1}^Q \widehat\omega_q \sum_{\mu=1}^L \omega_\mu\,
 p_{ij}^n(x_i^{(\mu)},\widetilde y_j^{(q)})
+
\kappa_2\sum_{q=1}^Q \widehat\omega_q \sum_{\mu=1}^L \omega_\mu\,
 p_{ij}^n(\widetilde x_i^{(q)},y_j^{(\mu)}),
\end{equation}
where $\kappa_1,\kappa_2>0$ and $\kappa_1+\kappa_2=1$. The classical positivity proof uses
\begin{equation}\label{eq:kappa-classic}
\kappa_1 = \frac{\alpha_1/\Delta x}{\alpha_1/\Delta x+\alpha_2/\Delta y},
\qquad
\kappa_2 = \frac{\alpha_2/\Delta y}{\alpha_1/\Delta x+\alpha_2/\Delta y},
\end{equation}
where $\alpha_1$ and $\alpha_2$ are bounds on the directional characteristic speeds in the $x$- and $y$-directions.

Define the one-dimensional slice averages
\begin{equation}\label{eq:rect-slice-avgs}
\bar{\mathbf U}_{ij,q}^{x,n}
:=
\sum_{\mu=1}^L \omega_\mu\,p_{ij}^n(x_i^{(\mu)},\widetilde y_j^{(q)}),
\qquad
\bar{\mathbf U}_{ij,q}^{y,n}
:=
\sum_{\mu=1}^L \omega_\mu\,p_{ij}^n(\widetilde x_i^{(q)},y_j^{(\mu)}).
\end{equation}
Then \eqref{eq:rect-cad} becomes
\[
\bar{\mathbf U}_{ij}^n
=
\kappa_1\sum_{q=1}^Q\widehat\omega_q\bar{\mathbf U}_{ij,q}^{x,n}
+
\kappa_2\sum_{q=1}^Q\widehat\omega_q\bar{\mathbf U}_{ij,q}^{y,n}.
\]
Introduce the directional candidate slice averages
\begin{equation}\label{eq:rect-slice-candidates}
\mathbf A_{ij,q}^{x,\star}
:=
\bar{\mathbf U}_{ij,q}^{x,n}
-
\frac{\lambda_x}{\kappa_1}
\Bigl(\widehat{\mathbf F}_{1,i+\frac12,q}^n-\widehat{\mathbf F}_{1,i-\frac12,q}^n\Bigr),
\end{equation}
\begin{equation}\label{eq:rect-slice-candidates-y}
\mathbf A_{ij,q}^{y,\star}
:=
\bar{\mathbf U}_{ij,q}^{y,n}
-
\frac{\lambda_y}{\kappa_2}
\Bigl(\widehat{\mathbf F}_{2,q,j+\frac12}^n-\widehat{\mathbf F}_{2,q,j-\frac12}^n\Bigr).
\end{equation}
Substituting \eqref{eq:rect-cad} into \eqref{eq:rect-update} immediately yields the exact convex representation
\begin{equation}\label{eq:rect-star-decomp}
\bar{\mathbf U}_{ij}^\star
=
\kappa_1\sum_{q=1}^Q \widehat\omega_q\,\mathbf A_{ij,q}^{x,\star}
+
\kappa_2\sum_{q=1}^Q \widehat\omega_q\,\mathbf A_{ij,q}^{y,\star}.
\end{equation}
The set of quadrature nodes entering \eqref{eq:rect-cad} is
\begin{equation}\label{eq:rect-node-set}
S_{ij}^{\mathrm R}
:=
\bigl(S_i^x\times \widetilde S_j^y\bigr)
\cup
\bigl(\widetilde S_i^x\times S_j^y\bigr),
\end{equation}
where $S_i^x:=\{x_i^{(\mu)}\}_{\mu=1}^L$, $S_j^y:=\{y_j^{(\mu)}\}_{\mu=1}^L$, $\widetilde S_i^x:=\{\widetilde x_i^{(q)}\}_{q=1}^Q$, and $\widetilde S_j^y:=\{\widetilde y_j^{(q)}\}_{q=1}^Q$.

\begin{theorem}[Weak geometric budget on rectangular meshes]\label{thm:rect-weakP}
Assume that the directional fluxes $\widehat{\mathbf F}_1$ and $\widehat{\mathbf F}_2$ are one-dimensional IDP fluxes with CFL number $c_0$ for the normal systems
\[
\partial_t \mathbf U + \partial_x\mathbf F_1(\mathbf U)=0,
\qquad
\partial_t \mathbf U + \partial_y\mathbf F_2(\mathbf U)=0.
\]
Assume further that all values of $p_{ij}^n$ at the nodes of $S_{ij}^{\mathrm R}$ belong to $G$. If
\begin{equation}\label{eq:rect-cfl-general}
\frac{\alpha_1\Delta t}{\kappa_1\Delta x}\le \omega_1 c_0,
\qquad
\frac{\alpha_2\Delta t}{\kappa_2\Delta y}\le \omega_1 c_0,
\end{equation}
then $\mathbf A_{ij,q}^{x,\star}\in G$ and $\mathbf A_{ij,q}^{y,\star}\in G$ for every $q$, and therefore
\begin{equation}\label{eq:rect-weakP-result}
\bar{\mathbf U}_{ij}^\star\in G.
\end{equation}
For the classical choice \eqref{eq:kappa-classic}, the CFL condition \eqref{eq:rect-cfl-general} becomes
\begin{equation}\label{eq:rect-cfl-classic}
\left(\frac{\alpha_1}{\Delta x}+\frac{\alpha_2}{\Delta y}\right)\Delta t
\le
\omega_1 c_0.
\end{equation}
\end{theorem}

\begin{proof}
See \cite{ZhangShuPP2010,WuZhangShuSurvey2025}.
\end{proof}

The same line-separable decomposition also yields a weak entropy stability, expressed as a weak entropy budget obtained by averaging one-dimensional slice budgets.

\begin{theorem}[Weak entropy stability on rectangular meshes]\label{thm:rect-weakE}
In the setting of Theorem~\ref{thm:rect-weakP}, assume in addition that for each fixed $q$ the directional slice candidates satisfy one-dimensional entropy budgets of the form
\begin{equation}\label{eq:rect-slice-budgets}
\eta(\mathbf A_{ij,q}^{x,\star})\le B_{ij,q}^{x},
\qquad
\eta(\mathbf A_{ij,q}^{y,\star})\le B_{ij,q}^{y}.
\end{equation}
Then the updated two-dimensional cell average satisfies
\begin{equation}\label{eq:rect-weakE-result}
\eta(\bar{\mathbf U}_{ij}^\star)
\le
B_{ij}^{\mathrm R}
:=
\kappa_1\sum_{q=1}^Q \widehat\omega_q B_{ij,q}^{x}
+
\kappa_2\sum_{q=1}^Q \widehat\omega_q B_{ij,q}^{y}.
\end{equation}

In particular, let $\mathcal Q(\mathbf U)=(\mathcal Q_1(\mathbf U),\mathcal Q_2(\mathbf U))$ be the entropy-flux vector associated with $\eta$, and define the edge states
\[
\mathbf U_{i+\frac12,q}^{-,n}:=p_{ij}^n(x_{i+\frac12},\widetilde y_j^{(q)}),
\qquad
\mathbf U_{i+\frac12,q}^{+,n}:=p_{i+1,j}^n(x_{i+\frac12},\widetilde y_j^{(q)}),
\]
\[
\mathbf U_{q,j+\frac12}^{-,n}:=p_{ij}^n(\widetilde x_i^{(q)},y_{j+\frac12}),
\qquad
\mathbf U_{q,j+\frac12}^{+,n}:=p_{i,j+1}^n(\widetilde x_i^{(q)},y_{j+\frac12}).
\]
If the slice budgets are furnished by the one-dimensional canonical mechanism of Proposition~\ref{prop:canonical-budget} applied to the normal systems in the $x$- and $y$-directions with effective time-step ratios $\lambda_x/\kappa_1$ and $\lambda_y/\kappa_2$, then one may take
\begin{align}
B_{ij,q}^{x}
&:=
\sum_{\mu=1}^L \omega_\mu\,
\eta\!\bigl(p_{ij}^n(x_i^{(\mu)},\widetilde y_j^{(q)})\bigr)
-\frac{\lambda_x}{\kappa_1}
\Bigl(
\widehat{\mathcal Q}_{1,i+\frac12,q}^n-\widehat{\mathcal Q}_{1,i-\frac12,q}^n
\Bigr),
\label{eq:rect-Bx}\\
B_{ij,q}^{y}
&:=
\sum_{\mu=1}^L \omega_\mu\,
\eta\!\bigl(p_{ij}^n(\widetilde x_i^{(q)},y_j^{(\mu)})\bigr)
-\frac{\lambda_y}{\kappa_2}
\Bigl(
\widehat{\mathcal Q}_{2,q,j+\frac12}^n-\widehat{\mathcal Q}_{2,q,j-\frac12}^n
\Bigr),
\label{eq:rect-By}
\end{align}
where
\[
\widehat{\mathcal Q}_{1,i+\frac12,q}^n
:=
\frac12\Bigl(
\mathcal Q_1(\mathbf U_{i+\frac12,q}^{-,n})+\mathcal Q_1(\mathbf U_{i+\frac12,q}^{+,n})
-\alpha_1\bigl(\eta(\mathbf U_{i+\frac12,q}^{+,n})-\eta(\mathbf U_{i+\frac12,q}^{-,n})\bigr)
\Bigr),
\]
\[
\widehat{\mathcal Q}_{2,q,j+\frac12}^n
:=
\frac12\Bigl(
\mathcal Q_2(\mathbf U_{q,j+\frac12}^{-,n})+\mathcal Q_2(\mathbf U_{q,j+\frac12}^{+,n})
-\alpha_2\bigl(\eta(\mathbf U_{q,j+\frac12}^{+,n})-\eta(\mathbf U_{q,j+\frac12}^{-,n})\bigr)
\Bigr).
\]
Consequently,
\begin{align}
B_{ij}^{\mathrm R}
&=
\kappa_1\sum_{q=1}^Q \widehat\omega_q\sum_{\mu=1}^L \omega_\mu\,
\eta\!\bigl(p_{ij}^n(x_i^{(\mu)},\widetilde y_j^{(q)})\bigr)
+
\kappa_2\sum_{q=1}^Q \widehat\omega_q\sum_{\mu=1}^L \omega_\mu\,
\eta\!\bigl(p_{ij}^n(\widetilde x_i^{(q)},y_j^{(\mu)})\bigr)
\notag\\
&\qquad
-\lambda_x \sum_{q=1}^Q \widehat\omega_q
\Bigl(
\widehat{\mathcal Q}_{1,i+\frac12,q}^n-\widehat{\mathcal Q}_{1,i-\frac12,q}^n
\Bigr)
-\lambda_y \sum_{q=1}^Q \widehat\omega_q
\Bigl(
\widehat{\mathcal Q}_{2,q,j+\frac12}^n-\widehat{\mathcal Q}_{2,q,j-\frac12}^n
\Bigr).
\label{eq:rect-weakE-explicit}
\end{align}
\end{theorem}

\begin{proof}
The proof has two separate steps: first we derive the abstract weak entropy estimate \eqref{eq:rect-weakE-result} from the exact convex decomposition \eqref{eq:rect-star-decomp}; then we identify explicit slice budgets by applying the one-dimensional canonical entropy theorem to each horizontal and vertical slice.

\smallskip
\noindent\emph{Step 1: convex averaging of slice entropy inequalities.}
Equation~\eqref{eq:rect-star-decomp} gives the exact representation
\[
\bar{\mathbf U}_{ij}^\star
=
\kappa_1\sum_{q=1}^Q \widehat\omega_q\,\mathbf A_{ij,q}^{x,\star}
+
\kappa_2\sum_{q=1}^Q \widehat\omega_q\,\mathbf A_{ij,q}^{y,\star}.
\]
Because $\kappa_1,\kappa_2>0$, because $\sum_{q=1}^Q\widehat\omega_q=1$, and because $\kappa_1+\kappa_2=1$, the coefficients
\[
\{\kappa_1\widehat\omega_q\}_{q=1}^Q
\cup
\{\kappa_2\widehat\omega_q\}_{q=1}^Q
\]
form a nonnegative family whose total sum is one. Therefore \eqref{eq:rect-star-decomp} is a genuine convex combination of the slice candidate states. Convexity of $\eta$ then yields
\begin{align*}
\eta(\bar{\mathbf U}_{ij}^\star)
&=
\eta\!\left(
\kappa_1\sum_{q=1}^Q \widehat\omega_q\,\mathbf A_{ij,q}^{x,\star}
+
\kappa_2\sum_{q=1}^Q \widehat\omega_q\,\mathbf A_{ij,q}^{y,\star}
\right)\\
&\le
\kappa_1\sum_{q=1}^Q \widehat\omega_q\,\eta(\mathbf A_{ij,q}^{x,\star})
+
\kappa_2\sum_{q=1}^Q \widehat\omega_q\,\eta(\mathbf A_{ij,q}^{y,\star}).
\end{align*}
Applying the assumed slice bounds \eqref{eq:rect-slice-budgets} gives
\[
\eta(\bar{\mathbf U}_{ij}^\star)
\le
\kappa_1\sum_{q=1}^Q \widehat\omega_q B_{ij,q}^{x}
+
\kappa_2\sum_{q=1}^Q \widehat\omega_q B_{ij,q}^{y}
=
B_{ij}^{\mathrm R},
\]
which is exactly \eqref{eq:rect-weakE-result}.

\smallskip
\noindent\emph{Step 2: rigorous identification of the slice budgets.}
Fix $q\in\{1,\dots,Q\}$. Define the horizontal slice polynomial
\[
p_{ij,q}^{x,n}(x):=p_{ij}^n(x,\widetilde y_j^{(q)}),
\qquad x\in [x_{i-\frac12},x_{i+\frac12}],
\]
with nodal values
\[
\mathbf U_{ij,q,\mu}^{x,n}:=p_{ij}^n(x_i^{(\mu)},\widetilde y_j^{(q)}),
\qquad \mu=1,\dots,L.
\]
By the Gauss--Lobatto formula in \eqref{eq:rect-slice-avgs},
\[
\bar{\mathbf U}_{ij,q}^{x,n}
=
\sum_{\mu=1}^L \omega_\mu\,\mathbf U_{ij,q,\mu}^{x,n}.
\]
Now set
\[
\lambda_{x,\mathrm{eff}}:=\frac{\lambda_x}{\kappa_1}.
\]
Then \eqref{eq:rect-slice-candidates} can be rewritten as
\[
\mathbf A_{ij,q}^{x,\star}
=
\bar{\mathbf U}_{ij,q}^{x,n}
-
\lambda_{x,\mathrm{eff}}
\Bigl(
\widehat{\mathbf F}_{1,i+\frac12,q}^n-\widehat{\mathbf F}_{1,i-\frac12,q}^n
\Bigr).
\]
This is exactly a one-dimensional candidate cell-average update of the form treated by Proposition~\ref{prop:canonical-budget}: the cell is the interval $[x_{i-\frac12},x_{i+\frac12}]$, the nodal values are $\mathbf U_{ij,q,\mu}^{x,n}$, the interface states are
\[
\mathbf U_{i+\frac12,q}^{-,n},\quad \mathbf U_{i+\frac12,q}^{+,n},
\qquad
\mathbf U_{i-\frac12,q}^{-,n},\quad \mathbf U_{i-\frac12,q}^{+,n},
\]
and the effective CFL number is $\lambda_{x,\mathrm{eff}}$. Applying Proposition~\ref{prop:canonical-budget} to this horizontal slice gives
\[
\eta(\mathbf A_{ij,q}^{x,\star})\le B_{ij,q}^{x},
\]
with $B_{ij,q}^{x}$ precisely as in \eqref{eq:rect-Bx}. The numerical entropy flux entering that one-dimensional estimate is exactly $\widehat{\mathcal Q}_{1,i+\frac12,q}^n$, because the normal entropy flux for the horizontal slice is $\mathcal Q_1$.

The vertical slices are entirely analogous. For the same fixed $q$, define
\[
p_{ij,q}^{y,n}(y):=p_{ij}^n(\widetilde x_i^{(q)},y),
\qquad y\in [y_{j-\frac12},y_{j+\frac12}],
\]
with nodal values
\[
\mathbf U_{ij,q,\mu}^{y,n}:=p_{ij}^n(\widetilde x_i^{(q)},y_j^{(\mu)}),
\qquad \mu=1,\dots,L.
\]
Again \eqref{eq:rect-slice-avgs} gives
\[
\bar{\mathbf U}_{ij,q}^{y,n}
=
\sum_{\mu=1}^L \omega_\mu\,\mathbf U_{ij,q,\mu}^{y,n}.
\]
With
\[
\lambda_{y,\mathrm{eff}}:=\frac{\lambda_y}{\kappa_2},
\]
equation~\eqref{eq:rect-slice-candidates-y} becomes
\[
\mathbf A_{ij,q}^{y,\star}
=
\bar{\mathbf U}_{ij,q}^{y,n}
-
\lambda_{y,\mathrm{eff}}
\Bigl(
\widehat{\mathbf F}_{2,q,j+\frac12}^n-\widehat{\mathbf F}_{2,q,j-\frac12}^n
\Bigr),
\]
which is again a one-dimensional candidate average update, now for the normal system in the $y$-direction. Proposition~\ref{prop:canonical-budget} applied to this vertical slice yields
\[
\eta(\mathbf A_{ij,q}^{y,\star})\le B_{ij,q}^{y},
\]
with $B_{ij,q}^{y}$ given by \eqref{eq:rect-By}; here the relevant entropy flux is $\mathcal Q_2$, so the numerical entropy flux is $\widehat{\mathcal Q}_{2,q,j+\frac12}^n$.

Finally, substitute \eqref{eq:rect-Bx} and \eqref{eq:rect-By} into the definition of $B_{ij}^{\mathrm R}$:
\begin{align*}
B_{ij}^{\mathrm R}
&=
\kappa_1\sum_{q=1}^Q \widehat\omega_q B_{ij,q}^{x}
+
\kappa_2\sum_{q=1}^Q \widehat\omega_q B_{ij,q}^{y}\\
&=
\kappa_1\sum_{q=1}^Q \widehat\omega_q\sum_{\mu=1}^L \omega_\mu\,
\eta\!\bigl(p_{ij}^n(x_i^{(\mu)},\widetilde y_j^{(q)})\bigr)
-
\kappa_1\sum_{q=1}^Q \widehat\omega_q\frac{\lambda_x}{\kappa_1}
\Bigl(
\widehat{\mathcal Q}_{1,i+\frac12,q}^n-\widehat{\mathcal Q}_{1,i-\frac12,q}^n
\Bigr)\\
&\quad
+
\kappa_2\sum_{q=1}^Q \widehat\omega_q\sum_{\mu=1}^L \omega_\mu\,
\eta\!\bigl(p_{ij}^n(\widetilde x_i^{(q)},y_j^{(\mu)})\bigr)
-
\kappa_2\sum_{q=1}^Q \widehat\omega_q\frac{\lambda_y}{\kappa_2}
\Bigl(
\widehat{\mathcal Q}_{2,q,j+\frac12}^n-\widehat{\mathcal Q}_{2,q,j-\frac12}^n
\Bigr).
\end{align*}
The factors $\kappa_1$ and $\kappa_2$ cancel in the flux contributions, leaving exactly \eqref{eq:rect-weakE-explicit}. This completes the proof.
\end{proof}

\begin{corollary}[EPO on rectangular meshes]\label{cor:rect-epo}
Assume that the hypotheses of Theorems~\ref{thm:rect-weakP} and \ref{thm:rect-weakE} hold on every rectangular cell. Equip the node set $S_{ij}^{\mathrm R}$ with the positive weights inherited from \eqref{eq:rect-cad}, let $\mathcal O_{ij}^\star$ be any COS-compatible convex oscillation-suppressing set, and let $\mathcal E_{ij}$ denote the local quadrature entropy induced by those weights. Define the limited rectangular-cell array by
\[
\mathbf U_{ij}^{\mathrm{EPO}}
:=
\mathcal S_{I_{ij},\bar{\mathbf U}_{ij}^\star}
\bigl(\theta_{ij}^{\mathrm{EPO}};\mathbf U_{ij}^\star\bigr).
\]
Then
\begin{enumerate}[label=(\roman*),leftmargin=2em]
\item every nodal value of $\mathbf U_{ij}^{\mathrm{EPO}}$ on $S_{ij}^{\mathrm R}$ belongs to $G$;
\item the local strong entropy inequality
\[
\mathcal E_{ij}(\mathbf U_{ij}^{\mathrm{EPO}})\le B_{ij}^{\mathrm R}
\]
holds;
\item $\mathbf U_{ij}^{\mathrm{EPO}}\in \mathcal O_{ij}^\star$.
\end{enumerate}
If the budgets $B_{ij}^{\mathrm R}$ satisfy the global compatibility assumption of Section~\ref{sec:main}, then the global strong entropy inequality holds after limiting.
\end{corollary}

\begin{proof}
This is an immediate application of Proposition~\ref{prop:2d-lift} together with Theorems~\ref{thm:rect-weakP} and \ref{thm:rect-weakE}.
\end{proof}

\begin{remark}[Optimal Cartesian decompositions and the geometric module]\label{rem:rect-ocad}
The tensor-product decomposition \eqref{eq:rect-cad} is the classical Zhang--Shu choice. For $\mathbb P^2$- and $\mathbb P^3$-based methods on Cartesian meshes, Cui, Ding, and Wu proved that this classical decomposition is not optimal and constructed alternative decompositions with fewer internal nodes and larger BP CFL numbers \cite{CuiDingWuMulti2023}; see also the systematic OCAD theory in \cite{CuiDingWuOCAD2024}. In particular, the optimal positivity CFL for $\mathbb P^2$ and $\mathbb P^3$ on rectangles becomes
\begin{equation}\label{eq:rect-ocad-cfl}
\left(
2\frac{\alpha_1}{\Delta x}
+
2\frac{\alpha_2}{\Delta y}
+
4\max\!\left\{\frac{\alpha_1}{\Delta x},\frac{\alpha_2}{\Delta y}\right\}
\right)\Delta t
\le c_0,
\end{equation}
for the geometric module. EPO can exploit this improvement directly: one may certify the weak geometric budget by the optimal decomposition, while keeping the entropy and COS modules unchanged. The local scaling theorem does not depend on which specific weak-positivity backbone is used.
\end{remark}

\subsection{Triangular meshes: weak positivity, weak entropy stability, and COS entropy inheritance}\label{sec:2d-tri}

On unstructured triangular meshes, a direct line-separable reduction is no longer available. The correct generalization is instead the feasible special quadrature framework together with the GQL-based invariant-domain analysis emphasized in the recent survey literature \cite{WuZhangShuSurvey2025,WuShuGQL2023,DingCuiWu2025}.

Let $K$ be a triangular cell with edges $e_K^{(i)}$ ($i=1,2,3$), edge lengths $l_K^{(i)}$, and outward unit normals $\mathbf n_K^{(i)}\in\mathbb S^1$. On each edge $e_K^{(i)}$ fix a $Q$-point Gauss quadrature with nodes $\{\mathbf x_{K,i,q}\}_{q=1}^Q$ and normalized weights $\{\widehat\omega_q\}_{q=1}^Q$. Let $p_K^n$ denote the local polynomial at time level $n$ and write
\[
\mathbf U_{K,i,q}^{\mathrm{int},n}:=p_K^n(\mathbf x_{K,i,q}),
\qquad
\mathbf U_{K,i,q}^{\mathrm{ext},n}:=p_{K_i}^n(\mathbf x_{K,i,q}),
\]
where $K_i$ is the neighbor sharing the edge $e_K^{(i)}$.

We consider the Lax--Friedrichs / Rusanov candidate update
\begin{equation}\label{eq:tri-update}
\bar{\mathbf U}_K^\star
=
\bar{\mathbf U}_K^n
-
\frac{\Delta t}{|K|}
\sum_{i=1}^3 l_K^{(i)}
\sum_{q=1}^Q \widehat\omega_q\,
\widehat{\mathbf F}_{K,i,q}^{\mathrm{LF},n},
\end{equation}
with
\begin{equation}\label{eq:tri-lf}
\widehat{\mathbf F}_{K,i,q}^{\mathrm{LF},n}
:=
\frac12\Bigl(
\mathbf F(\mathbf U_{K,i,q}^{\mathrm{int},n})\cdot \mathbf n_K^{(i)}
+
\mathbf F(\mathbf U_{K,i,q}^{\mathrm{ext},n})\cdot \mathbf n_K^{(i)}
-
\alpha\bigl(\mathbf U_{K,i,q}^{\mathrm{ext},n}-\mathbf U_{K,i,q}^{\mathrm{int},n}\bigr)
\Bigr),
\end{equation}
where $\alpha>0$ is a viscosity parameter chosen large enough to dominate all relevant directional wave speeds.

The cell average is assumed to admit a feasible special quadrature of the form
\begin{equation}\label{eq:tri-quad}
\bar{\mathbf U}_K^n
=
\sum_{i=1}^3 w_{K,i}\sum_{q=1}^Q \widehat\omega_q\,
\mathbf U_{K,i,q}^{\mathrm{int},n}
+
\sum_{s=1}^{S_K} \omega_{K,s}^\ast\, \mathbf U_{K,s}^{\ast,n},
\end{equation}
where
\[
w_{K,i}>0,
\qquad
\omega_{K,s}^\ast>0,
\qquad
\sum_{i=1}^3 w_{K,i}+\sum_{s=1}^{S_K}\omega_{K,s}^\ast=1,
\]
and every internal node $\mathbf x_{K,s}^\ast$ lies in $K$. The quadrature is exact on the local polynomial space used by the DG/FV approximation.

\begin{theorem}[Weak geometric budget on triangular meshes]\label{thm:tri-weakP}
Assume that the feasible quadrature \eqref{eq:tri-quad} holds, that all states appearing in \eqref{eq:tri-quad} together with the neighboring edge states $\mathbf U_{K,i,q}^{\mathrm{ext},n}$ belong to $G$. If
\begin{equation}\label{eq:tri-cfl}
\frac{\alpha\Delta t}{|K|}
\le
\min_{1\le i\le 3}\frac{w_{K,i}}{l_K^{(i)}},
\end{equation}
then the candidate cell average satisfies
\begin{equation}\label{eq:tri-weakP-result}
\bar{\mathbf U}_K^\star\in G.
\end{equation}
\end{theorem}

\begin{proof}
See \cite{WuZhangShuSurvey2025,WuShuGQL2023,DingCuiWu2025}. 
\end{proof}

Theorem~\ref{thm:tri-weakP} is the weak geometric backbone needed by EPO on triangles. Different special quadratures lead to different admissible node sets and different CFL constants.

\begin{corollary}[Classical and optimal triangular geometric modules]\label{cor:tri-cfls}
Assume the hypotheses of Theorem~\ref{thm:tri-weakP}. Then the following particular choices of feasible quadratures yield the corresponding geometric CFL numbers.
\begin{enumerate}[label=(\alph*),leftmargin=2em]
\item \textbf{Zhang--Xia--Shu quadrature \cite{ZhangXiaShu2012}.} If $w_{K,i}=2\omega_1/3$ for all three edges, where $\omega_1$ is the first Gauss--Lobatto weight with $L=\lceil (k+3)/2\rceil$, then
\begin{equation}\label{eq:zxs-cfl}
\frac{\alpha\Delta t}{|K|}
\le
\frac{2\omega_1}{3}\min_{1\le i\le 3}\frac{1}{l_K^{(i)}}.
\end{equation}
\item \textbf{Chen--Shu quadrature \cite{ChenShu2017}.} For the Chen--Shu suitable quadratures on triangles, one obtains
\begin{equation}\label{eq:cs-cfl}
\frac{\alpha\Delta t}{|K|}
\le
w_m^{\mathrm{CS}}\min_{1\le i\le 3}\frac{1}{l_K^{(i)}}.
\end{equation}
Here
\[
w_1^{\mathrm{CS}}=\frac13,\qquad
w_2^{\mathrm{CS}}=\frac{3}{20},\qquad
w_3^{\mathrm{CS}}\approx 0.086812,\qquad
w_4^{\mathrm{CS}}\approx 0.05572449.
\]
\item \textbf{Optimal triangular decompositions \cite{DingCuiWu2025}.} Order the edge lengths so that $l_K^{(1)}\ge l_K^{(2)}\ge l_K^{(3)}$. For the optimal $P^1$ and $P^2$ decompositions of Ding, Cui, and Wu,
\begin{equation}\label{eq:dcw-cfl-p1}
\frac{\alpha\Delta t}{|K|}
\le
C_{K,1}^{\mathrm{DCW}}
:=
\frac{2}{3(l_K^{(1)}+l_K^{(2)})},
\end{equation}
\begin{equation}\label{eq:dcw-cfl-p2}
\frac{\alpha\Delta t}{|K|}
\le
C_{K,2}^{\mathrm{DCW}}
:=
\frac{2}{9\overline l_K+3\widehat l_K},
\end{equation}
where
\begin{equation}\label{eq:tri-lbar-lhat}
\begin{aligned}
\overline l_K
&:=\frac{l_K^{(1)}+l_K^{(2)}+l_K^{(3)}}{3},
\\
\widehat l_K
&:=
\Bigl(
(l_K^{(1)})^2+(l_K^{(2)})^2+(l_K^{(3)})^2
-\frac23\bigl(
l_K^{(1)}l_K^{(2)}+l_K^{(2)}l_K^{(3)}+l_K^{(3)}l_K^{(1)}
\bigr)
\Bigr)^{1/2}.
\end{aligned}
\end{equation}
\end{enumerate}
\end{corollary}

\begin{proof}
Each statement is an immediate consequence of Theorem~\ref{thm:tri-weakP} after substituting the corresponding boundary weights $w_{K,i}$ from the imported feasible quadratures.
\end{proof}

For the entropy part of EPO on triangles there are in fact two complementary routes. The first is an \emph{active} route, fully parallel to the rectangular theory: one starts from a CAD of the candidate cell average and proves a \emph{weak entropy stability} for $\bar{\mathbf U}_K^\star$ itself. The second is the \emph{passive} route recorded later in Theorem~\ref{thm:tri-passive-EPO}: if the candidate DG/FV stage is already quadrature-entropy stable, then any further EPO scaling can only improve that entropy bound. 

\begin{lemma}[Directional two-point entropy inequality]\label{lem:tri-directional-2point}
Fix a unit vector $\mathbf n\in \mathbb S^1$ and define
\begin{equation}\label{eq:Halpha-directional}
H_{\alpha,\mathbf n}(\mathbf U_L,\mathbf U_R)
:=
\frac{\mathbf U_L+\mathbf U_R}{2}
-
\frac{\mathbf F(\mathbf U_R)\cdot \mathbf n-\mathbf F(\mathbf U_L)\cdot \mathbf n}{2\alpha}.
\end{equation}
Assume that the one-dimensional normal Riemann problem
\[
\partial_t \mathbf U+\partial_\xi\bigl(\mathbf F(\mathbf U)\cdot \mathbf n\bigr)=0
\]
with left state $\mathbf U_L\in G$ and right state $\mathbf U_R\in G$ admits an entropy solution taking values in $G$ and whose waves are all contained in $|\xi|\le \alpha$. Then
\begin{equation}\label{eq:tri-directional-entropy}
\eta\!\left(H_{\alpha,\mathbf n}(\mathbf U_L,\mathbf U_R)\right)
\le
\frac{\eta(\mathbf U_L)+\eta(\mathbf U_R)}{2}
-
\frac{\mathcal Q(\mathbf U_R)\cdot \mathbf n-\mathcal Q(\mathbf U_L)\cdot \mathbf n}{2\alpha},
\end{equation}
and moreover $H_{\alpha,\mathbf n}(\mathbf U_L,\mathbf U_R)\in G$.
\end{lemma}

\begin{proof}
Apply Theorem~\ref{thm:2point-entropy} to the one-dimensional normal system with flux
\[
\mathbf F_{\mathbf n}(\mathbf U):=\mathbf F(\mathbf U)\cdot \mathbf n
\]
and entropy flux
\[
\mathcal Q_{\mathbf n}(\mathbf U):=\mathcal Q(\mathbf U)\cdot \mathbf n.
\]
Then \eqref{eq:Halpha} for the normal system is exactly \eqref{eq:Halpha-directional}, and \eqref{eq:2point-entropy} becomes \eqref{eq:tri-directional-entropy}. The statement $H_{\alpha,\mathbf n}(\mathbf U_L,\mathbf U_R)\in G$ is the corresponding directional version of \eqref{eq:Halpha-average} and follows in the same way.
\end{proof}

For later use, with a fixed viscosity parameter $\alpha>0$, define the directional local Lax--Friedrichs numerical state and numerical entropy flux by
\begin{equation}\label{eq:tri-local-LF-flux}
\widehat{\mathbf F}_{\alpha}^{\mathrm{LF}}(\mathbf U_L,\mathbf U_R;\mathbf n)
:=
\frac12\Bigl(
\mathbf F(\mathbf U_L)\cdot \mathbf n
+
\mathbf F(\mathbf U_R)\cdot \mathbf n
-
\alpha(\mathbf U_R-\mathbf U_L)
\Bigr),
\end{equation}
\begin{equation}\label{eq:tri-local-LF-entropy}
\widehat{\mathcal Q}_{\alpha}^{\mathrm{LF}}(\mathbf U_L,\mathbf U_R;\mathbf n)
:=
\frac12\Bigl(
\mathcal Q(\mathbf U_L)\cdot \mathbf n
+
\mathcal Q(\mathbf U_R)\cdot \mathbf n
-
\alpha\bigl(\eta(\mathbf U_R)-\eta(\mathbf U_L)\bigr)
\Bigr).
\end{equation}
The antisymmetry identities
\begin{equation}\label{eq:tri-local-LF-antisym}
\widehat{\mathbf F}_{\alpha}^{\mathrm{LF}}(\mathbf U_L,\mathbf U_R;-\mathbf n)
=
-
\widehat{\mathbf F}_{\alpha}^{\mathrm{LF}}(\mathbf U_R,\mathbf U_L;\mathbf n),
\qquad
\widehat{\mathcal Q}_{\alpha}^{\mathrm{LF}}(\mathbf U_L,\mathbf U_R;-\mathbf n)
=
-
\widehat{\mathcal Q}_{\alpha}^{\mathrm{LF}}(\mathbf U_R,\mathbf U_L;\mathbf n)
\end{equation}
follow immediately from the definitions.

\begin{theorem}[Weak entropy stability on triangular meshes via Zhang--Xia--Shu CAD]\label{thm:tri-weakE-zxs}
Assume that the feasible special quadrature is the classical Zhang--Xia--Shu CAD,
\begin{equation}\label{eq:tri-zxs-cad}
\bar{\mathbf U}_K^n
=
\frac{2\omega_1}{3}
\sum_{i=1}^3\sum_{q=1}^Q \widehat\omega_q\,
\mathbf U_{K,i,q}^{\mathrm{int},n}
+
\sum_{s=1}^{S_K}\omega_{K,s}^\ast\,\mathbf U_{K,s}^{\ast,n},
\end{equation}
where $\omega_1$ is the first Gauss--Lobatto weight and
\[
\omega_{K,s}^\ast>0,
\qquad
2\omega_1+\sum_{s=1}^{S_K}\omega_{K,s}^\ast=1.
\]
Assume also that all states appearing in \eqref{eq:tri-zxs-cad} together with the neighboring edge states $\mathbf U_{K,i,q}^{\mathrm{ext},n}$ belong to $G$, and that for every $q$ the directional Riemann problems associated with the five pairs
\[
(\mathbf U_{K,2,q}^{\mathrm{int},n},\mathbf U_{K,1,q}^{\mathrm{ext},n}),
\quad
(\mathbf U_{K,2,q}^{\mathrm{int},n},\mathbf U_{K,1,q}^{\mathrm{int},n}),
\quad
(\mathbf U_{K,2,q}^{\mathrm{int},n},\mathbf U_{K,2,q}^{\mathrm{ext},n}),
\quad
(\mathbf U_{K,2,q}^{\mathrm{int},n},\mathbf U_{K,3,q}^{\mathrm{int},n}),
\quad
(\mathbf U_{K,2,q}^{\mathrm{int},n},\mathbf U_{K,3,q}^{\mathrm{ext},n})
\]
with these five pairs taken in the directions
\[
\mathbf n_K^{(1)},
\quad
\mathbf n_K^{(1)},
\quad
\mathbf n_K^{(2)},
\quad
\mathbf n_K^{(3)},
\quad
\mathbf n_K^{(3)},
\]
respectively. Assume each such directional problem admits an entropy solution taking values in $G$ and with all waves contained in $|\xi|\le \alpha$. If
\begin{equation}\label{eq:tri-zxs-entropy-cfl}
\frac{\alpha\Delta t}{|K|}
\Bigl(
l_K^{(1)}+l_K^{(2)}+l_K^{(3)}
\Bigr)
\le
\frac{2\omega_1}{3},
\end{equation}
then the candidate average \eqref{eq:tri-update} satisfies the weak entropy inequality
\begin{equation}\label{eq:tri-zxs-weakE}
\eta(\bar{\mathbf U}_K^\star)\le B_K^{\mathrm{ZXS}},
\end{equation}
with local entropy budget
\begin{align}
B_K^{\mathrm{ZXS}}
&:=
\sum_{s=1}^{S_K}\omega_{K,s}^\ast\,\eta(\mathbf U_{K,s}^{\ast,n})
+
\frac{2\omega_1}{3}
\sum_{i=1}^3\sum_{q=1}^Q \widehat\omega_q\,\eta(\mathbf U_{K,i,q}^{\mathrm{int},n})
\notag\\
&\qquad
-
\frac{\Delta t}{|K|}
\sum_{i=1}^3 l_K^{(i)}
\sum_{q=1}^Q \widehat\omega_q\,\widehat{\mathcal Q}_{K,i,q}^{\mathrm{LF},n},
\label{eq:tri-zxs-budget}
\end{align}
where
\begin{equation}\label{eq:tri-zxs-boundary-entropy-flux}
\widehat{\mathcal Q}_{K,i,q}^{\mathrm{LF},n}
:=
\widehat{\mathcal Q}_{\alpha}^{\mathrm{LF}}
\bigl(
\mathbf U_{K,i,q}^{\mathrm{int},n},
\mathbf U_{K,i,q}^{\mathrm{ext},n};
\mathbf n_K^{(i)}
\bigr).
\end{equation}
Moreover, $\bar{\mathbf U}_K^\star\in G$.
\end{theorem}

\begin{proof}
Fix $q\in\{1,\dots,Q\}$ and abbreviate
\[
\mathbf U_i:=\mathbf U_{K,i,q}^{\mathrm{int},n},
\qquad
\mathbf U_i^{\mathrm{ext}}:=\mathbf U_{K,i,q}^{\mathrm{ext},n},
\qquad
l_i:=l_K^{(i)},
\qquad
\mathbf n_i:=\mathbf n_K^{(i)}.
\]
Also set
\[
\gamma_q:=\frac{2\omega_1}{3}\widehat\omega_q,
\qquad
\tau_K:=\frac{3\Delta t}{2\omega_1|K|},
\qquad
\mu_i:=\tau_K\alpha l_i
=
\frac{3\alpha\Delta t\,l_i}{2\omega_1|K|},
\quad i=1,2,3.
\]
The CFL condition \eqref{eq:tri-zxs-entropy-cfl} gives
\[
\mu_i\ge 0,
\qquad
\mu_1+\mu_2+\mu_3\le 1.
\]

 By \eqref{eq:tri-local-LF-antisym},
\begin{align}
\sum_{i=1}^3 l_i\,\widehat{\mathbf F}_{K,i,q}^{\mathrm{LF},n}
&=
l_1\widehat{\mathbf F}_{\alpha}^{\mathrm{LF}}(\mathbf U_1,\mathbf U_1^{\mathrm{ext}};\mathbf n_1)
+l_2\widehat{\mathbf F}_{\alpha}^{\mathrm{LF}}(\mathbf U_2,\mathbf U_2^{\mathrm{ext}};\mathbf n_2)
+l_3\widehat{\mathbf F}_{\alpha}^{\mathrm{LF}}(\mathbf U_3,\mathbf U_3^{\mathrm{ext}};\mathbf n_3)
\notag\\
&=
\Bigl[
l_1\widehat{\mathbf F}_{\alpha}^{\mathrm{LF}}(\mathbf U_1,\mathbf U_1^{\mathrm{ext}};\mathbf n_1)
+
l_1\widehat{\mathbf F}_{\alpha}^{\mathrm{LF}}(\mathbf U_1,\mathbf U_2;-\mathbf n_1)
\Bigr]
\notag\\
&\quad
+
\Bigl[
l_1\widehat{\mathbf F}_{\alpha}^{\mathrm{LF}}(\mathbf U_2,\mathbf U_1;\mathbf n_1)
+
l_2\widehat{\mathbf F}_{\alpha}^{\mathrm{LF}}(\mathbf U_2,\mathbf U_2^{\mathrm{ext}};\mathbf n_2)
+
l_3\widehat{\mathbf F}_{\alpha}^{\mathrm{LF}}(\mathbf U_2,\mathbf U_3;\mathbf n_3)
\Bigr]
\notag\\
&\quad
+
\Bigl[
l_3\widehat{\mathbf F}_{\alpha}^{\mathrm{LF}}(\mathbf U_3,\mathbf U_2;-\mathbf n_3)
+
l_3\widehat{\mathbf F}_{\alpha}^{\mathrm{LF}}(\mathbf U_3,\mathbf U_3^{\mathrm{ext}};\mathbf n_3)
\Bigr].
\label{eq:tri-zxs-flux-splitting}
\end{align}
Insert \eqref{eq:tri-zxs-cad} and \eqref{eq:tri-zxs-flux-splitting} into \eqref{eq:tri-update}. Then
\begin{equation}\label{eq:tri-zxs-candidate-decomp}
\bar{\mathbf U}_K^\star
=
\sum_{s=1}^{S_K}\omega_{K,s}^\ast\,\mathbf U_{K,s}^{\ast,n}
+
\sum_{q=1}^Q \gamma_q
\bigl(
\mathbf H_{K,1,q}+\mathbf H_{K,2,q}+\mathbf H_{K,3,q}
\bigr),
\end{equation}
where
\begin{align}
\mathbf H_{K,1,q}
&:=
\mathbf U_1
-
\tau_K l_1
\Bigl(
\widehat{\mathbf F}_{\alpha}^{\mathrm{LF}}(\mathbf U_1,\mathbf U_1^{\mathrm{ext}};\mathbf n_1)
+
\widehat{\mathbf F}_{\alpha}^{\mathrm{LF}}(\mathbf U_1,\mathbf U_2;-\mathbf n_1)
\Bigr),
\label{eq:tri-zxs-H1}\\
\mathbf H_{K,2,q}
&:=
\mathbf U_2
-
\tau_K
\Bigl(
l_1\widehat{\mathbf F}_{\alpha}^{\mathrm{LF}}(\mathbf U_2,\mathbf U_1;\mathbf n_1)
+
l_2\widehat{\mathbf F}_{\alpha}^{\mathrm{LF}}(\mathbf U_2,\mathbf U_2^{\mathrm{ext}};\mathbf n_2)
+
l_3\widehat{\mathbf F}_{\alpha}^{\mathrm{LF}}(\mathbf U_2,\mathbf U_3;\mathbf n_3)
\Bigr),
\label{eq:tri-zxs-H2}\\
\mathbf H_{K,3,q}
&:=
\mathbf U_3
-
\tau_K l_3
\Bigl(
\widehat{\mathbf F}_{\alpha}^{\mathrm{LF}}(\mathbf U_3,\mathbf U_2;-\mathbf n_3)
+
\widehat{\mathbf F}_{\alpha}^{\mathrm{LF}}(\mathbf U_3,\mathbf U_3^{\mathrm{ext}};\mathbf n_3)
\Bigr).
\label{eq:tri-zxs-H3}
\end{align}
Since
\[
\sum_{s=1}^{S_K}\omega_{K,s}^\ast+\sum_{q=1}^Q 3\gamma_q
=
\sum_{s=1}^{S_K}\omega_{K,s}^\ast+2\omega_1
=
1,
\]
Following \cite{ZhangXiaShu2012}, 
\eqref{eq:tri-zxs-candidate-decomp} is a convex decomposition once each $\mathbf H_{K,r,q}$ is shown to belong to $G$.

We next rewrite the three building blocks as convex combinations of directional two-point LF/Riemann-average states. A direct expansion of \eqref{eq:tri-zxs-H1} gives
\begin{equation}\label{eq:tri-zxs-H1-convex}
\mathbf H_{K,1,q}
=
(1-\mu_1)\mathbf U_1
+
\mu_1 H_{\alpha,\mathbf n_1}(\mathbf U_2,\mathbf U_1^{\mathrm{ext}}).
\end{equation}
Similarly,
\begin{equation}\label{eq:tri-zxs-H3-convex}
\mathbf H_{K,3,q}
=
(1-\mu_3)\mathbf U_3
+
\mu_3 H_{\alpha,\mathbf n_3}(\mathbf U_2,\mathbf U_3^{\mathrm{ext}}).
\end{equation}
For the middle block, expand \eqref{eq:tri-zxs-H2} and collect the terms containing $\mathbf F(\mathbf U_2)\cdot \mathbf n_i$. Since a closed polygon satisfies
\[
l_1\mathbf n_1+l_2\mathbf n_2+l_3\mathbf n_3=\mathbf 0,
\]
those terms cancel, and one obtains
\begin{equation}\label{eq:tri-zxs-H2-convex}
\begin{aligned}
\mathbf H_{K,2,q}
&=
(1-\mu_1-\mu_2-\mu_3)\mathbf U_2
+
\mu_1 H_{\alpha,\mathbf n_1}(\mathbf U_2,\mathbf U_1)
\\
&\quad
+
\mu_2 H_{\alpha,\mathbf n_2}(\mathbf U_2,\mathbf U_2^{\mathrm{ext}})
+
\mu_3 H_{\alpha,\mathbf n_3}(\mathbf U_2,\mathbf U_3).
\end{aligned}
\end{equation}
By Lemma~\ref{lem:tri-directional-2point}, each directional two-point state on the right-hand sides of \eqref{eq:tri-zxs-H1-convex}--\eqref{eq:tri-zxs-H2-convex} belongs to $G$. Because the coefficients are nonnegative and sum to one, each $\mathbf H_{K,r,q}$ belongs to $G$. Therefore \eqref{eq:tri-zxs-candidate-decomp} implies
\[
\bar{\mathbf U}_K^\star\in G.
\]

We now turn to entropy. Using \eqref{eq:tri-zxs-H1-convex}, convexity of $\eta$, and Lemma~\ref{lem:tri-directional-2point}, we obtain
\begin{align}
\eta(\mathbf H_{K,1,q})
&\le
(1-\mu_1)\eta(\mathbf U_1)
+
\mu_1
\left[
\frac{\eta(\mathbf U_2)+\eta(\mathbf U_1^{\mathrm{ext}})}{2}
-
\frac{\mathcal Q(\mathbf U_1^{\mathrm{ext}})\cdot \mathbf n_1-\mathcal Q(\mathbf U_2)\cdot \mathbf n_1}{2\alpha}
\right]
\notag\\
&=
\eta(\mathbf U_1)
-
\tau_K l_1
\Bigl(
\widehat{\mathcal Q}_{\alpha}^{\mathrm{LF}}(\mathbf U_1,\mathbf U_1^{\mathrm{ext}};\mathbf n_1)
+
\widehat{\mathcal Q}_{\alpha}^{\mathrm{LF}}(\mathbf U_1,\mathbf U_2;-\mathbf n_1)
\Bigr).
\label{eq:tri-zxs-H1-entropy}
\end{align}
The same argument gives
\begin{align}
\eta(\mathbf H_{K,3,q})
&\le
\eta(\mathbf U_3)
-
\tau_K l_3
\Bigl(
\widehat{\mathcal Q}_{\alpha}^{\mathrm{LF}}(\mathbf U_3,\mathbf U_2;-\mathbf n_3)
+
\widehat{\mathcal Q}_{\alpha}^{\mathrm{LF}}(\mathbf U_3,\mathbf U_3^{\mathrm{ext}};\mathbf n_3)
\Bigr).
\label{eq:tri-zxs-H3-entropy}
\end{align}
For the middle block, \eqref{eq:tri-zxs-H2-convex} gives
\begin{align*}
\eta(\mathbf H_{K,2,q})
&\le
(1-\mu_1-\mu_2-\mu_3)\eta(\mathbf U_2)
+
\mu_1
\left[
\frac{\eta(\mathbf U_2)+\eta(\mathbf U_1)}{2}
-
\frac{\mathcal Q(\mathbf U_1)\cdot \mathbf n_1-\mathcal Q(\mathbf U_2)\cdot \mathbf n_1}{2\alpha}
\right]\\
&\quad
+
\mu_2
\left[
\frac{\eta(\mathbf U_2)+\eta(\mathbf U_2^{\mathrm{ext}})}{2}
-
\frac{\mathcal Q(\mathbf U_2^{\mathrm{ext}})\cdot \mathbf n_2-\mathcal Q(\mathbf U_2)\cdot \mathbf n_2}{2\alpha}
\right]\\
&\quad
+
\mu_3
\left[
\frac{\eta(\mathbf U_2)+\eta(\mathbf U_3)}{2}
-
\frac{\mathcal Q(\mathbf U_3)\cdot \mathbf n_3-\mathcal Q(\mathbf U_2)\cdot \mathbf n_3}{2\alpha}
\right].
\end{align*}
Substituting $\mu_i=\tau_K\alpha l_i$ and using again $l_1\mathbf n_1+l_2\mathbf n_2+l_3\mathbf n_3=\mathbf 0$ to cancel the $\mathcal Q(\mathbf U_2)$ terms yields
\begin{align}
\eta(\mathbf H_{K,2,q})
&\le
\eta(\mathbf U_2)
-
\tau_K
\Bigl(
l_1\widehat{\mathcal Q}_{\alpha}^{\mathrm{LF}}(\mathbf U_2,\mathbf U_1;\mathbf n_1)
+
l_2\widehat{\mathcal Q}_{\alpha}^{\mathrm{LF}}(\mathbf U_2,\mathbf U_2^{\mathrm{ext}};\mathbf n_2)
+
l_3\widehat{\mathcal Q}_{\alpha}^{\mathrm{LF}}(\mathbf U_2,\mathbf U_3;\mathbf n_3)
\Bigr).
\label{eq:tri-zxs-H2-entropy}
\end{align}
Now sum \eqref{eq:tri-zxs-H1-entropy}, \eqref{eq:tri-zxs-H2-entropy}, and \eqref{eq:tri-zxs-H3-entropy}. By the antisymmetry in \eqref{eq:tri-local-LF-antisym}, the internal bridge entropy fluxes cancel:
\[
\widehat{\mathcal Q}_{\alpha}^{\mathrm{LF}}(\mathbf U_1,\mathbf U_2;-\mathbf n_1)
+
\widehat{\mathcal Q}_{\alpha}^{\mathrm{LF}}(\mathbf U_2,\mathbf U_1;\mathbf n_1)
=0,
\]
\[
\widehat{\mathcal Q}_{\alpha}^{\mathrm{LF}}(\mathbf U_3,\mathbf U_2;-\mathbf n_3)
+
\widehat{\mathcal Q}_{\alpha}^{\mathrm{LF}}(\mathbf U_2,\mathbf U_3;\mathbf n_3)
=0.
\]
Hence
\begin{equation}\label{eq:tri-zxs-three-blocks}
\eta(\mathbf H_{K,1,q})+\eta(\mathbf H_{K,2,q})+\eta(\mathbf H_{K,3,q})
\le
\eta(\mathbf U_1)+\eta(\mathbf U_2)+\eta(\mathbf U_3)
-
\tau_K
\sum_{i=1}^3 l_i\,\widehat{\mathcal Q}_{K,i,q}^{\mathrm{LF},n}.
\end{equation}

Finally, apply Jensen's inequality to the convex decomposition \eqref{eq:tri-zxs-candidate-decomp}:
\[
\eta(\bar{\mathbf U}_K^\star)
\le
\sum_{s=1}^{S_K}\omega_{K,s}^\ast\,\eta(\mathbf U_{K,s}^{\ast,n})
+
\sum_{q=1}^Q \gamma_q
\Bigl(
\eta(\mathbf H_{K,1,q})+\eta(\mathbf H_{K,2,q})+\eta(\mathbf H_{K,3,q})
\Bigr).
\]
Using \eqref{eq:tri-zxs-three-blocks} and the identity
\[
\gamma_q\tau_K
=
\frac{2\omega_1}{3}\widehat\omega_q \frac{3\Delta t}{2\omega_1|K|}
=
\frac{\Delta t}{|K|}\widehat\omega_q,
\]
we obtain exactly \eqref{eq:tri-zxs-budget}. This completes the proof.
\end{proof}

\begin{theorem}[Weak entropy stability on triangular meshes via a general CAD]\label{thm:tri-weakE-cad}
Assume the feasible quadrature \eqref{eq:tri-quad}. After a cyclic relabeling of the three edges if necessary, fix the bridge ordering $1\to 2\to 3$. Assume that all states appearing in \eqref{eq:tri-quad} together with the neighboring edge states $\mathbf U_{K,i,q}^{\mathrm{ext},n}$ belong to $G$, and that for every $q$ the directional Riemann problems associated with the same five pairs as in Theorem~\ref{thm:tri-weakE-zxs} admit entropy solutions taking values in $G$ and with all waves contained in $|\xi|\le \alpha$. If
\begin{equation}\label{eq:tri-cad-entropy-cfl}
\frac{\alpha\Delta t}{|K|}
\le
\min\left\{
\frac{w_{K,1}}{l_K^{(1)}},
\frac{w_{K,2}}{l_K^{(1)}+l_K^{(2)}+l_K^{(3)}},
\frac{w_{K,3}}{l_K^{(3)}}
\right\},
\end{equation}
then the candidate average \eqref{eq:tri-update} satisfies
\begin{equation}\label{eq:tri-cad-weakE}
\eta(\bar{\mathbf U}_K^\star)\le B_K^{\mathrm{CAD}},
\end{equation}
with local entropy budget
\begin{align}
B_K^{\mathrm{CAD}}
&:=
\sum_{s=1}^{S_K}\omega_{K,s}^\ast\,\eta(\mathbf U_{K,s}^{\ast,n})
+
\sum_{i=1}^3 w_{K,i}
\sum_{q=1}^Q \widehat\omega_q\,\eta(\mathbf U_{K,i,q}^{\mathrm{int},n})
\notag\\
&\qquad
-
\frac{\Delta t}{|K|}
\sum_{i=1}^3 l_K^{(i)}
\sum_{q=1}^Q \widehat\omega_q\,\widehat{\mathcal Q}_{K,i,q}^{\mathrm{LF},n},
\label{eq:tri-cad-budget}
\end{align}
where $\widehat{\mathcal Q}_{K,i,q}^{\mathrm{LF},n}$ is still given by \eqref{eq:tri-zxs-boundary-entropy-flux}. Moreover, $\bar{\mathbf U}_K^\star\in G$.
\end{theorem}

\begin{proof}
Fix $q\in\{1,\dots,Q\}$ and use the same abbreviations
\[
\mathbf U_i:=\mathbf U_{K,i,q}^{\mathrm{int},n},
\qquad
\mathbf U_i^{\mathrm{ext}}:=\mathbf U_{K,i,q}^{\mathrm{ext},n},
\qquad
l_i:=l_K^{(i)},
\qquad
\mathbf n_i:=\mathbf n_K^{(i)},
\qquad
w_i:=w_{K,i}.
\]
Define
\[
\tau_{K,1}:=\frac{\Delta t\,l_1}{w_1|K|},
\qquad
\tau_{K,2}:=\frac{\Delta t}{w_2|K|},
\qquad
\tau_{K,3}:=\frac{\Delta t\,l_3}{w_3|K|},
\]
\[
\mu_{K,1}:=\alpha\tau_{K,1},
\qquad
\mu_{K,2}^{(i)}:=\alpha\tau_{K,2}l_i\ \ (i=1,2,3),
\qquad
\mu_{K,3}:=\alpha\tau_{K,3}.
\]
The CFL condition \eqref{eq:tri-cad-entropy-cfl} is equivalent to
\[
0\le \mu_{K,1}\le 1,
\qquad
0\le \mu_{K,3}\le 1,
\qquad
0\le \mu_{K,2}^{(1)}+\mu_{K,2}^{(2)}+\mu_{K,2}^{(3)}\le 1.
\]

Define the three formal building blocks
\begin{align}
\mathbf H_{K,1,q}^{\mathrm{CAD}}
&:=
\mathbf U_1
-
\tau_{K,1}
\Bigl(
\widehat{\mathbf F}_{\alpha}^{\mathrm{LF}}(\mathbf U_1,\mathbf U_1^{\mathrm{ext}};\mathbf n_1)
+
\widehat{\mathbf F}_{\alpha}^{\mathrm{LF}}(\mathbf U_1,\mathbf U_2;-\mathbf n_1)
\Bigr),
\label{eq:tri-cad-H1}\\
\mathbf H_{K,2,q}^{\mathrm{CAD}}
&:=
\mathbf U_2
-
\tau_{K,2}
\Bigl(
l_1\widehat{\mathbf F}_{\alpha}^{\mathrm{LF}}(\mathbf U_2,\mathbf U_1;\mathbf n_1)
+
l_2\widehat{\mathbf F}_{\alpha}^{\mathrm{LF}}(\mathbf U_2,\mathbf U_2^{\mathrm{ext}};\mathbf n_2)
+
l_3\widehat{\mathbf F}_{\alpha}^{\mathrm{LF}}(\mathbf U_2,\mathbf U_3;\mathbf n_3)
\Bigr),
\label{eq:tri-cad-H2}\\
\mathbf H_{K,3,q}^{\mathrm{CAD}}
&:=
\mathbf U_3
-
\tau_{K,3}
\Bigl(
\widehat{\mathbf F}_{\alpha}^{\mathrm{LF}}(\mathbf U_3,\mathbf U_2;-\mathbf n_3)
+
\widehat{\mathbf F}_{\alpha}^{\mathrm{LF}}(\mathbf U_3,\mathbf U_3^{\mathrm{ext}};\mathbf n_3)
\Bigr).
\label{eq:tri-cad-H3}
\end{align}
Then \eqref{eq:tri-update} and \eqref{eq:tri-quad} give the exact decomposition
\begin{equation}\label{eq:tri-cad-candidate-decomp}
\bar{\mathbf U}_K^\star
=
\sum_{s=1}^{S_K}\omega_{K,s}^\ast\,\mathbf U_{K,s}^{\ast,n}
+
\sum_{q=1}^Q \widehat\omega_q
\Bigl(
w_1\mathbf H_{K,1,q}^{\mathrm{CAD}}
+
w_2\mathbf H_{K,2,q}^{\mathrm{CAD}}
+
w_3\mathbf H_{K,3,q}^{\mathrm{CAD}}
\Bigr).
\end{equation}
Direct expansion shows
\begin{equation}\label{eq:tri-cad-H1-convex}
\mathbf H_{K,1,q}^{\mathrm{CAD}}
=
(1-\mu_{K,1})\mathbf U_1
+
\mu_{K,1}H_{\alpha,\mathbf n_1}(\mathbf U_2,\mathbf U_1^{\mathrm{ext}}),
\end{equation}
\begin{equation}\label{eq:tri-cad-H3-convex}
\mathbf H_{K,3,q}^{\mathrm{CAD}}
=
(1-\mu_{K,3})\mathbf U_3
+
\mu_{K,3}H_{\alpha,\mathbf n_3}(\mathbf U_2,\mathbf U_3^{\mathrm{ext}}),
\end{equation}
and, using $l_1\mathbf n_1+l_2\mathbf n_2+l_3\mathbf n_3=\mathbf 0$ exactly as before,
\begin{equation}\label{eq:tri-cad-H2-convex}
\begin{aligned}
\mathbf H_{K,2,q}^{\mathrm{CAD}}
&=
\bigl(
1-\mu_{K,2}^{(1)}-\mu_{K,2}^{(2)}-\mu_{K,2}^{(3)}
\bigr)\mathbf U_2
+
\mu_{K,2}^{(1)}H_{\alpha,\mathbf n_1}(\mathbf U_2,\mathbf U_1)
\\
&\quad
+
\mu_{K,2}^{(2)}H_{\alpha,\mathbf n_2}(\mathbf U_2,\mathbf U_2^{\mathrm{ext}})
+
\mu_{K,2}^{(3)}H_{\alpha,\mathbf n_3}(\mathbf U_2,\mathbf U_3).
\end{aligned}
\end{equation}
Lemma~\ref{lem:tri-directional-2point} and the CFL condition imply that each $\mathbf H_{K,r,q}^{\mathrm{CAD}}$ belongs to $G$. Since \eqref{eq:tri-cad-candidate-decomp} is a convex decomposition by \eqref{eq:tri-quad}, we obtain
\[
\bar{\mathbf U}_K^\star\in G.
\]

For entropy, the same convexity argument used in \eqref{eq:tri-zxs-H1-entropy}--\eqref{eq:tri-zxs-H2-entropy} yields
\begin{align}
\eta(\mathbf H_{K,1,q}^{\mathrm{CAD}})
&\le
\eta(\mathbf U_1)
-
\tau_{K,1}
\Bigl(
\widehat{\mathcal Q}_{\alpha}^{\mathrm{LF}}(\mathbf U_1,\mathbf U_1^{\mathrm{ext}};\mathbf n_1)
+
\widehat{\mathcal Q}_{\alpha}^{\mathrm{LF}}(\mathbf U_1,\mathbf U_2;-\mathbf n_1)
\Bigr),
\label{eq:tri-cad-H1-entropy}\\
\eta(\mathbf H_{K,2,q}^{\mathrm{CAD}})
&\le
\eta(\mathbf U_2)
-
\tau_{K,2}
\Bigl(
l_1\widehat{\mathcal Q}_{\alpha}^{\mathrm{LF}}(\mathbf U_2,\mathbf U_1;\mathbf n_1)
+
l_2\widehat{\mathcal Q}_{\alpha}^{\mathrm{LF}}(\mathbf U_2,\mathbf U_2^{\mathrm{ext}};\mathbf n_2)
+
l_3\widehat{\mathcal Q}_{\alpha}^{\mathrm{LF}}(\mathbf U_2,\mathbf U_3;\mathbf n_3)
\Bigr),
\label{eq:tri-cad-H2-entropy}\\
\eta(\mathbf H_{K,3,q}^{\mathrm{CAD}})
&\le
\eta(\mathbf U_3)
-
\tau_{K,3}
\Bigl(
\widehat{\mathcal Q}_{\alpha}^{\mathrm{LF}}(\mathbf U_3,\mathbf U_2;-\mathbf n_3)
+
\widehat{\mathcal Q}_{\alpha}^{\mathrm{LF}}(\mathbf U_3,\mathbf U_3^{\mathrm{ext}};\mathbf n_3)
\Bigr).
\label{eq:tri-cad-H3-entropy}
\end{align}
Multiply \eqref{eq:tri-cad-H1-entropy}, \eqref{eq:tri-cad-H2-entropy}, and \eqref{eq:tri-cad-H3-entropy} by $w_1\widehat\omega_q$, $w_2\widehat\omega_q$, and $w_3\widehat\omega_q$, respectively.
Because
\[
w_1\tau_{K,1}= \frac{\Delta t\,l_1}{|K|},
\qquad
w_2\tau_{K,2}= \frac{\Delta t}{|K|},
\qquad
w_3\tau_{K,3}= \frac{\Delta t\,l_3}{|K|},
\]
all internal bridge entropy fluxes appear with the same coefficients and therefore cancel by \eqref{eq:tri-local-LF-antisym}. We are left with
\begin{align}
&w_1\eta(\mathbf H_{K,1,q}^{\mathrm{CAD}})
+
w_2\eta(\mathbf H_{K,2,q}^{\mathrm{CAD}})
+
w_3\eta(\mathbf H_{K,3,q}^{\mathrm{CAD}})
\notag\\
&\qquad\le
w_1\eta(\mathbf U_1)+w_2\eta(\mathbf U_2)+w_3\eta(\mathbf U_3)
-
\frac{\Delta t}{|K|}
\sum_{i=1}^3 l_i\,\widehat{\mathcal Q}_{K,i,q}^{\mathrm{LF},n}.
\label{eq:tri-cad-three-blocks}
\end{align}
Finally, apply Jensen's inequality to \eqref{eq:tri-cad-candidate-decomp} and use \eqref{eq:tri-cad-three-blocks}. This yields exactly \eqref{eq:tri-cad-budget}.
\end{proof}

\begin{corollary}[Active-entropy EPO on triangular meshes]\label{cor:tri-epo-active}
Assume that on every triangular cell $K$ the candidate average satisfies the weak geometric budget of Theorem~\ref{thm:tri-weakP}, and that a weak entropy budget is available from either Theorem~\ref{thm:tri-weakE-zxs} or Theorem~\ref{thm:tri-weakE-cad}. Equip the corresponding triangular CAD node set with its positive quadrature weights, let $\mathcal E_K$ denote the induced local quadrature entropy, and let $\mathcal O_K^\star$ be any COS-compatible convex oscillation-suppressing set. Define
\[
\mathbf U_K^{\mathrm{EPO}}
:=
\mathcal S_{K,\bar{\mathbf U}_K^\star}
\bigl(\theta_K^{\mathrm{EPO}};\mathbf U_K^\star\bigr).
\]
Then
\begin{enumerate}[label=(\roman*),leftmargin=2em]
\item every nodal value of $\mathbf U_K^{\mathrm{EPO}}$ on the triangular CAD node set belongs to $G$;
\item the local strong entropy inequality
\[
\mathcal E_K(\mathbf U_K^{\mathrm{EPO}})\le B_K
\]
holds;
\item $\mathbf U_K^{\mathrm{EPO}}\in \mathcal O_K^\star$.
\end{enumerate}
If the local budgets satisfy the global compatibility assumption of Section~\ref{sec:main}, then the global strong entropy inequality holds after EPO limiting.
\end{corollary}

\begin{proof}
This is the triangular counterpart of Corollary~\ref{cor:rect-epo}. Proposition~\ref{prop:2d-lift} transfers Theorem~\ref{thm:localEPO} verbatim to the present cellwise triangular setting, while Theorem~\ref{thm:tri-weakP} supplies the weak geometric input and Theorem~\ref{thm:tri-weakE-zxs} or Theorem~\ref{thm:tri-weakE-cad} supplies the weak entropy budget.
\end{proof}

The active triangular budgets above give a local weak entropy stability on triangles without assuming that the candidate DG/FV stage is already entropy stable. 

\begin{theorem}[EPO on triangular meshes]\label{thm:tri-passive-EPO}
Assume that on every triangular cell $K$ a positive quadrature node set $S_K$ and weights $\{\varpi_{K,\nu}\}$ are chosen. Suppose that
\begin{enumerate}[label=(\roman*),leftmargin=2em]
\item the candidate cell averages satisfy the weak geometric budget $\bar{\mathbf U}_K^\star\in G$ furnished by Theorem~\ref{thm:tri-weakP};
\item the entire cell-average-anchored segment stays inside the entropy domain,
\[
\bar{\mathbf U}_K^\star+\theta(\mathbf U_{K,\nu}^\star-\bar{\mathbf U}_K^\star)\in D_\eta,
\qquad 0\le \theta\le 1,\quad \nu=1,\dots,N_K;
\]
\item the candidate nodal arrays are already entropy stable in the passive sense
\[
\mathcal E_K(\mathbf U_K^\star)\le B_K;
\]
\item the local passive budgets are globally compatible in the sense that
\begin{equation}\label{eq:tri-passive-global-candidate}
\sum_{K\in\mathcal T_h} |K|\,B_K
\le
\mathfrak B^n;
\end{equation}
\item each cell $K$ is equipped with a COS-compatible closed convex oscillation-suppressing set $\mathcal O_K^\star$.
\end{enumerate}
Let the geometric and oscillation radii be computed cellwise on $S_K$, and set
\begin{equation}\label{eq:tri-theta}
\theta_K^{\mathrm{EPO}}:=\min\{\theta_K^{\mathrm P},\theta_K^{\mathrm O}\}.
\end{equation}
Define the limited array by
\[
\mathbf U_K^{\mathrm{EPO}}
:=
\mathcal S_{K,\bar{\mathbf U}_K^\star}
\bigl(\theta_K^{\mathrm{EPO}};\mathbf U_K^\star\bigr).
\]
Then the EPO-limited nodal array on every cell satisfies
\begin{enumerate}[label=(\alph*),leftmargin=2em]
\item nodal membership in $G$;
\item membership in $\mathcal O_K^\star$;
\item the cellwise entropy monotonicity
\[
\mathcal E_K(\mathbf U_K^{\mathrm{EPO}})
\le
\mathcal E_K(\mathbf U_K^\star)
\le
B_K,
\]
and consequently the global entropy inequality
\begin{equation}\label{eq:tri-passive-global-limited}
\sum_{K\in\mathcal T_h} |K|\,\mathcal E_K(\mathbf U_K^{\mathrm{EPO}})
\le
\mathfrak B^n.
\end{equation}
\end{enumerate}
Equivalently, in this passive-entropy setting one may simply identify the entropy radius with the geometric one, that is, $\theta_K^{\mathrm{PE}}=\theta_K^{\mathrm P}$.
\end{theorem}

\begin{proof}
Because $\theta_K^{\mathrm{EPO}}\le \theta_K^{\mathrm P}$, the nodal membership statement follows from the two-dimensional lifting of Proposition~\ref{prop:thetaP} in Proposition~\ref{prop:2d-lift} together with Theorem~\ref{thm:tri-weakP}. Because $\theta_K^{\mathrm{EPO}}\le \theta_K^{\mathrm O}$, the two-dimensional lifting of Proposition~\ref{prop:thetaO-basic} yields $\mathbf U_K^{\mathrm{EPO}}\in \mathcal O_K^\star$.

For the entropy statement, assumption (ii) is exactly the domain hypothesis required by Proposition~\ref{prop:2d-passive-entropy}. Applying that proposition to the scaling parameter $\theta_K^{\mathrm{EPO}}\in[0,1]$ gives
\[
\mathcal E_K(\mathbf U_K^{\mathrm{EPO}})
\le
\mathcal E_K(\mathbf U_K^\star)
\le
B_K.
\]
Multiplying by $|K|$ and summing over all cells yields
\[
\sum_{K\in\mathcal T_h}|K|\,\mathcal E_K(\mathbf U_K^{\mathrm{EPO}})
\le
\sum_{K\in\mathcal T_h}|K|\,B_K
\le
\mathfrak B^n,
\]
which is \eqref{eq:tri-passive-global-limited}. The final statement is simply the observation that, along a passively entropy-stable ray, no additional entropy truncation beyond the geometric radius is needed.
\end{proof}

\begin{remark}
Theorem~\ref{thm:tri-passive-EPO} does \emph{not} require a separate active weak entropy budget for the cell average. What it does require is the entropy-domain hypothesis on the full cell-average-anchored segment and a globally compatible family of passive candidate bounds $B_K$. This is precisely the scenario in which the  entropy-stable DG discretizations fit most naturally into the EPO philosophy.
\end{remark}

\subsection{A COS(DG)-compatible oscillation module in two dimensions}

The one-dimensional oscillation formulas extend cellwise to rectangles, triangles, and more general meshes with only notational changes. The required inputs are candidate cell functions $\mathbf U_K^\star(\mathbf x)$ (DG polynomials or FV reconstructions), their cell averages $\bar{\mathbf U}_K^\star$, and immediate neighboring-cell data.

Let $K$ be a cell of area $|K|$ and let $\mathcal N_K$ denote the set of cells sharing a full edge with $K$. For an edge $E=\partial K\cap\partial K'$ with outward unit normal $\mathbf n_{E,K}$ from $K$ to $K'$, define the normal Jacobian
\[
\mathbf A_{\mathbf n_{E,K}}(\mathbf U):=\sum_{r=1}^2 n_{E,K}^{(r)}\,\frac{\partial \mathbf F_r}{\partial \mathbf U}(\mathbf U),
\]
and choose the compact local wave-speed bound
\begin{equation}\label{eq:2d-cos-beta}
\beta_E^{\mathrm{COS}}:=\max\Bigl\{\operatorname{spr}(\mathbf A_{\mathbf n_{E,K}}(\bar{\mathbf U}_K^\star)),\operatorname{spr}(\mathbf A_{\mathbf n_{E,K}}(\bar{\mathbf U}_{K'}^\star))\Bigr\}.
\end{equation}
With either entropy-induced distance \eqref{eq:cos-jensen-distance} or \eqref{eq:cos-frozen-distance}, define
\begin{equation}\label{eq:2d-cos-sigma}
\Theta_{E,K}:=\|\mathbf U_{K'}^\star-\bar{\mathbf U}_K^\star\|_{L_g^2(K')}^2,\qquad
\sigma_{E,K}^{\mathrm{COS}}:=\begin{cases}
C_k\dfrac{\|\mathbf U_{K'}^\star-\mathbf U_K^\star\|_{L_g^2(K')}^2}{\Theta_{E,K}},& \Theta_{E,K}\ge \varepsilon_K,\\[1.0em]
0,&\text{otherwise},
\end{cases}
\end{equation}
where $\varepsilon_K>0$ is a roundoff floor. For FV reconstructions the norms over $K'$ are computed by evaluating the candidate reconstruction polynomials in $K'$; no DG-specific basis transformation is required.

The canonical two-dimensional COS coefficient is then
\begin{equation}\label{eq:2d-cos-lambda}
\lambda_K^{\mathrm{COS}}:=\exp\!\left(-\frac{\Delta t}{|K|}\sum_{E\subset\partial K}|E|\,\beta_E^{\mathrm{COS}}\,\sigma_{E,K}^{\mathrm{COS}}\right),
\end{equation}
and the cellwise COS operator is
\begin{equation}\label{eq:2d-cos-op}
\mathcal F^{\mathrm{COS}}(\mathbf U_h^\star)\big|_{K}:=(1-\lambda_K^{\mathrm{COS}})\bar{\mathbf U}_K^\star+\lambda_K^{\mathrm{COS}}\mathbf U_K^\star(\mathbf x).
\end{equation}
Again, this formula is valid for both DG stage polynomials and FV reconstructions.

A local COS variant is obtained by activating only selected edges. Following \textsf{COS(DG)}, define
\[
\Lambda_{E,K}^\pm(\mathbf U):=\lambda_{\max/\min}\bigl(\mathbf A_{\mathbf n_{E,K}}(\mathbf U)\bigr),
\]
that is, the extremal eigenvalues of the normal Jacobian in the edge-normal direction $\mathbf n_{E,K}$. We mark an interface $E=K\cap K'$ as a potential shock if either
\begin{equation}\label{eq:2d-cos-local-criterion}
\Lambda_{E,K}^+(\bar{\mathbf U}_K^\star)-\Lambda_{E,K}^+(\bar{\mathbf U}_{K'}^\star)>\delta\Bigl(|\Lambda_{E,K}^+(\bar{\mathbf U}_K^\star)|+|\Lambda_{E,K}^+(\bar{\mathbf U}_{K'}^\star)|\Bigr),
\end{equation}
or
\begin{equation}\label{eq:2d-cos-local-criterion-minus}
\Lambda_{E,K}^-(\bar{\mathbf U}_K^\star)-\Lambda_{E,K}^-(\bar{\mathbf U}_{K'}^\star)>\delta\Bigl(|\Lambda_{E,K}^-(\bar{\mathbf U}_K^\star)|+|\Lambda_{E,K}^-(\bar{\mathbf U}_{K'}^\star)|\Bigr),
\end{equation}
The edge $E$ is declared COS-active when one of \eqref{eq:2d-cos-local-criterion}--\eqref{eq:2d-cos-local-criterion-minus} holds; otherwise its contribution in \eqref{eq:2d-cos-sigma} is set to zero. This gives the local COS version on general meshes.

The canonical two-dimensional oscillation-admissible segment is
\begin{equation}\label{eq:2d-cos-set}
\mathcal O_K^{\mathrm{COS},\star}:=\Bigl\{\mathcal S_{\bar{\mathbf U}_K^\star}(\theta;\mathbf U_K^\star):0\le \theta\le \lambda_K^{\mathrm{COS}}\Bigr\}.
\end{equation}

\begin{proposition}[Two-dimensional canonical COS module]\label{prop:2d-cos}
For every cell $K$, the set $\mathcal O_K^{\mathrm{COS},\star}$ defined by \eqref{eq:2d-cos-set} is nonempty, closed, and convex. It contains the constant state $\mathbf C_K(\bar{\mathbf U}_K^\star)$ and its oscillation radius along the cell-average-anchored ray is
\begin{equation}\label{eq:2d-cos-radius}
\theta_K^{\mathrm O}=\lambda_K^{\mathrm{COS}}.
\end{equation}
Moreover, the operator \eqref{eq:2d-cos-op} preserves the cell average and satisfies the same $L^2$-nonexpansiveness and entropy-inheritance properties cellwise as in Proposition~\ref{prop:cos-basic}.
\end{proposition}

\begin{proof}
The first statement follows exactly as in Proposition~\ref{prop:canonical-cos-segment}: \eqref{eq:2d-cos-set} is the image of the closed interval $[0,\lambda_K^{\mathrm{COS}}]$ under an affine map. The value $\theta=0$ gives the constant state and the largest admissible ray parameter is therefore \eqref{eq:2d-cos-radius}. Mean preservation, $L^2$ nonexpansiveness, and entropy inheritance are proved cellwise by the same calculations as in Proposition~\ref{prop:cos-basic}, with $I_j$ replaced by $K$ and $h_j$ replaced by $|K|$.
\end{proof}
\subsection{Unified two-dimensional EPO theorem}\label{sec:2d-unified}

We can now summarize the rectangular and triangular active-entropy pathways, together with the triangular passive-entropy alternative, in one theorem.

\begin{theorem}[Unified two-dimensional EPO]\label{thm:2d-epo}
Let $\mathcal T_h$ be a mesh of rectangles or triangles. On each cell $K\in\mathcal T_h$, choose a positive quadrature node set $S_K$ and positive weights $\{\varpi_{K,\nu}\}_{\nu=1}^{N_K}$ satisfying the exact mean relation \eqref{eq:2d-cellavg}. Let $\mathbf U_K^\star$ be a candidate nodal array with average $\bar{\mathbf U}_K^\star$.

Assume the following cellwise hypotheses.
\begin{enumerate}[label=(\roman*),leftmargin=2em]
\item \textbf{Weak geometric budget.} The candidate average satisfies $\bar{\mathbf U}_K^\star\in G$. This may be supplied, for example, by Theorem~\ref{thm:rect-weakP} on rectangles or by Theorem~\ref{thm:tri-weakP} on triangles.
\item \textbf{Entropy pathway.} Either
\begin{enumerate}[label=(\alph*),leftmargin=2em]
\item an active local budget $\eta(\bar{\mathbf U}_K^\star)\le B_K$ is available (for example Theorem~\ref{thm:rect-weakE}, Theorem~\ref{thm:tri-weakE-zxs}, or Theorem~\ref{thm:tri-weakE-cad}), or
\item the candidate nodal array is already entropy stable in the passive sense $\mathcal E_K(\mathbf U_K^\star)\le B_K$, and the full cell-average-anchored segment stays in $D_\eta$.
\end{enumerate}
\item \textbf{Oscillation pathway.} A COS-compatible closed convex set $\mathcal O_K^\star$ is prescribed.
\end{enumerate}
In the active-entropy case, let $\theta_K^{\mathrm{PE}}$ be the positivity-first entropy radius and define
\[
\theta_K^{\mathrm{EPO}}:=\min\{\theta_K^{\mathrm{PE}},\theta_K^{\mathrm O}\}.
\]
In the passive-entropy case, set equivalently $\theta_K^{\mathrm{PE}}:=\theta_K^{\mathrm P}$ and again define
\[
\theta_K^{\mathrm{EPO}}:=\min\{\theta_K^{\mathrm{PE}},\theta_K^{\mathrm O}\}=\min\{\theta_K^{\mathrm P},\theta_K^{\mathrm O}\}.
\]
Finally set
\[
\mathbf U_K^{\mathrm{EPO}}
:=
\mathcal S_{K,\bar{\mathbf U}_K^\star}
\bigl(\theta_K^{\mathrm{EPO}};\mathbf U_K^\star\bigr).
\]
Then the EPO-limited nodal array on every cell satisfies
\begin{enumerate}[label=(\alph*),leftmargin=2em]
\item nodal membership in $G$;
\item the relevant entropy inequality:
\[
\mathcal E_K(\mathbf U_K^{\mathrm{EPO}})\le B_K
\qquad \text{in the active case},
\]
\[
\mathcal E_K(\mathbf U_K^{\mathrm{EPO}})\le \mathcal E_K(\mathbf U_K^\star)\le B_K
\qquad \text{in the passive case};
\]
\item membership in $\mathcal O_K^\star$.
\end{enumerate}
If the local budgets are globally compatible in the active case, or if a global candidate/passive compatibility bound is available in the passive case, then the corresponding global strong entropy inequality holds after EPO limiting.
\end{theorem}

\begin{proof}
Proposition~\ref{prop:2d-lift} applies cellwise. In the active-entropy case, Theorem~\ref{thm:localEPO} is invoked with the cellwise weak geometric and weak entropy budgets; on triangles those budgets may be supplied, for example, by Theorem~\ref{thm:tri-weakE-zxs} or Theorem~\ref{thm:tri-weakE-cad}. In the passive-entropy case, the entropy statement follows from Proposition~\ref{prop:2d-passive-entropy}. In both cases, the oscillation statement is furnished by the lifted form of Proposition~\ref{prop:thetaO-basic} contained in Proposition~\ref{prop:2d-lift}. The global claim follows by summation over the mesh under the stated compatibility hypotheses.
\end{proof}

Theorem~\ref{thm:2d-epo} is the precise two-dimensional counterpart of the one-dimensional EPO theory developed earlier. It shows that the three ingredients of EPO continue to play exactly the same roles in two dimensions:
\begin{itemize}[leftmargin=2em]
\item special quadratures and convex decompositions provide the \emph{weak geometric budget};
\item one-dimensional slice budgets, triangular CAD-based weak budgets, or entropy-stable DG discretizations provide the \emph{entropy budget};
\item COS-compatible convex gauges provide the \emph{oscillation budget};
\item the final EPO limiter is the same local average-anchored radial scaling as before.
\end{itemize}

\begin{procedure}[2D EPO procedure]\label{proc:2d-epo}
On each cell $K$ of a rectangular or triangular mesh, proceed as follows.
\begin{enumerate}[leftmargin=2em]
\item Choose a positive quadrature/node set $S_K$ suitable for the local polynomial space and for the desired positivity or entropy analysis.
\item Evolve the underlying DG/FV scheme to obtain candidate nodal values and the candidate average $\bar{\mathbf U}_K^\star$.
\item Certify the weak geometric budget of the average:
\begin{itemize}[leftmargin=2em]
\item on rectangles, use the line-separable slice decomposition and Theorem~\ref{thm:rect-weakP};
\item on triangles, use a feasible special quadrature together with Theorem~\ref{thm:tri-weakP}.
\end{itemize}
\item Choose the entropy pathway:
\begin{itemize}[leftmargin=2em]
\item active-entropy path: construct a local budget $B_K$ (e.g. Theorem~\ref{thm:rect-weakE}, Theorem~\ref{thm:tri-weakE-zxs}, or Theorem~\ref{thm:tri-weakE-cad}) and compute the positivity-first entropy radius $\theta_K^{\mathrm{PE}}$;
\item passive-entropy path: start from an already entropy-stable candidate discretization and identify $\theta_K^{\mathrm{PE}}=\theta_K^{\mathrm P}$.
\end{itemize}
\item Build a COS-compatible oscillation set $\mathcal O_K^\star$ from compact cellwise data and compute $\theta_K^{\mathrm O}$.
\item Compute $\theta_K^{\mathrm{EPO}}=\min\{\theta_K^{\mathrm{PE}},\theta_K^{\mathrm O}\}$ and replace the candidate nodal array by the scaled array on the ray from the constant state generated by $\bar{\mathbf U}_K^\star$.
\item For rigorous SSPRK realizations, repeat the required $P/E$ steps stagewise and, if desired, postpone the canonical $O$-module to the final RK stage as in Proposition~\ref{prop:stagewise-PE-endO}. For high-order-in-time SSP multistep realizations, use the single end-of-step entropy budget of Theorem~\ref{thm:sspms-budget} and apply EPO once per time step.
\end{enumerate}
\end{procedure}

\section{Realizations, formulas, and implementation examples}\label{sec:realizations}

This section illustrates how the abstract EPO framework specializes in representative settings. The point is not to exhaust all possibilities but to show that the framework is computationally meaningful and connects directly with existing positivity, entropy, and oscillation technologies.

\subsection{Scalar conservation laws}

Consider the scalar conservation law
\[
u_t + f(u)_x = 0,
\]
and let the admissible interval be
\[
G=[u_{\min},u_{\max}].
\]
Suppose the candidate nodal values in cell $I_j$ are
\[
u_{j,1}^\star,\dots,u_{j,L}^\star,
\qquad
\bar u_j^\star = \sum_{\nu=1}^L \omega_\nu u_{j,\nu}^\star.
\]

\subsubsection{Geometric radius}

For the interval constraint, the positivity/invariant-set radius is explicit. Let
\[
M_j^\star := \max_{\nu} u_{j,\nu}^\star,
\qquad
m_j^\star := \min_{\nu} u_{j,\nu}^\star.
\]
Then
\begin{equation}\label{eq:scalar-thetaP}
\theta_j^{\mathrm P}
=
\min\left\{
1,\
\frac{u_{\max}-\bar u_j^\star}{M_j^\star-\bar u_j^\star}
\ \text{(if }M_j^\star>\bar u_j^\star\text{)},\
\frac{\bar u_j^\star-u_{\min}}{\bar u_j^\star-m_j^\star}
\ \text{(if }m_j^\star<\bar u_j^\star\text{)}
\right\}.
\end{equation}
This is exactly the classical Zhang--Shu scalar scaling factor.

\subsubsection{Entropy radius}

Let $\eta$ be any convex entropy. Then
\[
\Psi_j(\theta)
=
\sum_{\nu=1}^L
\omega_\nu
\eta\!\left(\bar u_j^\star+\theta(u_{j,\nu}^\star-\bar u_j^\star)\right)
\]
is convex and nondecreasing on $[0,1]$. The entropy radius is the largest $\theta$ such that $\Psi_j(\theta)\le B_j^n$.

For example, consider the quadratic entropy
\[
\eta(u)=\frac{u^2}{2}.
\]
Then
\begin{align}
\Psi_j(\theta)
&=
\frac12 \sum_{\nu=1}^L \omega_\nu
\left(\bar u_j^\star+\theta(u_{j,\nu}^\star-\bar u_j^\star)\right)^2 \notag\\
&=
\frac12(\bar u_j^\star)^2
+
\frac{\theta^2}{2}\sum_{\nu=1}^L \omega_\nu (u_{j,\nu}^\star-\bar u_j^\star)^2,
\label{eq:quadratic-Psi}
\end{align}
because the linear term vanishes by the mean constraint. Therefore
\begin{equation}\label{eq:quadratic-thetaE}
\theta_j^{\mathrm E}
=
\begin{cases}
1,
& \Psi_j(1)\le B_j^n,\\[0.6em]
\sqrt{
\dfrac{2(B_j^n-\frac12(\bar u_j^\star)^2)}
{\sum_{\nu=1}^L \omega_\nu (u_{j,\nu}^\star-\bar u_j^\star)^2}
},
& \Psi_j(1)>B_j^n.
\end{cases}
\end{equation}
This formula shows directly that the entropy limiter penalizes the variance of the nodal deviations from the cell average.

\subsubsection{Oscillation radius}

A canonical COS radius is obtained by specializing \eqref{eq:cos-lambda-1d} with the quadratic sensor entropy $g(u)=u^2/2$. Let $u_j^\star(x)$ denote either the DG stage polynomial or an FV reconstruction on $I_j$, with average $\bar u_j^\star$. Then
\begin{equation}\label{eq:scalar-cos-radius}
\theta_j^{\mathrm O}:=\lambda_j^{\mathrm{COS}}=\exp\!\left(-\frac{\beta_j^{\mathrm{COS}}\Delta t}{h_j}\,\sigma_j^{\mathrm{COS}}\right),
\end{equation}
where
\[
\beta_j^{\mathrm{COS}}=\max_{j-1\le \ell\le j+1}|f'(\bar u_\ell^\star)|,\qquad \sigma_j^{\mathrm{COS}}=\sigma_j^-+\sigma_j^+,
\]
and
\[
\sigma_j^{\pm}=\begin{cases}
C_k\dfrac{\|u_{j\pm1}^\star-u_j^\star\|_{L^2(I_{j\pm1})}^2}{\|u_{j\pm1}^\star-\bar u_j^\star\|_{L^2(I_{j\pm1})}^2},
&
\|u_{j\pm1}^\star-\bar u_j^\star\|_{L^2(I_{j\pm1})}^2\ge \varepsilon_j
\ \text{and}\ 
\chi_{j\pm\frac12}=1,
\\[1.0em]
0,&\text{otherwise}.
\end{cases}
\]
Here $\chi_{j+\frac12}$ is the local COS marker \eqref{eq:cos-local-criterion-1d} specialized to the scalar case; in the global COS variant one simply sets $\chi_{j\pm\frac12}=1$. The oscillation-suppressing update is therefore
\[
\mathcal F^{\mathrm{COS}}(u_h^\star)\big|_{I_j}=(1-\theta_j^{\mathrm O})\bar u_j^\star+\theta_j^{\mathrm O}u_j^\star(x),
\]
which is equally valid whether $u_j^\star$ is a DG polynomial or an FV reconstruction.

\subsection{Compressible Euler equations}\label{sec:Euler}
Consider the one-dimensional compressible Euler equations
\[
\mathbf U=
\begin{pmatrix}
\rho\\ m\\ E
\end{pmatrix},
\qquad
\mathbf F(\mathbf U)=
\begin{pmatrix}
m\\
\dfrac{m^2}{\rho}+p\\
\dfrac{m}{\rho}(E+p)
\end{pmatrix},
\]
with ideal-gas pressure
\[
p(\mathbf U) = (\gamma-1)\left(E-\frac{m^2}{2\rho}\right).
\]

\subsubsection{Numerical admissible set}

The physical admissible set is
\[
G_{\rm phys} = \{ \rho>0,\ p(\mathbf U)>0\},
\]
which is open. For numerical limiting we work with
\begin{equation}\label{eq:Euler-Geps}
G_\varepsilon
=
\left\{
(\rho,m,E)^T:\ \rho\ge \varepsilon,\ p(\mathbf U)\ge \varepsilon
\right\}
\end{equation}
for some fixed $\varepsilon>0$.  $G_\varepsilon$ is closed and convex. Therefore it fits Assumption~\ref{ass:G}.

\subsubsection{Geometric radius and the Zhang--Shu positivity limiter}

The geometric radius $\theta_j^{\mathrm P}$ is the largest scalar such that
\[
\bar{\mathbf U}_j^\star
+
\theta(\mathbf U_{j,\nu}^\star-\bar{\mathbf U}_j^\star)\in G_\varepsilon
\qquad \forall \nu.
\]
In practice this may be computed in several equivalent ways; see, for example, the original Zhang--Shu positivity limiter \cite{ZhangShuPP2010}.

\subsubsection{Positivity-first entropy radius}

For the Euler equations, the mathematically primary construction is the positivity-first entropy radius. First compute the geometric radius $\theta_j^{\mathrm P}$ so that every state on the truncated ray remains in $G_\varepsilon\subset D_\eta$. Then define
\[
\mathbf U_j^{\mathrm P}:=\mathcal S_{\bar{\mathbf U}_j^\star}(\theta_j^{\mathrm P};\mathbf U_j^\star),
\]
and compute the largest secondary parameter $\vartheta_j^{\mathrm E|P}\in[0,1]$ such that
\[
\sum_{\nu=1}^L \omega_\nu
\eta\!\left(
\bar{\mathbf U}_j^\star+\vartheta\bigl(\mathbf U_{j,\nu}^{\mathrm P}-\bar{\mathbf U}_j^\star\bigr)
\right)
\le
B_j^n.
\]
The resulting Euler entropy radius is
\[
\theta_j^{\mathrm{PE}}=\theta_j^{\mathrm P}\,\vartheta_j^{\mathrm E|P}.
\]
In general this scalar problem is solved numerically by bisection or Newton iteration, but it remains only one-dimensional and is therefore inexpensive. If the full candidate ray is known \emph{a priori} to remain in $D_\eta$, one may instead compute the auxiliary full-ray radius $\theta_j^{\mathrm E}$ directly and then recover $\theta_j^{\mathrm{PE}}=\min\{\theta_j^{\mathrm P},\theta_j^{\mathrm E}\}$. For the mathematical entropy associated with ideal-gas dynamics, one may take, up to an affine scaling and additive constant,
\[
\eta(\mathbf U)
=
-\rho s,
\qquad
s=\log p - \gamma \log \rho,\qquad {\mathcal Q} = - m s 
\]
and $\eta(\mathbf U)$ is convex on $G_{\rm phys}$ and hence on $G_\varepsilon$.

\subsubsection{Oscillation radius}

For the Euler equations, the canonical oscillation radius is again the COS coefficient, now built from entropy-induced distances. Let $\mathbf U_j^\star(x)$ be either the DG stage polynomial or an FV reconstruction on $I_j$, with average $\bar{\mathbf U}_j^\star$. Choose a convex oscillation entropy $g$, for example the physical mathematical entropy or the frozen-Hessian surrogate of \textsf{COS(DG)}. Then
\begin{equation}\label{eq:euler-cos-radius}
\theta_j^{\mathrm O}:=\lambda_j^{\mathrm{COS}}=\exp\!\left(-\frac{\beta_j^{\mathrm{COS}}\Delta t}{h_j}\,\sigma_j^{\mathrm{COS}}\right),
\end{equation}
with $\beta_j^{\mathrm{COS}}$ given by \eqref{eq:cos-lambda-1d} and
\[
\sigma_j^{\mathrm{COS}}=\sigma_j^-+\sigma_j^+,
\qquad
\sigma_j^{\pm}=\begin{cases}
C_k\dfrac{\|\mathbf U_{j\pm1}^\star-\mathbf U_j^\star\|_{L_g^2(I_{j\pm1})}^2}{\|\mathbf U_{j\pm1}^\star-\bar{\mathbf U}_j^\star\|_{L_g^2(I_{j\pm1})}^2},& \|\mathbf U_{j\pm1}^\star-\bar{\mathbf U}_j^\star\|_{L_g^2(I_{j\pm1})}^2\ge \varepsilon_j,\\[1.0em]
0,&\text{otherwise}.
\end{cases}
\]
 The oscillation-suppressing update is
\[
\mathcal F^{\mathrm{COS}}(\mathbf U_h^\star)\big|_{I_j}=(1-\theta_j^{\mathrm O})\bar{\mathbf U}_j^\star+\theta_j^{\mathrm O}\mathbf U_j^\star(x).
\]
This is the precise Euler realization of the COS component used by EPO. If a Zhang--Shu positivity limiter is also enforced before or together with COS, one recovers the combined bound-preserving and oscillation-suppressing scaling familiar from \textsf{COS(DG)}.

\subsubsection{Final EPO limiter}

The final Euler EPO limiter is
\[
\theta_j^{\mathrm{EPO}}
=
\min\bigl\{\theta_j^{\mathrm{PE}},\theta_j^{\mathrm O}\bigr\},
\qquad
\mathbf U_{j,\nu}^{n+1}
=
\bar{\mathbf U}_j^\star
+
\theta_j^{\mathrm{EPO}}
(\mathbf U_{j,\nu}^\star-\bar{\mathbf U}_j^\star).
\]
If the auxiliary direct radius $\theta_j^{\mathrm E}$ is available on the whole candidate ray, then this reduces to
\[
\theta_j^{\mathrm{EPO}}=
\min\bigl\{\theta_j^{\mathrm P},\theta_j^{\mathrm E},\theta_j^{\mathrm O}\bigr\}.
\]
Hence the same nodal update enforces the admissibility constraint, the entropy budget, and the oscillation constraint. For a finite family of entropy pairs, one computes one entropy radius per pair and takes the minimum before combining with $\theta_j^{\mathrm O}$.

\subsection{A practical EPO algorithm}

For reference we summarize the cellwise implementation in algorithmic form.

\begin{procedure}[Cellwise EPO limiter]\label{alg:EPO}
Given candidate nodal values $\mathbf U_{j,\nu}^\star$ and average $\bar{\mathbf U}_j^\star$ in cell $I_j$:
\begin{enumerate}[label=\arabic*.,leftmargin=2em]
\item Compute or import the weak entropy budget. For one entropy pair this is a single number $B_j^n$ satisfying \eqref{eq:weakE}. For several prescribed entropy pairs, compute one budget $B_j^{n,(r)}$ for each pair. For SSP Runge--Kutta methods, these should be the stagewise budgets from Theorem~\ref{thm:ssp-budget}; for SSP multistep methods, use the single end-of-step budget from Theorem~\ref{thm:sspms-budget}.
\item Compute the geometric radius $\theta_j^{\mathrm P}$ from \eqref{eq:thetaP}.
\item Compute the positivity-first entropy radius. For one entropy pair, compute $\theta_j^{\mathrm{PE}}$ by Proposition~\ref{prop:thetaE-after-P}. For several prescribed entropy pairs, compute $\theta_j^{\mathrm{PE},(r)}$ for each pair and set
\[
\theta_j^{\mathrm{PE,all}}:=\min_{1\le r\le M}\theta_j^{\mathrm{PE},(r)}.
\]
If the full candidate ray is known \emph{a priori} to remain in $D_\eta$, one may instead compute the auxiliary direct radius $\theta_j^{\mathrm E}$ from \eqref{eq:thetaE}; for several entropies this gives one direct radius per entropy pair.
\item Compute the oscillation radius $\theta_j^{\mathrm O}$. In the canonical COS realization one sets $\theta_j^{\mathrm O}=\lambda_j^{\mathrm{COS}}$ from \eqref{eq:cos-lambda-1d} in one dimension or \eqref{eq:2d-cos-lambda} in two dimensions. More generally, if a larger COS-compatible closed convex set $\mathcal O_j^\star$ is used, compute $\theta_j^{\mathrm O}$ from \eqref{eq:thetaO}.
\item Set
\[
\theta_j^{\mathrm{EPO}}=
\begin{cases}
\min\{\theta_j^{\mathrm{PE}},\theta_j^{\mathrm O}\},&\text{for one entropy pair},\\[0.4em]
\min\{\theta_j^{\mathrm{PE,all}},\theta_j^{\mathrm O}\},&\text{for several entropy pairs}.
\end{cases}
\]
\item Update each node by
\[
\mathbf U_{j,\nu}^{n+1}
=
\bar{\mathbf U}_j^\star+\theta_j^{\mathrm{EPO}}(\mathbf U_{j,\nu}^\star-\bar{\mathbf U}_j^\star).
\]
\end{enumerate}
\end{procedure}

\begin{remark}
Procedure~\ref{alg:EPO} is local. It requires only the candidate DG/FV cellwise state, its weighted nodal values, the updated cell average, and the neighboring-cell data entering the weak budgets and the oscillation radius. The conservative update of the cell average is not modified. In the canonical COS realization the oscillation radius is explicit, while the positivity-first entropy radius is a scalar one-dimensional search. For several prescribed entropy pairs, the only change is that one computes one entropy radius per pair and then takes the minimum.
\end{remark}

\section{Numerical experiments}\label{sec:num}

We now turn to numerical experiments for the compressible Euler equations; see Section~\ref{sec:Euler}. The examples are chosen to assess the practical performance of the proposed framework, with particular emphasis on the entropy-stable correction and the full EPO procedure.

Unless stated otherwise, all computations use a third-order discontinuous Galerkin discretization with $P^2$ polynomials and the third-order SSP multistep time discretization \eqref{eq:sspms-third-budget}.

\begin{example}[Accuracy test]\label{ex:accuracy}
	We first verify that the entropy module preserves the designed order of accuracy for smooth solutions. Since the positivity-preserving limiter of Zhang--Shu \cite{ZhangShuPP2010} and the COS procedure of Cao--Huang--Li--Wu \cite{CaoHuangLiWu2026} are already known not to destroy high-order accuracy, this test isolates the effect of the entropy correction.
	
	We consider the one-dimensional compressible Euler equations on the periodic domain $[0,1]$. The initial data are
	\[
	\rho(x,0)=1.0+0.2\sin(2\pi x),\qquad
	u(x,0)=0.7,\qquad
	p(x,0)=0.1.
	\]
	Because the velocity and pressure are constant, the exact solution is a smooth translation of the initial profile:
	\[
	\rho(x,t)=1.0+0.2\sin(2\pi(x-0.7t)),\qquad
	u(x,t)=0.7,\qquad
	p(x,t)=0.1.
	\]
	The computation is advanced to the final time $T=1.0$ with $\mathrm{CFL}=0.02$.
	
	Table~\ref{tab:accuracy_results} reports the errors and observed convergence rates. The results show uniform third-order convergence in the $L_1$, $L_2$, and $L_\infty$ norms. This confirms that the entropy module, like the positivity-preserving and COS modules, does not degrade the formal accuracy for smooth solutions.
\end{example}

\begin{table}[htbp]
	\centering
	\caption{Errors and observed convergence rates for the smooth Euler accuracy test at $T=1.0$.}
	\label{tab:accuracy_results}
	\vspace{2mm}
	\begin{tabular}{rcccccc}
		\toprule
		$N$ & $L_1$ error & $L_1$ order & $L_2$ error & $L_2$ order & $L_\infty$ error & $L_\infty$ order \\
		\midrule
		16   & 6.25e-05 & --   & 8.27e-05 & --   & 2.94e-04 & --   \\
		32   & 7.60e-06 & 3.03 & 1.01e-05 & 3.03 & 3.32e-05 & 3.14 \\
		64   & 9.43e-07 & 3.01 & 1.41e-06 & 2.84 & 7.45e-06 & 2.15 \\
		128  & 8.22e-08 & 3.52 & 1.13e-07 & 3.64 & 4.57e-07 & 4.02 \\
		256  & 1.03e-08 & 3.00 & 1.41e-08 & 3.00 & 5.72e-08 & 2.99 \\
		512  & 1.28e-09 & 3.00 & 1.76e-09 & 3.00 & 7.16e-09 & 2.99 \\
		1024 & 1.61e-10 & 2.99 & 2.20e-10 & 2.99 & 8.99e-10 & 2.99 \\
		\bottomrule
	\end{tabular}
\end{table}

We next consider a standard set of one-dimensional benchmark problems for the compressible Euler equations. These tests probe complementary aspects of the proposed EPO framework, including shock resolution, contact capturing, robustness under strong wave interaction, near-vacuum behavior, and the evolution of the global discrete entropy.

All computations below are carried out using the $P^2$ DG scheme with the EPO framework and the third-order SSP multistep time discretization \eqref{eq:sspms-third-budget}. For each example, we report the density, velocity, and pressure at the final time, together with the time evolution of the global discrete entropy
\begin{equation}\label{eq:global-discrete-entropy-uniform}
	E_\eta
	:=
	\Delta x\sum_j \mathcal E_j(\mathbf U_j).
\end{equation}

\begin{example}[Sod shock tube]\label{ex:sod}
	The Sod problem is the standard Riemann test containing a left rarefaction, a contact discontinuity, and a right-moving shock. The computational domain is $[-5,5]$, the mesh consists of $256$ uniform cells, and the numerical solution is reported at the final time $t=1.3$. The initial data are
	\[
	(\rho,u,p)(x,0)=
	\begin{cases}
		(1,\,0,\,1), & x<0,\\
		(0.125,\,0,\,0.1), & x>0.
	\end{cases}
	\]
	Figure~\ref{fig:sod-sol} shows the density, velocity, and pressure at the final time, together with the time evolution of the global discrete entropy $E_\eta$. The numerical solution reproduces the expected wave pattern of the Sod problem without visible spurious oscillations near the discontinuities. The monotone decay of $E_\eta$ is consistent with the dissipative character of the fully discrete scheme in the presence of shocks.
\end{example}

\begin{figure}[htbp]
	\centering
	\includegraphics[width=0.99\textwidth]{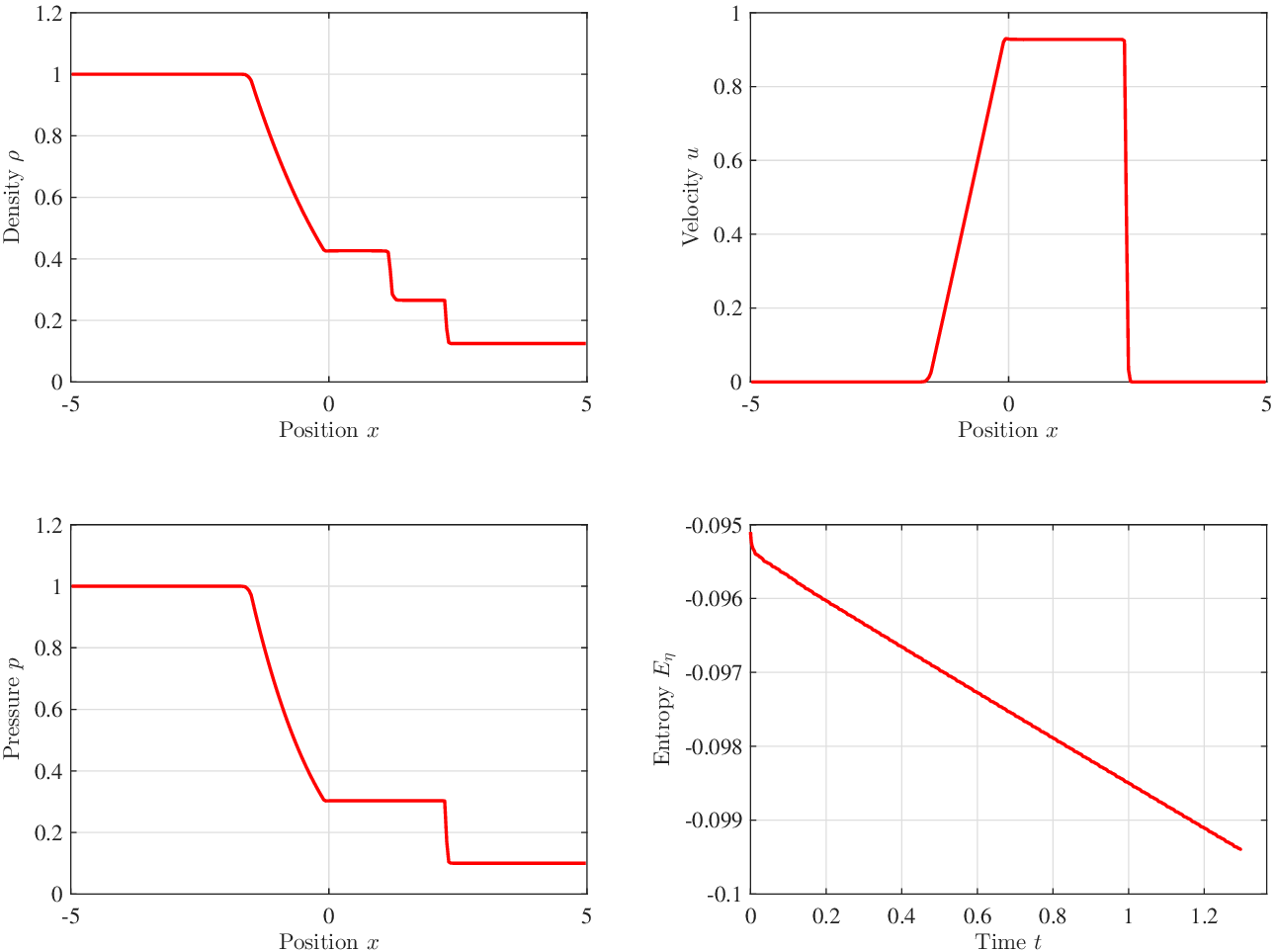}
	\caption{Sod shock tube problem computed using the $P^2$ DG scheme with the EPO framework on $256$ uniform cells. Shown are the density, velocity, and pressure at the final time $t=1.3$, together with the time evolution of the global discrete entropy $E_\eta$ on $0\le t\le 1.3$.}
	\label{fig:sod-sol}
\end{figure}

\begin{example}[Lax problem]\label{ex:lax}
	The Lax problem is a sharper Riemann test with stronger compression and a thinner post-shock structure. The computational domain is $[-5,5]$, the mesh consists of $256$ uniform cells, and the numerical solution is reported at the final time $t=1.3$. The initial data are
	\[
	(\rho,u,p)(x,0)=
	\begin{cases}
		(0.445,\,0.698,\,3.528), & x<0,\\
		(0.5,\,0,\,0.571), & x>0.
	\end{cases}
	\]
	Figure~\ref{fig:lax-sol} displays the density, velocity, and pressure at the final time, together with the time evolution of the global discrete entropy $E_\eta$. The figure shows that the method resolves the compressed region and the contact structure without obvious ringing. The stronger decay of $E_\eta$ than in the Sod test reflects the larger amount of entropy dissipation associated with the sharper shock structure.
\end{example}

\begin{figure}[htbp]
	\centering
	\includegraphics[width=0.99\textwidth]{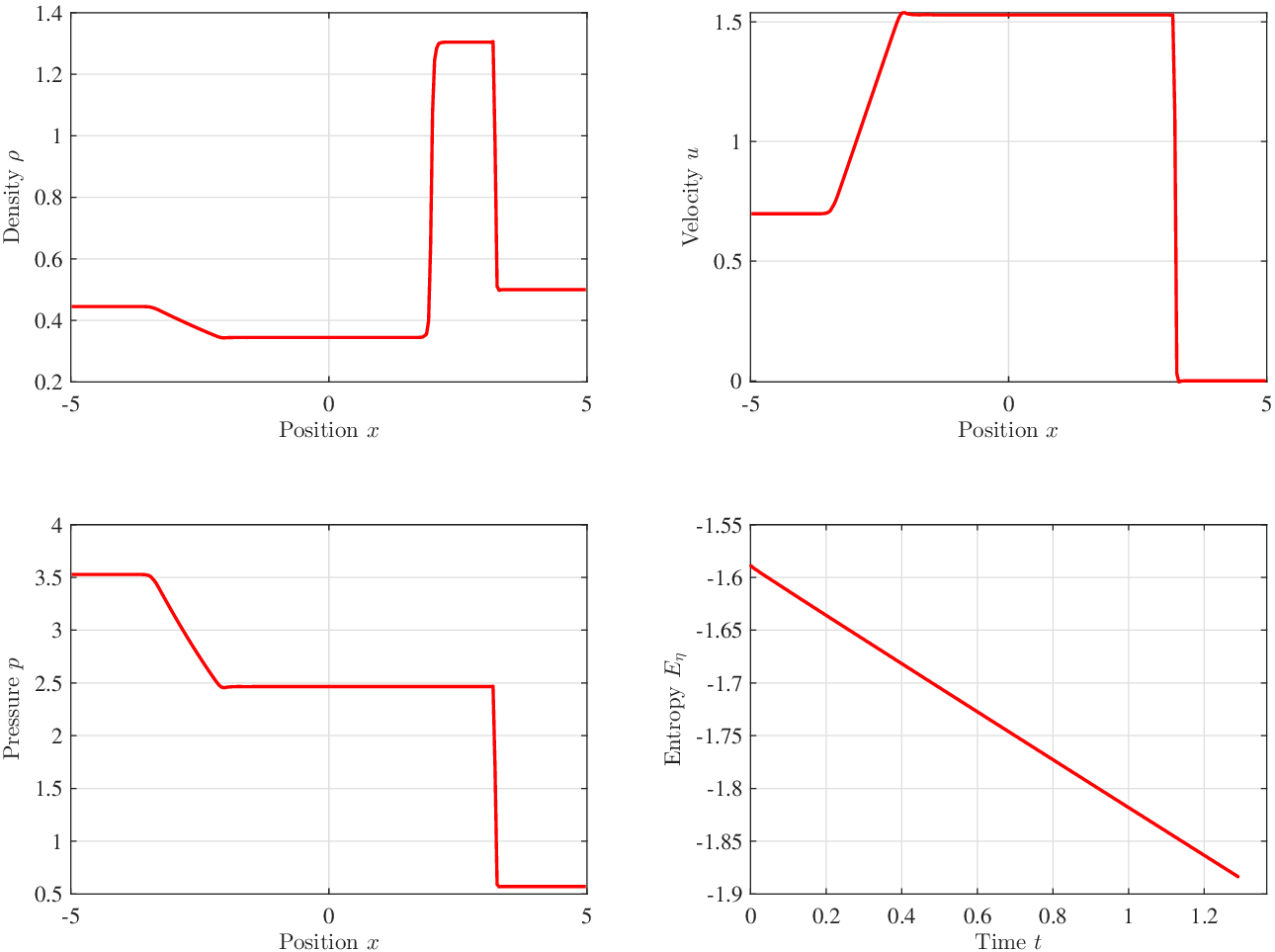}
	\caption{Lax problem computed using the $P^2$ DG scheme with the EPO framework on $256$ uniform cells. Shown are the density, velocity, and pressure at the final time $t=1.3$, together with the time evolution of the global discrete entropy $E_\eta$.}
	\label{fig:lax-sol}
\end{figure}

\begin{example}[Two-blast-wave interaction]\label{ex:blast}
	The two-blast-wave interaction problem is a severe test involving repeated shock reflections and strong internal wave interaction. The computational domain is $[0,1]$, the mesh consists of $960$ uniform cells, and the numerical solution is reported at the final time $t=0.038$. The initial data are
	\[
	(\rho,u,p)(x,0)=
	\begin{cases}
		(1,\,0,\,1000), & 0\le x<0.1,\\
		(1,\,0,\,0.01), & 0.1<x<0.9,\\
		(1,\,0,\,100), & 0.9<x\le 1.
	\end{cases}
	\]
	Figure~\ref{fig:blast-sol} shows the density, velocity, and pressure at the final time, together with the time evolution of the global discrete entropy $E_\eta$. The computed profiles remain stable even in regions with extremely steep gradients and narrow internal structures. The entropy decreases rapidly during the main interaction phase and then levels off, indicating that most of the dissipation is concentrated in the strongest compression events.
\end{example}

\begin{figure}[htbp]
	\centering
	\includegraphics[width=0.99\textwidth]{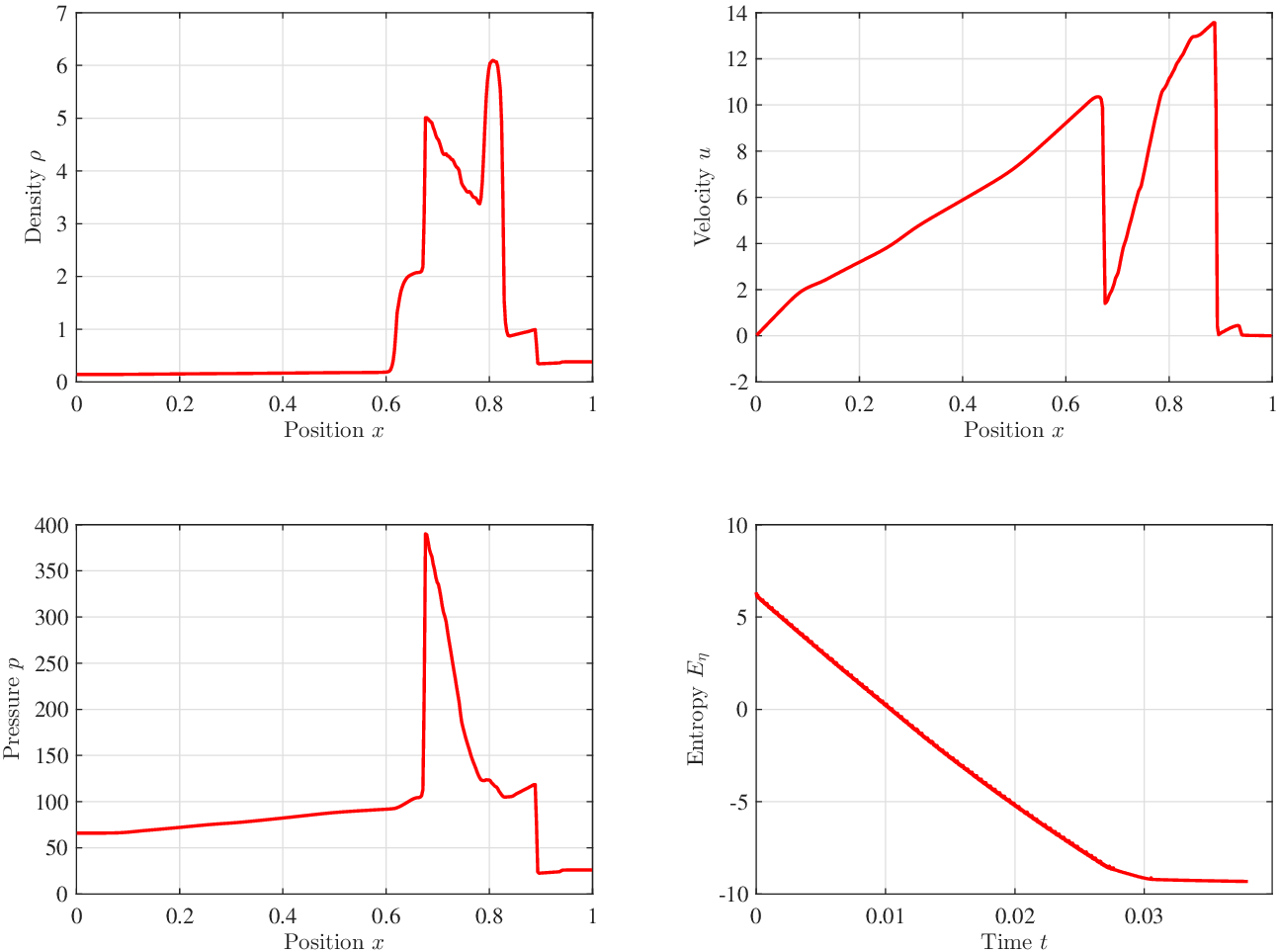}
	\caption{Two-blast-wave interaction problem computed using the $P^2$ DG scheme with the EPO framework on $960$ uniform cells. Shown are the density, velocity, and pressure at the final time $t=0.038$, together with the time evolution of the global discrete entropy $E_\eta$.}
	\label{fig:blast-sol}
\end{figure}

\begin{example}[Shu--Osher problem]\label{ex:shuosher}
	The Shu--Osher problem couples a shock with a smooth oscillatory density field and is a standard test of the ability to retain fine-scale post-shock structure while controlling nonphysical oscillations. The computational domain is $[-5,5]$, the mesh consists of $200$ uniform cells, and the numerical solution is reported at the final time $t=1.8$. The initial data are
	\[
	(\rho,u,p)(x,0)=
	\begin{cases}
		(3.857143,\,2.629369,\,10.333333), & x<-4,\\
		(1+0.2\sin(5x),\,0,\,1), & x>-4.
	\end{cases}
	\]
	Figure~\ref{fig:shuosher-sol} shows the density, velocity, and pressure at the final time, together with the time evolution of the global discrete entropy $E_\eta$. The figure shows that the method captures the main shock sharply while retaining the oscillatory post-shock structure behind it. The smooth decay of $E_\eta$ indicates that the entropy correction stabilizes the shock without wiping out the physically relevant fine-scale waves.
\end{example}

\begin{figure}[htbp]
	\centering
	\includegraphics[width=0.99\textwidth]{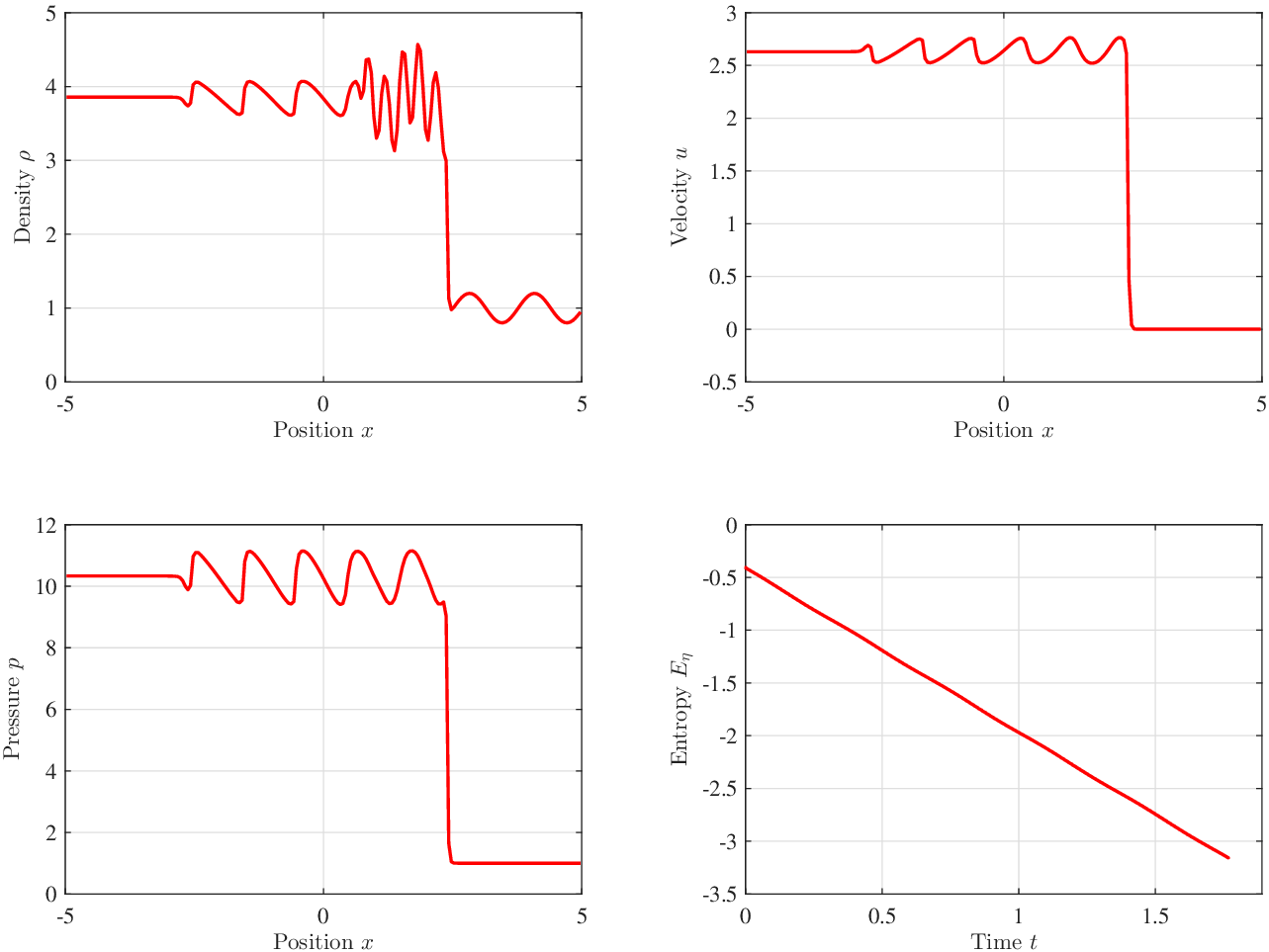}
	\caption{Shu--Osher problem computed using the $P^2$ DG scheme with the EPO framework on $200$ uniform cells. Shown are the density, velocity, and pressure at the final time $t=1.8$, together with the time evolution of the global discrete entropy $E_\eta$.}
	\label{fig:shuosher-sol}
\end{figure}

\begin{example}[Sedov blast wave]\label{ex:sedov}
	The Sedov problem is a stringent test for robustness in the presence of a highly concentrated energy deposition. The computational domain is $[-2,2]$, the mesh consists of $201$ uniform cells, and the numerical solution is reported at the final time $t=0.001$.
	
	The initial data are prescribed in conservative variables. The density and velocity are
	\[
	\rho(x,0)=1,\qquad u(x,0)=0,
	\]
	and the total energy is given by
	\[
	E(x,0)=
	\begin{cases}
		E_0/\Delta x, & |x|\le \tfrac12\Delta x,\\
		10^{-12}, & |x|>\tfrac12\Delta x,
	\end{cases}
	\qquad E_0=3.2\times 10^6,
	\]
	where $\Delta x$ is the mesh size. Thus the initial blast energy is deposited in the central cell, while the rest of the domain carries a very small background energy.
	
	Figure~\ref{fig:sedov-sol} shows the density, velocity, and pressure at the final time, together with the time evolution of the global discrete entropy $E_\eta$. The solution exhibits the expected blast-wave structure with two strong outgoing fronts and a low-density region around the center. The fact that the density and pressure remain nonnegative throughout this severe test provides direct evidence of robustness in a near-vacuum regime.
\end{example}

\begin{figure}[htbp]
	\centering
	\includegraphics[width=0.99\textwidth]{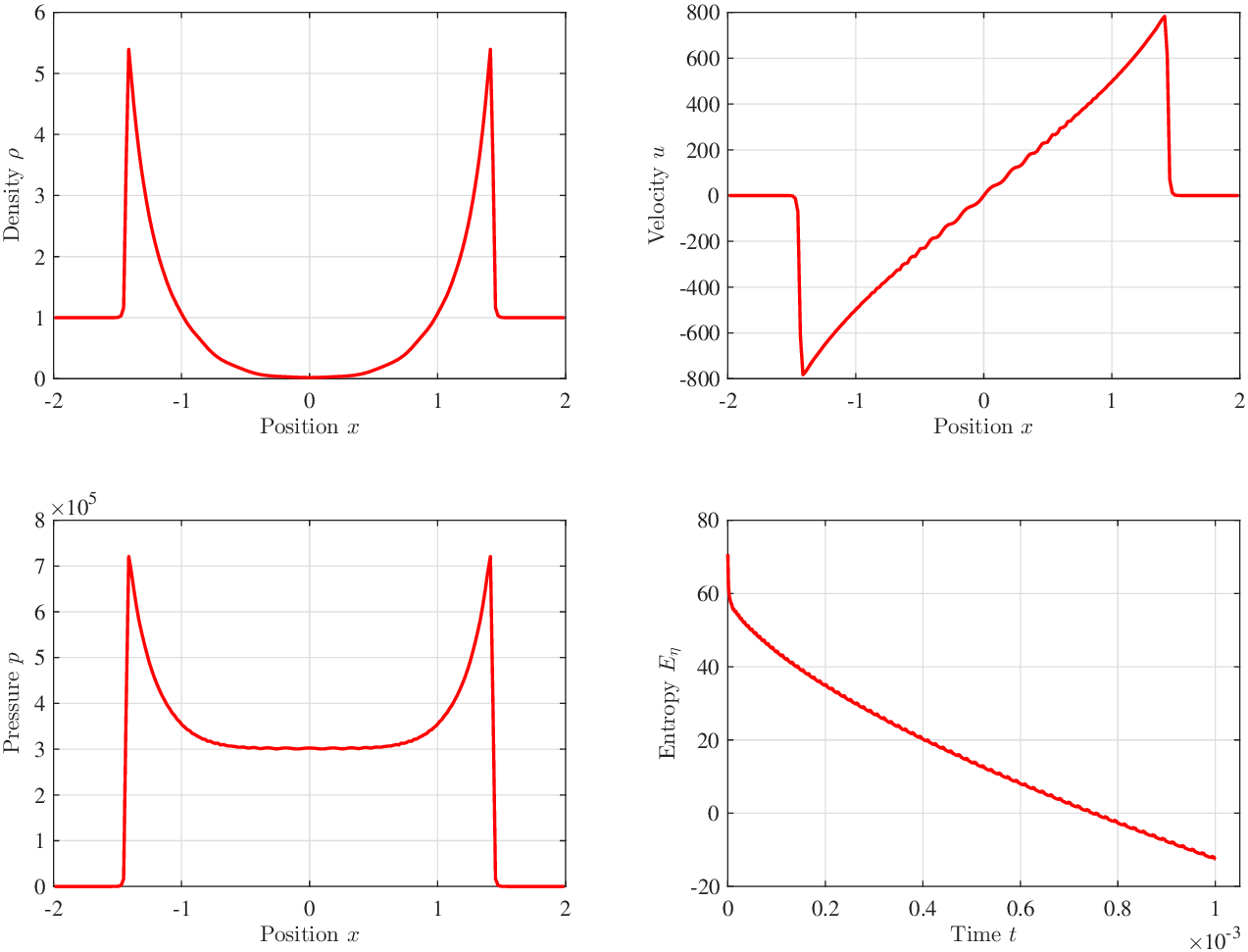}
	\caption{Sedov blast-wave problem computed using the $P^2$ DG scheme with the EPO framework on $201$ uniform cells. Shown are the density, velocity, and pressure at the final time $t=0.001$, together with the time evolution of the global discrete entropy $E_\eta$ on $0\le t\le 0.001$.}
	\label{fig:sedov-sol}
\end{figure}

\begin{example}[Leblanc shock tube]\label{ex:leblanc}
	The Leblanc problem is a severe Riemann test with extremely large jumps in density and pressure. It is particularly useful for assessing robustness in very low-density regions. The computational domain is $[-10,10]$, the mesh consists of $6400$ uniform cells, and the numerical solution is reported at the final time $t=10^{-4}$. The initial data are
	\[
	(\rho,u,p)(x,0)=
	\begin{cases}
		(2,\,0,\,10^{9}), & x<0,\\
		(10^{-3},\,0,\,1), & x>0.
	\end{cases}
	\]
	Figure~\ref{fig:leblanc-sol} shows the density, velocity, and pressure at the final time, together with the time evolution of the global discrete entropy $E_\eta$. The computed profiles remain physically admissible across several orders of magnitude in density and pressure. The monotone decay of $E_\eta$, together with the absence of visible undershoots in the low-density region, indicates that the scheme remains robust in this very demanding regime.
\end{example}

\begin{figure}[htbp]
	\centering
	\includegraphics[width=0.99\textwidth]{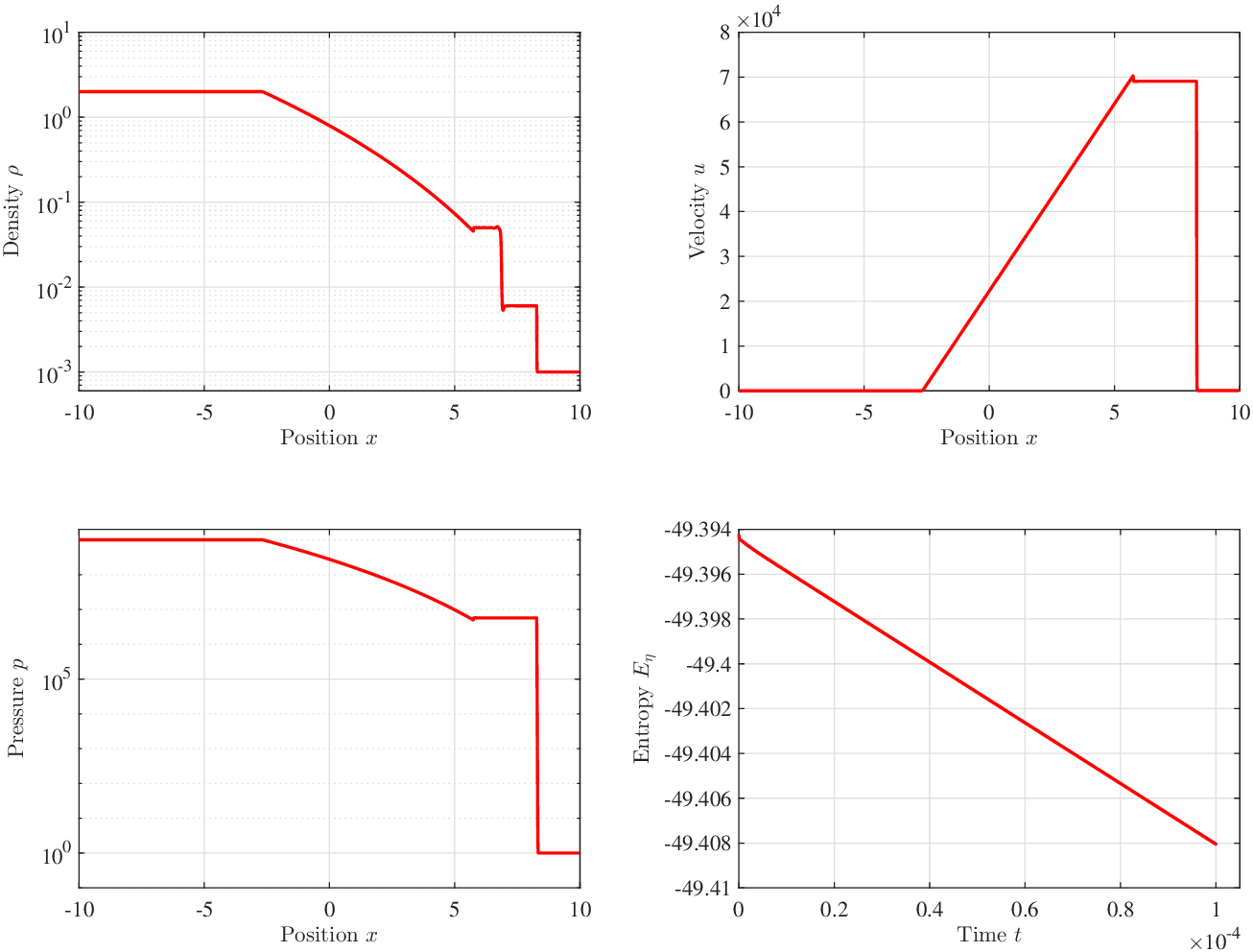}
	\caption{Leblanc shock tube computed using the $P^2$ DG scheme with the EPO framework on $6400$ uniform cells. Shown are the density, velocity, and pressure at the final time $t=10^{-4}$, together with the time evolution of the global discrete entropy $E_\eta$ on $0\le t\le 10^{-4}$.}
	\label{fig:leblanc-sol}
\end{figure}

\section{Conclusions and outlook}\label{sec:conclusion}

This paper has proposed \EPO{}, a unified framework for fully discrete entropy stability, positivity/bound preservation, and spurious oscillation control,  built on one scaling ray anchored at the updated cell average. The admissible-state constraint gives the geometric radius, each prescribed entropy pair gives an entropy radius, and the oscillation constraint gives an oscillation radius. The final limiter uses the minimum of these quantities. The main analytical input is the weak entropy stability of the updated cell average, obtained here from a two-point Lax--Friedrichs/Riemann-average entropy inequality and then lifted to a strong quadrature-based entropy inequality by radial scaling.

This viewpoint is distinguished from the classical entropy-stable literature based on Tadmor flux differencing, SBP/SAT operators, split forms, and suitable quadrature constructions \cite{Tadmor1987,Tadmor2003,FisherCarpenter2013,Gassner2013,CarpenterFisherNielsenFrankel2014,SvardNordstrom2014,DelReyFernandezHickenZingg2014,ChenShu2017,ChanFernandezCarpenter2019}. Those approaches build entropy stability into the spatial discretization, typically for one chosen entropy pair. \EPO{} does not replace them and can be combined with them, but it serves a different purpose: given a candidate FV/DG update and suitable weak budgets, it produces a fully discrete nodal entropy-stable correction. In particular, any prescribed finite family of entropy pairs can be enforced simultaneously by computing one entropy radius for each pair and taking the minimum.

We also discussed two time-discrete realizations. Stagewise SSP Runge--Kutta inherits forward-Euler weak budgets stage by stage. SSP multistep provides an end-of-step budget and therefore supports a single end-of-step entropy correction without reducing the designed temporal order. The two-dimensional rectangular and triangular analyses show that the same local geometry persists once the corresponding weak budgets have been established.

Further work includes extensions to balance laws, nonconservative systems, and multiphysics coupling systems, such as MHD equations.

\appendix

\section{Explicit formulas}\label{app:explicit}

\subsection{Affine-constraint formula for the geometric radius}

Suppose the admissible set is given by finitely many affine or concave inequalities
\[
G = \{ U\in \mathbb R^m : \ell_r(U)\ge 0,\ r=1,\dots,R\},
\]
where each $\ell_r$ is affine or concave. Assume $\bar{\mathbf U}_j^\star\in G$. Then the geometric radius \eqref{eq:thetaP} can be defined as 
\begin{equation}\label{eq:thetaP-affine}
\theta_j^{\mathrm P}
=
\min\Biggl\{
1,\
\min_{\substack{1\le \nu\le L,\ 1\le r\le R\\ \ell_r(\mathbf U_{j,\nu}^\star)<0}}
\frac{\ell_r(\bar{\mathbf U}_j^\star)}
{\ell_r(\bar{\mathbf U}_j^\star)-\ell_r(\mathbf U_{j,\nu}^\star)}
\Biggr\},
\end{equation}
where the inner minimum is understood to be omitted if there are no violating pairs $(\nu,r)$.

\subsection{Quadratic entropy formula}

For the quadratic entropy
\[
\eta(U)=\frac12 |U|^2
\]
in a Euclidean state space, the entropy profile is
\[
\Psi_j(\theta)
=
\frac12 |\bar U_j^\star|^2
+
\frac{\theta^2}{2}
\sum_{\nu=1}^L \omega_\nu |U_{j,\nu}^\star-\bar U_j^\star|^2.
\]
Hence the entropy radius is
\begin{equation}\label{eq:thetaE-quadratic}
\theta_j^{\mathrm E}
=
\begin{cases}
1,
& \Psi_j(1)\le B_j^n,\\[0.6em]
\sqrt{
\dfrac{2\bigl(B_j^n-\frac12|\bar U_j^\star|^2\bigr)}
{\sum_{\nu=1}^L \omega_\nu |U_{j,\nu}^\star-\bar U_j^\star|^2}
},
& \Psi_j(1)>B_j^n.
\end{cases}
\end{equation}
The proof is immediate from the identity
\[
\sum_{\nu=1}^L \omega_\nu (U_{j,\nu}^\star-\bar U_j^\star)=0.
\]

\bibliographystyle{siamplain}
\bibliography{epo_refs}

\end{document}